\documentclass[11pt,a4paper]{article}
\usepackage{amsmath}
\usepackage{amsthm}
\usepackage{amssymb}
\usepackage{amsfonts}
\usepackage{mathrsfs}
\usepackage{mathtools}
\usepackage{parskip}
\usepackage{dsfont}
\usepackage{geometry}
\usepackage{todonotes}
\usepackage{bm}
\makeatletter
\def\thm@space@setup{%
  \thm@preskip=\parskip \thm@postskip=0pt
}
\makeatother

\newtheorem{theorem}{THEOREM}%[section] links numbering to section
\newtheorem*{theorem*}{THEOREM}
\newtheorem{lemma}[theorem]{Lemma}  %[theorem] between brackets links numbering to theorems and for props etc.
\newtheorem*{lemma*}{Lemma}
\newtheorem{prop}[theorem]{Proposition}
\newtheorem{corollary}[theorem]{Corollary}
\newtheorem{definition}[theorem]{Definition}

\theoremstyle{definition}

\newtheorem*{note}{Note}

\newtheorem*{remark}{Remark}

\newcommand{\E}[1]{\ensuremath{\mathbb{E} \left[#1 \right]}}
\newcommand{\Prob}[1]{\ensuremath{\mathbb{P} \left(#1 \right)}}

\newcommand{\Bin}[2]{\ensuremath{\mathrm{Bin}\left(#1,#2 \right)}}

\newcommand{\R}{\ensuremath{\mathbb{R}}}
\newcommand{\Z}{\ensuremath{\mathbb{Z}}}
\newcommand{\N}{\ensuremath{\mathbb{N}}}
\newcommand{\Ch}[2]{\ensuremath{\begin{pmatrix} #1 \\ #2 \end{pmatrix}}}
\newcommand{\ch}[2]{\ensuremath{\left( \begin{smallmatrix} #1 \\ #2
\end{smallmatrix} \right)}}
\newcommand{\F}{\ensuremath{\mathcal{F}}}
\newcommand{\C}{\ensuremath{\mathcal{C}}}
\newcommand{\fl}[1]{\ensuremath{\lfloor #1 \rfloor}}
\newcommand{\ce}[1]{\ensuremath{\lceil #1 \rceil}}
\newcommand{\equidist}{\ensuremath{\stackrel{d}{=}}}
\newcommand{\ER}{Erd\H{o}s-R\'enyi }

\newcommand{\cV}{{\mathcal{V}}}
\newcommand{\cA}{{\mathcal{A}}}
\newcommand{\cU}{{\mathcal{U}}}
\newcommand{\cZ}{{\mathcal{Z}}}
\newcommand{\cC}{{\mathcal{C}}}

\newcommand{\kH}{\check{H}}

\newcommand{\tH}{\bar{H}}

\newcommand{\tod}{\,\stackrel{d}\rightarrow\,}

\begin{document}
\title{\textbf{Critical random forests}}
%\author{James B.\ Martin and Dominic Yeo}
\author{James Martin \thanks{Department of Statistics, University of Oxford. \texttt{martin@stats.ox.ac.uk}} \and Dominic Yeo \thanks{
Faculty of Industrial Engineering and Management, Technion. 
\texttt{yeo@technion.ac.il}}}
\date{2 July 2018}
\maketitle
%\tableofcontents
\begin{abstract}
Let $F(N,m)$ denote a random forest on a set of $N$ 
vertices, chosen uniformly from all forests with $m$
edges. Let $F(N,p)$ denote the forest obtained by 
conditioning the \ER graph $G(N,p)$ to be acyclic. 
We describe scaling limits for the
largest components of $F(N,p)$ and $F(N,m)$, 
in the \textit{critical window}
$p=N^{-1}+O(N^{-4/3})$ or $m=N/2+O(N^{2/3})$. 
Aldous \cite{Aldous97} described a scaling limit
for the largest components of $G(N,p)$ 
within the critical window in terms of the
excursion lengths of a reflected Brownian
motion with time-dependent drift. 
Our scaling limit for critical random forests
is of a similar nature, but now based on a 
reflected diffusion whose drift depends
on space as well as on time.
%We review {\bf [James to write/rewrite]} Aldous's results \cite{Aldous97} about the distribution of the sequence of component-sizes in $G(N,p)$, when $p=p(N)$ lies in the \emph{critical window} $p(N)=\frac{1+\lambda N^{-1/3}}{N}$. Aldous describes the scaling limit for the largest such components in terms of the excursions of a reflected Brownian motion with time-dependent drift. We prove a similar result for the random forest obtained by conditioning $G(N,p)$ to have no cycles, for the same range of $p$. We describe a scaling limit for the largest components of such a \emph{critical random forest}, but now using a reflected diffusion whose drift is space-dependent as well as time-dependent.
\end{abstract}

\renewcommand{\thefootnote}{}

\section{Introduction}
Let $G(N,p)$ be the \ER random graph 
with vertex set $[N]$, in which each 
of the $\binom{N}{2}$ possible edges appears independently
with probability $p$. 
%{\bf [Check ok]} Strengthening the work of Bollob\'as \cite{BB84}, 
%{\L}uczak \cite{Luczak90} showed that within the \emph{critical window} $p(N)=\frac{1}{N}+O(N^{-4/3})$, the largest
%components of $G(N,p)$ are of order $N^{2/3}$.
%In a seminal paper, Aldous \cite{Aldous97}
%gave a scaling limit for the 
%joint distribution of the sizes of these 
%largest components of $G(N,p)$ within this critical window. 
%Under suitable rescaling, their 
%sizes converge to the lengths of the excursions 
%of a reflected Brownian motion with time-dependent 
%drift. 
%The central result of \cite{Aldous97} may
%be written as follows: 
In a seminal paper, Aldous \cite{Aldous97}
gave a scaling limit for the 
joint distribution of the sizes of the
largest components of $G(N,p)$ within the
critical window $p(N)=\frac{1}{N}+O(N^{-4/3})$.
In this regime, the largest components
are of order $N^{2/3}$ 
(as was shown first by Bollob\'as in \cite{BB84} up to a logarithmic correction, and then by {\L}uczak in \cite{Luczak90});
after rescaling by $N^{2/3}$, their 
sizes converge to the lengths of the excursions 
of a reflected Brownian motion with time-dependent 
drift. 
The central result of \cite{Aldous97} may
be written as follows: 
\begin{prop}\label{prop:AldousGNp}
\newcommand{\CG}{C^{G(N,p)}}
\newcommand{\cCBlambda}{\mathcal{C}^{B^\lambda}}
Let $\lambda\in\R$,
and consider the sequence of random graphs $G(N,p)$
with $p=p(N)=N^{-1}+\lambda N^{-4/3}$.
Let $\CG_1, \CG_2, \dots$
be the sequence of component sizes of $G(N,p)$ 
written in non-increasing order, 
augmented by zeros. 

Let $B^\lambda(t), t\geq 0$ be a
(time-inhomogeneous) reflected
Brownian motion, with drift $\lambda-t$ at time $t$,
and $B^\lambda(0)=0$. 
Let $\cCBlambda$ be the lengths of the excursions
of $B^\lambda$, written in non-increasing order. Then
\begin{equation}\label{Aldousconvergence}
N^{-2/3}\left(\CG_1, \CG_2, \dots\right)
\tod
\cCBlambda
\text{ as }
N\to\infty
\end{equation}
with respect to the $\ell^2$ topology.
\end{prop}
A similar result may be written for the 
random graph $G(N,m)$ (that is, a graph chosen uniformly 
from all those with vertex set $[N]$ and with
$m$ edges), in the regime $m=N/2+O(N^{2/3})$. 

Aldous's result has been extended in multiple 
ways. The same Brownian scaling limit has
been shown to arise in more general settings
including configuration models and inhomogeneous
random graphs, provided the tail of the degree
distribution is sufficiently light
\cite{Turova, Riordan, Joseph, BvdHvL1, BhaBudWan, DvdHvLS-critical}.
In some such cases, finer scaling limits describing
the metric structure of the large components, 
as well as their size, have been obtained,
in terms of objects related to the Brownian 
continuum random tree \cite{A-BBG10, A-BBG09, BBSW, BSW}.
When the third moment of the vertex degrees
is infinite, different scaling limits arise
\cite{Joseph, BvdHvL2, DvdHvLS-heavy} which can be 
described in terms of excursion lengths
of the ``thinned L{\'e}vy processes" introduced
in \cite{AldousLimic98}. 
Finally, dynamic models have been studied
in which the developing component structure
of a random graph process (or more generally
a multiplicative coalescent) is described
(as a process) by excursions of a Brownian motion
or thinned L{\'e}vy process whose drift changes with time
\cite{armendariz_phd, BroutinMarckert, Limic-eternal, MartinRath}.

In this paper we develop in a new direction,
to consider the sizes of trees in random forests. 
Write $F(N,m)$ for a graph chosen uniformly at random 
from all forests on $[N]$ with $m$ edges
(equivalently, those forests consisting of $N-m$ trees). 
Write also $F(N,p)$ for the graph $G(N,p)$ conditioned
to be acyclic. Our main results give a  scaling limit for the joint distribution of the sizes of the largest trees in $F(N,m)$ or in $F(N,p)$ in the critical regime (this critical regime coincides with 
that for $G(N,m)$ and $G(N,p)$ above). 
The limit is given by the collection of excursion lengths of a diffusion, as at (\ref{Aldousconvergence}) 
but now the limiting diffusion is
inhomogeneous in space as well as in time. 

Just as for Aldous's proof of 
Proposition \ref{prop:AldousGNp},
these convergence results are proved by 
analysing the \textit{graph exploration process},
which encodes enough of the graph structure
to recover the sequence of component sizes. 
We discuss the exploration process, and its 
scaling limit, in Section \ref{subsec:intro-exploration};
before that, we introduce the notation needed
to state our main results. 

\subsection{Definition of the diffusion}
\label{subsec:intro-diffusion}
{\L}uczak and Pittel \cite{LuczakPittel92}
studied the model $F(N,m)$, and identified
subcritical, critical and supercritical regimes.
Within the critical window, 
specifically for $m=N/2+(\lambda+o(1))N^{2/3}$,
their Theorem 4.1 establishes convergence
in distribution for $N^{-2/3}C_k$, where
$C_k$ is the size of the $k$th largest tree
in $F(N,m)$. 

{\L}uczak and Pittel's analysis relies
on enumeration results of Britikov \cite{Britikov},
which we also use extensively. 
Britikov gives asymptotics for $f(N,m)$, the number of forests on $[N]$ with exactly $m$ edges, via Bell polynomials. The relevant regime of Britikov's result is summarised by Lemma 2.1(ii) of 
\cite{LuczakPittel92}:

\begin{lemma}
%\cite[Theorem 2.1.ii]{LuczakPittel92}
For any constant $c>0$, as $N\rightarrow\infty$,
\begin{equation}\label{eq:Britikovresult}f(N,m)=(1+o(1))\frac{\sqrt{2\pi}N^{N-1/6}}{2^{N-m}(N-m)!} g\left(\frac{2m-N}{N^{2/3}}\right),\end{equation}
uniformly for $m\in\left[N/2-cN^{2/3}, N/2+cN^{2/3}\right]$,
%$ satisfying $|2m-N|^3/N^2\le c$, 
where
\begin{equation}\label{eq:defng}g(x)=\frac{1}{\pi}\int_0^\infty \exp(-\tfrac{4}{3}t^{3/2})\cos(xt+\tfrac43 t^{3/2})dt,\end{equation}
is the density of a stable distribution with parameter 3/2.
\end{lemma}

As we shall see in Lemma \ref{acyclicboundlemma}, it follows that the asymptotic probability that $G(N,p)$ is acyclic in this regime is $\Theta(N^{-1/6})$. 
%{\bf (Could move this remark? Seems OK here though)}

\begin{definition}
For $b>0$ and $\lambda\in\mathbb{R}$, let
\begin{equation}\label{eq:defnalpha}
\alpha(b,\lambda):=\frac{\int_0^\infty a^{-1/2}g(\lambda-a)\exp\left(\tfrac{(\lambda-a)^3}{6}\right)\exp(-\frac{b^2}{2a})\mathrm{d}a}{\int_0^\infty a^{-3/2}g(\lambda-a)\exp\left(\tfrac{(\lambda-a)^3}{6}\right)\exp(-\frac{b^2}{2a})\mathrm{d}a}.
\end{equation}
%which will turn out to be the additive correction one must %make to the drift of the rescaled reflected exploration %process of $G(N,\frac{1+\lambda N^{-1/3}}{N})$ when the %height is $b$, to account for acyclicity.
\end{definition}

\begin{lemma}\label{gregularitylemma}
The function $g$ defined in \eqref{eq:defng} is positive, bounded, and uniformly continuous, and satisfies $g(x)\rightarrow 0$ as $x\rightarrow \pm\infty$. 
The integrals in the numerator and denominator of (\ref{eq:defnalpha}) both
converge for all $b$ and $\lambda$. The function $\alpha$ is continuous
and increasing in its first argument, and satisfies $\alpha(b,\lambda)\rightarrow 0$ as 
$b\downarrow 0$, uniformly on $\lambda$ in compact intervals.
\end{lemma}

We will prove Lemma \ref{gregularitylemma} in Section \ref{alpharegularity}, after the main probabilistic arguments.

We now define a reflected diffusion, $Z^\lambda$, whose
drift at time $s$ and height $b$ is 
$\lambda-s-\alpha(b, \lambda-s)$. 
The excursion lengths of $Z^\lambda$ will describe 
the scaling limits of the largest trees in our critical
random forests. Comparing with the definition of $B^\lambda$ in Proposition \ref{prop:AldousGNp}, we see
that the function $\alpha$ provides the correction
to the drift which is required to account for the acyclicity condition.
\begin{prop}\label{Zlambdaexists}
Consider a standard Brownian motion $W(\cdot)$ with natural filtration $\F^W$. For each $\lambda\in\R$, there exists a unique pair of non-negative $\mathcal{F}^W$-adapted processes $Z^\lambda, K^\lambda$ satisfying:
\begin{equation}\label{eq:ZSDE}
\begin{cases}
Z^\lambda(0)=0,\\ 
\mathrm{d}Z^\lambda(t)=
\left[\lambda - t -\alpha\left(Z^\lambda(t),\lambda-t\right)\right]\mathrm{d}t
+\mathrm{d}W(t) + K^\lambda(t),\end{cases}\end{equation}
where $K^\lambda$ is the local-time process of $Z^\lambda$ at zero. 
That is, $K^\lambda(\cdot)$ is continuous and increasing, 
with $K^\lambda(0)=0$, and $\int_0^\infty Z^\lambda(t)\mathrm{d}K^\lambda(t)=0$. 
\end{prop}

Since the drift term in \eqref{eq:ZSDE} is dominated by $\lambda-t$, $Z^\lambda$ almost surely has a well-defined largest excursion, and second-largest excursion, and so on.
\begin{definition}
Let $\cC^\lambda:=(C^\lambda_1, C^\lambda_2, \dots)$
be the sequence of lengths of the excursions
of $Z^\lambda$, written in non-increasing order.
\end{definition}
\subsection{Main results}
\label{subsec:intro-main}
We can now state the main results of the paper. 

\begin{theorem}\label{FNmtheorem}
Fix $\lambda\in\R$ and suppose that $m(N)$ is a sequence
of integers such that $m=N/2+(\lambda+o(1))N^{2/3}$ as $N\to\infty$. 
Consider the sequence of random forests $F(N,m)$. 
Let $C_1^{F(N,m)}\ge C_2^{F(N,m)}\ge\ldots$ be the sequence of tree sizes in $F(N,m)$, in non-increasing order, augmented with zeros. Then 
\begin{equation}\label{eq:l2convFNm}
N^{-2/3}\left(
C_1^{F(N,m)},C_2^{F(N,m)},\ldots
\right) \tod 
\cC^\lambda
\end{equation}
as $N\to\infty$, with respect to the $\ell^2$ topology.
\end{theorem}

\begin{theorem}\label{FNptheorem}
Fix $\lambda\in\R$ and suppose that $p(N)$ is a sequence such that $p=N^{-1}+(\lambda+o(1))N^{-4/3}$ as $N\to\infty$.
Consider the sequence of random forests $F(N,p)$. Let $C_1^{F(N,p)}\ge C_2^{F(N,p)}\ge\ldots$ be the sequence of tree sizes in $F(N,p)$, in non-increasing order, augmented with zeros. Then 
\begin{equation}\label{eq:l2convFNp}N^{-2/3}\left(C_1^{F(N,p)},C_2^{F(N,p)},\ldots\right)\tod \mathcal{C}^\lambda\end{equation}
as $N\to\infty$, with respect to the $\ell^2$ topology.
\end{theorem}

We will work mostly in the context of the model $F(N,p)$. In Section \ref{mtopsection} we will deduce separately that Theorem \ref{FNptheorem} for $F(N,p)$ implies Theorem \ref{FNmtheorem} for $F(N,m)$. 

\subsection{Exploration processes}\label{subsec:intro-exploration}
As was the case for Aldous's 
Proposition \ref{prop:AldousGNp}, our 
proof of Theorem \ref{FNptheorem} 
is based on an analysis of
the \textit{exploration process} of the graph. 
To be specific, we will work with a breadth-first ordering (although the argument
would work equally well with various other orderings). 

Let $G$ be any graph with vertex set $[N]$. We define
the breadth-first ordering $v_1, v_2, \dots, v_N$ in the following way. 
For a vertex $v$, let $\Gamma(v)$ be the set of neighbours of $v$
in $G$. For each $n=0,1,\dots, N$, denote
\[
\mathcal{Z}_n:=\Gamma(v_1)\cup\ldots\cup\Gamma(v_n) \backslash\{v_1,\ldots,v_n\}.
\]
Now recursively, for each $n=0,1,\dots, N-1$:
\begin{itemize}
\item if $|\mathcal{Z}_n|=0$, then let $v_{n+1}$ be the smallest element of $[N]\setminus \{v_1,\dots, v_n\}$, and let $v_{n+2},\ldots,v_{n+1+|\Gamma(v_{n+1})|}$ be the elements of $\Gamma(v_{n+1})$, in increasing order;
\item if $|\mathcal{Z}_n|=r>0$, then let
$\{v_{n+r+1}, \dots, v_{n+r+a}\}$ 
be the elements of 
$\mathcal{Z}_{n+1}\setminus \mathcal{Z}_n$ 
in increasing order,
where $a=|\mathcal{Z}_{n+1}\setminus \mathcal{Z}_n|$.
\end{itemize}
Note that $\mathcal{Z}_0=\emptyset$, $v_1=1$, and $\mathcal{Z}_1=\Gamma(1)$.

We can interpret the construction as follows. 
We imagine exploring the graph one vertex at a time,
revealing neighbours as we proceed.
$\mathcal{Z}_n$ is the \textit{stack} after step $n$, consisting 
of the vertices that we have seen but not yet processed. 
At the next step $n$, if there are any vertices on the stack,
we process the one which was added earliest (namely $v_{n+1}$), removing
it from the stack and adding to the stack all its neighbours that have 
not previously been seen. If instead the stack is empty, 
we select a new vertex (the smallest-labelled vertex that
has not previously been seen) and process that vertex in the same way.

We define the \textit{reflected exploration process} $(Z_n)_{n\geq 0}$
by $Z_n=|\mathcal{Z}_n|$. If we define $0=n_0, n_1, \dots, n_C=N$ to be the
times $n$ such that $Z_n=0$, written in increasing order, then the 
components of $G$ are 
$\{v_{n_i+1}, v_{n_i+2}, \dots, v_{n_{i+1}}\}$ for $i=0,\dots, C-1$. 
In this way we can interpret the component sizes of $G$ as the 
lengths of excursions from 0 of the reflected exploration process. 

The strategy of proof of Theorem \ref{FNptheorem} is now to show that
the exploration processes $Z^{N,p}$ of the forests $F(N,p)$, 
in the regime of Theorem \ref{FNptheorem},
converge as $N\to\infty$, when suitably rescaled, 
to the diffusion $Z^\lambda$, in such a way 
that the rescaled lengths of the longest excursions of $Z^{N,p}$ converge
to the lengths of the longest excursions of $Z^\lambda$. Our main 
convergence result is the following:
\begin{theorem}\label{explconvtheorem}With $\lambda,(p(N))$ as in Theorem \ref{FNptheorem}, let $(Z^{N,p}_n)_{n\ge 0}$ be the reflected exploration process of $F(N,p)$. For $s\ge 0$, set
\begin{equation}\label{eq:defntildeZ}\tilde Z^{N,p}_s:= N^{-1/3}Z^{N,p}_{\fl{N^{2/3}s}}.\end{equation}
Then we have $\tilde Z^{N,p} \stackrel{d}\rightarrow Z^\lambda$, uniformly on compact time-intervals.
\end{theorem}

Let us try to give some intuition for this result
and for the role played by the function $\alpha$, by comparing the behaviour of the exploration processes
for $G(N,p)$ and for $F(N,p)$, for $p=N^{-1}+\lambda N^{-4/3}$.

We first recall the heuristic for the 
scaling in the $G(N,p)$ case. The exploration process $Z_n, n\geq 0$ is a Markov chain. Condition on $Z_n$, the size of the stack after $n$ steps, being equal to $r>0$ and consider the distribution of
$Z_{n+1}-Z_n+1$, which is 1 more than the next increment of the process. This quantity is the number of neighbours
that the vertex $v_n$ has in 
$[N]\setminus\{v_1,\dots,v_{n+r}\}$,
and it has $\Bin{N-n-r}{p}$ distribution, with 
mean $(N-n-r)p$. 

If we write $n=tN^{2/3}$ and $r=bN^{1/3}$, 
the mean of that increment is then
\[
(N-tN^{2/3}-bN^{1/3})(N^{-1/3}+\lambda N^{-4/3})-1
\]
which (for $t$ and $b$ of constant order)
is approximately $(\lambda-t)N^{-1/3}$. 
Meanwhile the variance is $1+O(N^{-1/3})$. 
If we rescale time by a factor $N^{2/3}$ and space by a factor $N^{1/3}$,
we converge to a process with drift $\lambda-t$ and 
variance $1$ per unit time, namely the diffusion 
$B^\lambda$ of Proposition \ref{prop:AldousGNp}.

Now consider instead the exploration process for $F(N,p)$,
which is $G(N,p)$ conditioned to be acyclic. 
We will see in Section \ref{exploration-transitions} 
that the exploration process is still a 
Markov chain, but the acyclicity condition changes
the distribution of the increments. Suppose again $Z_n=r$, 
so that the current stack is $\{v_{n+1},\dots,v_{n+r}\}$. 
These stack vertices are already known to be 
in the same component of the graph. For the graph to remain acyclic, we now require that the subgraph induced by the vertices $\{v_{n+1}, \dots, v_N\}$ is a forest, and furthermore no two of the stack vertices are in the same tree of this forest. As a result of this conditioning, 
the quantity $Z_{n+1}-Z_n+1$ no longer has $\Bin{N-n-r}{p}$ distribution as in the $G(N,p)$ case just discussed, but is stochastically dominated by $\Bin{N-n-r}{p}$; furthermore, 
the downward bias produced is stronger when the stack size $r$ is higher. 

What we establish is that, in the same regime as above, this bias produces a change in the expected increment which is again of order $N^{-1/3}$ and depends on the size of the stack. After rescaling as above, the drift obtained is now 
instead $\lambda-t-\alpha(b, \lambda-t)$, leading to 
the diffusion $Z^\lambda$ defined at (\ref{eq:ZSDE}).

The particular convergence properties that we need in order 
to obtain Theorem \ref{explconvtheorem} are collected
in the following result:

%In the following two results, we will fix $\lambda\in\R$ 
%and work with $F(N,p)$ where $p(N)$ satisfies the 
%conditions of Theorem \ref{FNptheorem}. As we will see 
%shortly, each $(Z^{N,p})$, the exploration process of 
%$F(N,p)$, is Markov. A general framework for showing 
%convergence of Markov processes to the solutions of SDEs 
%was introduced by Stroock and Varadhan (see, for example, 
%\cite{SVbook}), using convergence of generators. Roughly 
%speaking, it is necessary to check that the expected 
%increments of $\tilde Z^{N,p}$ converge to the drift of $Z^
%\lambda$, and some other regularity results to ensure 
%tightness and non-sticky reflection.

\begin{prop}\label{limdriftprop}
Fix $\lambda\in\R$, and let $p=p(N)$ satisfy the conditions of Theorem \ref{FNptheorem}. 
For each $N\in\N$, the reflected exploration process $Z^{N,p}$ of $F(N,p)$ is a Markov chain. 
Further, fix any $T, K<\infty$ and 
$\delta>0$.
Then, uniformly on $n\in [0,TN^{2/3}]$ and $r\in[1,K N^{1/3}]$,
\begin{gather}\label{eq:limdrift} 
N^{1/3}\E{Z^{N,p}_{n+1}-Z^{N,p}_n \, |\, Z^{N,p}_n=r}-\left[ \lambda -\tfrac{n}{N^{2/3}}+ \alpha\left(\tfrac{r}{N^{1/3}},\lambda - \tfrac{n}{N^{2/3}}\right)\right]  \rightarrow 0,
\\
%\label{varjumpprop}{\bf With the same notation and
%conditions as the previous theorem,} and for any 
%$\delta >0$ uniformly on $m\in [0,TN^{2/3}]$ 
%and $r\in[1,\rho N^{1/3}]$,
\label{eq:limvar} 
\E{\left[Z^{N,p}_{n+1}-Z^{N,p}_n\right]^2 \, \Big|\, Z^{N,p}_n=r}\rightarrow 1,\\
\label{eq:limnojumps} 
N^{2/3}\Prob{\big|Z^{N,p}_{n+1}-Z^{N,p}_n \big|> \delta N^{1/3}\,\Big| \, Z^{N,p}_n=r} \rightarrow 0,
\end{gather}
as $N\rightarrow\infty$. In addition,
\begin{equation}\label{eq:limnonsticky}
\liminf_{N\rightarrow\infty}\inf_{n\in[0,TN^{2/3}]}\E{\left[Z^{N,p}_{n+1}\right]^2 \, \Big|\, Z^{N,p}_n=0}>0.
\end{equation}
\end{prop}
Here (\ref{eq:limdrift}) 
and (\ref{eq:limvar}) give the required convergence of
the mean and variance of the increments respectively. 
Then (\ref{eq:limnojumps}) will imply that the limit process
does not have jumps, and finally (\ref{eq:limnonsticky})
ensures that the limit process reflects appropriately
at zero.

%The proof of this proposition is completed in Section \ref{Etiltsection}, after some preliminary asymptotic calculations concerning forests in random graphs.
%
%\par
%It is also necessary to establish the convergence of the variance of the rescaled increments, and regularity properties that ensure the limit process is continuous and does not stick at zero.
%
%
%The proof of this proposition occupies Section \ref{varjumpproof}.

\par
In Section \ref{SDEsection}, we show that 
Proposition \ref{limdriftprop}
is sufficient to imply Theorem \ref{explconvtheorem}. The main ingredient will be Theorem \ref{altSVreflthm}, a special case of Stroock and Varadhan's very general results \cite{SVDiffBCs} on the convergence of Markov processes to reflected diffusions.

We mention one further technical point which causes 
extra complication in the proof of Theorem \ref{FNptheorem},
compared to that of Aldous's Proposition \ref{prop:AldousGNp}.
We will need to go slightly beyond Theorem \ref{explconvtheorem}
in showing that the excursions of the discrete exploration process (whose lengths are the tree sizes of the forest) converge appropriately, after rescaling, to the excursions 
of $Z^\lambda$. 
%It is reasonably straightforward to show
%that, with high probability, no large trees appear late
%enough in the exploration that they are not represented in the limit process. 
To do so, we need to exclude the possibility that zeros of $Z^\lambda$ arise only
as the limits of small \textit{positive} local minima of the 
discrete processes; for this we will use the fact that, conditional on its vertex set, a tree appearing in $F(N,p)$ is a uniform random tree, whose exploration process we can approximate by a Brownian excursion. 
(In the case of 
Proposition \ref{prop:AldousGNp},
the limiting diffusion $B^\lambda$ is 
homogeneous in space; hence
Aldous was able to work instead with the unreflected
process, and correspondingly with a slightly different version of the exploration process, whose height at step $m$ is equal to the stack size minus the number of complete components already explored. Then one only needs to show that excursions above the running minimum of the discrete processes converge to excursions above the running minimum of the diffusion, which follows easily from the uniform convergence of the paths.)

\subsection{Discussion}
Before embarking on the proof of our main results, 
we discuss various aspects of the ensembles $F(N,p)$ and
$G(N,p)$, the limiting diffusion processes $Z^\lambda$ and $B^\lambda$, 
and other related models. 
 
\subsubsection{Excursions of $B^\lambda$ and of $Z^\lambda$}
Just as for the process $B^\lambda$, 
the excursions of $Z^\lambda$ occur in size-biased order. This property is inherited from the discrete exploration processes -- since the graph is exchangeable, the exploration visits the components in size-biased order. 

We also have that, conditional on their length, the excursions of $Z^\lambda$ are Brownian excursions. That is, if we condition on the set of excursion intervals of the process, the paths of the process on these intervals are independent Brownian excursions. This follows from the fact that the trees of $F(N,p)$ are uniformly distributed, given their vertex sets, and the fact that the exploration process of a uniform tree converges in distribution to a Brownian excursion \cite{LeGall05}. For the process $B^\lambda$, in contrast, the excursions are Brownian excursions weighted
by the exponential of their area
(see \cite{Aldous97}); relative to Brownian excursion, the higher drift at the beginning and lower drift at the end of the interval favours excursions with higher area in $B^\lambda$, but in $Z^\lambda$ this bias turns out to be precisely cancelled by the negative contribution to the drift from the $\alpha$ term.
These properties of size-biased ordering and Brownian excursions are not at all obvious from the definition of $Z^\lambda$. It is interesting to ask whether there are other diffusions which have both these properties. 

\subsubsection{Monotonicity properties}
\label{subsubsec:intro-monotonicity}
It's straightforward that the edge set of $G(N,p)$ is 
stochastically dominated by that of $G(N,p')$ when $p<p'$;
similarly we have stochastic domination of $G(N,m)$ 
by $G(N,m')$ for $m<m'$. 

However, there seems no obvious argument leading to analogous properties to hold for the families $F(N,p)$ and $F(N,m)$; as far as we know, the question of whether these monotonicity properties hold is open. 

From the combinatorial calculations that we use
to estimate the probability that $G(N,p)$ is
acyclic,
we obtain that in the critical window,
the number of edges in 
the forest $F(N,p)$ 
typically behaves like $N^2p/2+O(N^{1/2})$.
(In fact, much more strongly, one could immediately obtain
that the local central limit theorem for the number of
edges is the same in $F(N,p)$ as in $G(N,p)$). 
%In Proposition \ref{prop:locallimit} in Section
%\ref{mtopsection}, we prove a local central limit
%theorem for the number of edges in the graph $F(N,p)$. 
If we also had a monotonicity result, it would then be 
easy to deduce Theorem \ref{FNmtheorem} for $F(N,m)$ from
Theorem \ref{FNptheorem} for $F(N,p)$, using a simple sandwiching argument. Without it, we need to work a little harder. In Lemma \ref{lem:almostmonotone} in 
Section \ref{mtopsection}, we prove an `almost monotonicity result': for parameters in an appropriate range, we can couple a sequence of random forests 
with different numbers of edges in such a way that, 
with high probability, the edge sets are indeed monotonic. 

\subsubsection{$\lambda\to\infty$, the supercritical phase, and random planar graphs}
In the subcritical regime, the behaviours of $G(N,m)$ and $F(N,m)$ are very similar. Consider for example $m\sim cN$ 
where $0<c<1/2$. Then with probability bounded away from $1$ as $N\to\infty$, the graph $G(N,m)$ is itself acyclic. 
In both models, the size of the largest component is on the order of $\log N$. 
More broadly, the results of {\L}uczak and Pittel \cite{LuczakPittel92} 
indicate that the scaling limit for the largest components
is the same for the two models whenever $N/2-m\gg N^{2/3}$.

However the supercritical behaviour of $G(N,m)$ and $F(N,m)$ is very different. First consider the regime where 
$m\sim cN$ where $c > 1/2$. For both models, we see a single ``giant component" of linear size, and the 
second-largest component has sub-linear size. In $G(N,m)$, the second-largest component has size $O(\log N)$; we 
have the well-known ``duality" property whereby, 
once the giant component is removed, the rest of the graph
looks like a subcritical random graph. 
For $F(N,m)$, on the other hand,  
\cite{LuczakPittel92} show that the size of the second-largest tree (and, in fact, of the $k$th-largest for any $k\geq 2$) is on the order of $N^{2/3}$. The number of
vertices outside the giant tree is sufficiently large that
the remainder of the graph looks
\textit{critical} rather than sub-critical.

The authors of \cite{LuczakPittel92} also show (in Theorem 5.1) a distributional scaling limit for the $O(N^{2/3})$ fluctuations of the size of the giant tree around its mean. A supercritical random forest without its giant tree can then be treated as a critical forest with random criticality parameter $\lambda$.

As a consequence, for the models $F(N,m)$ and $F(N,p)$, the scaling limits described in terms of the excursions of the diffusion remain relevant in describing the $k$th largest components for $k\geq2$ in the supercritical regime as well as in the critical window. Although we do not state such a result here, one can show that the scaling limit for these components is (up to a uniform multiplicative correction) a \emph{mixture} of the distributions obtained in our main theorem.

The different behaviour between the random graph model and the random forest model is already visible ``at the top of the scaling window". Suppose $m=N/2+s$ where $N^{2/3}\ll s\ll N$.
For the random graph case \cite{BB84}, one has $|C_1^{G(N,m)}|\sim (4+o(1))s$ and $|C_2^{G(N,m)}|=o(N^{2/3})$. In the random 
forest, the largest component grows approximately half as quickly, with $|C_1^{F(N,m)}|\sim (2+o(1))s$, and for $k\geq 2$, $|C_k^{F(N,m)}|$ remains on the order of $N^{2/3}$.

We can also see the difference between the two models reflected in the behaviour of the diffusion processes $B^\lambda$ and $Z^\lambda$, as $\lambda\to\infty$.
When $\lambda$ becomes large, $B^\lambda$ typically  
has a single large excursion, which begins at time $o(1)$ and ends at time $2\lambda\pm o(1)$. At the end of the excursion, the drift of the process is $-\lambda+o(1)$,
and all subsequent excursions are very small. On the other hand, one can show from Theorem \ref{explconvtheorem} that the large excursion of $Z^\lambda$ again begins at time $o(1)$, but it is roughly half as long as for $B^\lambda$, ending at time 
$\lambda\pm O(1)$. At this time, the drift of the process is $O(1)$, and the next largest excursions remain of constant order as $\lambda\to\infty$. 

Another related model is that of the \textit{random planar graph} $P(N,m)$, uniformly
chosen from all planar graphs on $[N]$ with $m$ edges, which in a sense 
interpolates between $F(N,m)$ and $G(N,m)$.  
Kang and {\L}uczak \cite{KangLuczak2012} analysed the behaviour of $P(N,m)$ in 
various regimes, including the critical window $m=N/2+O(N^{2/3})$. 
In this window, the largest components of $P(N,m)$ 
are again on the scale of $N^{2/3}$, and
towards the top of the window, the scaling is similar to that of $F(N,m)$ rather 
than $G(N,m)$; if $m=N/2+s$ with $N^{2/3}\ll s\ll N$, 
then $|C_1^{P(N,m)}|\sim (2+o(1))s$, and for
$k\geq 2$, $|C_k^{P(N,m)}|=\Theta(N^{2/3})$.
It's interesting to speculate about whether one could also
obtain a scaling limit for the joint distribution of the sizes of the 
largest components of $P(N,m)$ in terms of the excursion lengths of a diffusion.
However, it's not clear whether 
one can formulate an exploration process of the graph $P(N,m)$ which has the Markov property; without this, it would perhaps be less plausible to 
obtain suitable convergence to a diffusion.

\subsection{Plan of the paper}

Section \ref{Zncvgsection} 
is devoted to proving 
Proposition \ref{limdriftprop}.
We show the Markov property for
the exploration process, and
establish the estimates
on the expectation and variance of its jumps,
and the necessary properties concerning
continuity and reflection at 0 of the limiting process.

In order to maintain the flow of the
argument as much as possible, some
of the more involved combinatorial calculations
required for Section \ref{Zncvgsection}
are postponed to Section \ref{section:combinatorial}.

In Section \ref{cptsizes}
we prove that Theorem \ref{explconvtheorem} 
(giving convergence of the exploration 
process on compact time-intervals)
implies Theorem \ref{FNptheorem}
(our main scaling limit result for $F(N,p)$). 

Section \ref{alpharegularity}
covers various technical aspects,
first justifying the regularity properties in Lemma 
\ref{gregularitylemma} and the existence
of the diffusion $Z^\lambda$,
and then applying Stroock and Varadhan's 
general theory for the convergence of Markov processes
to diffusions in order to show that 
Theorem \ref{explconvtheorem} follows from 
Proposition \ref{limdriftprop}.

At this point we have completed the proof
of Theorem \ref{FNptheorem}. Finally Section \ref{mtopsection} is devoted to the coupling
arguments needed to deduce the result for
$F(N,m)$ in Theorem \ref{FNmtheorem} 
from that for $F(N,p)$ in Theorem \ref{FNptheorem}.

%\bigskip
%Section \ref{cptsizes}: 
%Proof that Theorem \ref{explconvtheorem} implies
%Theorem \ref{FNptheorem}. (check that excursion of the
%limiting diffusion are matched by excursions of the 
%discrete exploration process.)
%
%\bigskip
%Section \ref{section:combinatorial}: 
%combinatorial proofs of results from Section 
%\ref{Zncvgsection}, specifically:
%proof of Lemma \ref{acyclicboundlemma} 
%on asymptotics of $f(N,p)$;
%proof of Lemma \ref{Anrklemma} about 
%the probability of cyclicity with certain vertices staying separated;
%proof of Lemma \ref{geomratiolemma} 
%about further asymptotics of probability of acyclicity
%and separation.
%
%\bigskip
%Section \ref{alpharegularity}: regularity of $\alpha$ and $g$; 
%existence of the diffusion $Z^\lambda$; 
%justification from Stroock/Varadhan that Proposition \ref{limdriftprop} implies Theorem \ref{explconvtheorem}. 
%
%\bigskip
%Section \ref{mtopsection}: deducing the result for $F(N,m)$ from the one for $F(N,p)$, ``almost monotone coupling".

\section{Convergence of the reflected exploration process}\label{Zncvgsection}
This section is devoted to the proof of 
Proposition \ref{limdriftprop}.
After collecting a few basic results 
concerning couplings and expected component
sizes for the models $G(N,p)$ and $F(N,p)$,
we turn to the Markov property for the 
exploration process of $F(N,p)$, and give its
transition probabilities. Then we embark
on various combinatorial calculations
concerning the probability of acyclicity
in various critical random graphs. 
Some of the more involved calculations will be completed
in Section \ref{section:combinatorial}.

%\bigskip
%Section \ref{Zncvgsection} 
%proves Proposition \ref{limdriftprop}
%\\
%Section \ref{subsec:couplings}: basic couplings and the 
%Janson/Spencer lemma.
%Section \ref{exploration-transitions}: definitions of separation, stack forest. 
%Markov property and transition probabilities for
%the exploration process. 
%\\
%Section \ref{Pacyclic}: statement of result
%giving asymptotics of $f(N,p)$. 
%Definition of the sets of variables on which 
%we prove uniform estimates. 
%\\
%Section \ref{Estack}: expected size of the stack forest.
%\\
%Section \ref{Etiltsection}: Proof
%of result about convergence of drift.
%\\
%Section \ref{varjumpproof}: Proof of results
%about convergence of variance, continuity, and non-sticky reflection.

\subsection{Proper couplings and estimates}\label{subsec:couplings}
First, we state two standard results, which we will use regularly. The first couples $G(N,p)$ as $p$ varies. The second relates $G(N,p)$ and $F(N,p)$, and follows from Strassen's theorem \cite{Strassen} and the Harris inequality \cite{Harris60}, since acyclity is a decreasing event.

\begin{lemma}\label{GNpcoupling}
For all $N\in\N$, $p\le q\in[0,1]$, there exists a coupling of $G(N,p)$ and $G(N,q)$ such that $E(G(N,p))\subseteq E(G(N,q))$ almost surely.
\end{lemma}

\begin{lemma}\label{barGstochdom}
For all $N\in\N$, $p\in[0,1)$, there exists a coupling of $G(N,p)$ and $F(N,p)$ such that $E( F(N,p)) \subseteq E( G(N,p))$ almost surely.
\end{lemma}

The following result, adapted from Janson and Spencer \cite{JansonSpencer}, controls the expected size of the component of a uniformly-chosen vertex from $G(N,p)$ in the critical window.

\begin{lemma}\cite[Corollary 5.2]{JansonSpencer}
\label{JansonSpencerprop} Fix $\lambda\in\R$, and let $(p(N))$ satisfy $N^{1/3}(Np(N)-1)\rightarrow\lambda$. Let $|C_{G(N,p)}(v)|$ be the size of the component containing a uniformly-chosen vertex $v$ in $G\left(N,p\right)$. Then there exists $\Theta^\lambda\in(0,\infty)$ such that
\begin{equation}\label{eq:ECG}
N^{-1/3}\E{|C_{G(N,p)}(v)|} \rightarrow \Theta^\lambda
\end{equation}
as $N\rightarrow\infty$. Thus by Lemma \ref{barGstochdom}, if we now let $|C_{F(N,p)}(v)|$ be the size of the component containing a uniformly-chosen vertex in $F(N,p)$, we have
\begin{equation}\label{eq:barECvlimit} \limsup_{N\rightarrow\infty}N^{-1/3}\E{|C_{F(N,p)}(v)|}\le \Theta^{\lambda}.\end{equation}
\end{lemma}

\begin{lemma}\label{Thetalimit}
$\Theta^\lambda$ is increasing as a function of $\lambda$, and $\Theta^\lambda\rightarrow 0$ as $\lambda\rightarrow-\infty$.
\begin{proof}
The increasing property follows from Lemma \ref{GNpcoupling}. Then, take $\lambda<0$ and $p=\frac{1+\lambda N^{-1/3}}{N}$. It is well-known (see \cite{vdHRGCN} for details) that the exploration process of $G(N,p)$ 
is stochastically dominated by the exploration process of $\mathcal{T}^{Np}$, the Galton--Watson tree with Poisson$(Np)$ offspring distribution. From this, we obtain 
%$|C_{G(N,p)}(v)|\le |\mathcal{T}^{Np}|$. Then
\[
\E{|C_{G(N,p)}(v)|}\le \E{|\mathcal{T}^{Np}|} 
= \frac{1}{1-Np} = \frac{N^{1/3}}{|\lambda|},
\]
and the result follows on taking $\lambda\rightarrow-\infty$.
\end{proof}
\end{lemma}

Let $S^2(G)$ be the sum of the squares of the 
component sizes in a graph $G$. 
Then the expectation in (\ref{eq:ECG}) is simply
$\E{S^2(G(N,p))}/N$ (since $C(v)$ is a size-biased
choice from the components of $G$), and the 
following corollary is an immediate 
consequence of Lemma
\ref{JansonSpencerprop}.
\begin{corollary}\label{S2corollary}
%Let $S^2(G)$ be the sum of the squares of the component sizes in a graph $G$. 
Suppose that $p(N)=1/N+O(N^{-4/3})$ as $N\to\infty$. Then
$N^{-4/3}S^2(F(N,p))$ and $N^{-4/3}S^2(G(N,p))$ are bounded in expectation.
\end{corollary}

\subsection{Stack forests}
\label{exploration-transitions}
\begin{definition}
For a graph $G$, we say a set $A\subseteq V(G)$ is \emph{separated} in $G$ if no pair of vertices in $A$ lie in the same component of $G$.
\end{definition}

Recall from Section \ref{subsec:intro-exploration} that we are considering a breadth-first exploration process of $F(N,p)$. For the remainder of this short section, we suppress notation on $N$ and $p$ in the exploration process, since the result to follow holds for all $p\in(0,1)$. Then $\mathcal{Z}_n$ is the \emph{stack} of vertices which have been seen but not explored yet. Note that all the vertices in $\mathcal{Z}_n$ are in the same component of $F(N,p)$, since components are explored one-by-one. In particular, in the graph restricted to $[N]\backslash \{v_1,\ldots,v_n\}$, no pair of vertices in $\mathcal{Z}_n$ lie in the same component, as otherwise there would be a cycle in $F(N,p)$. We refer to the $Z_n$ trees on $[N]\backslash\{v_1,\ldots,v_n\}$ containing each $v\in\mathcal{Z}_n$ as the \emph{stack forest}, as in Figure \ref{fig:stackforestdiagram}. We can see that the vertices in $\mathcal{Z}_n$ are \emph{separated} in the restricted graph on $[N]\backslash\{v_1,\ldots,v_n\}$.

\begin{figure}[h]
   \centering
\includegraphics[width=0.7\textwidth]{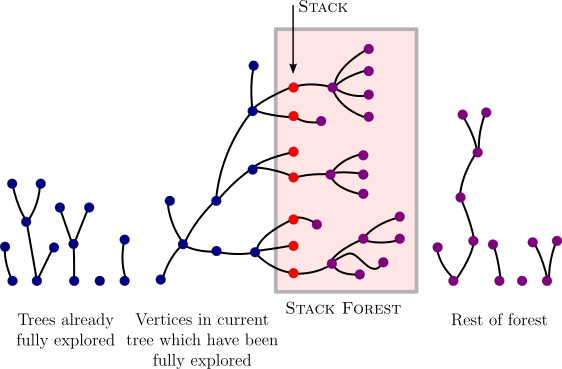}
    \caption{Illustration of the definition of \emph{stack forest}}
    \label{fig:stackforestdiagram}
\end{figure}

\par
Now, suppose we condition on $\{v_1,\ldots,v_n\}\cup \mathcal{Z}_n$, and the structure of $F(N,p)$ on these $n+Z_n$ vertices. Then, the graph restricted to $[N]\backslash\{v_1,\ldots,v_n\}$ has the same distribution as
$$F([N]\backslash\{v_1,\ldots,v_n\},p),$$
with the extra condition that no pair of vertices from $\mathcal{Z}_n$ lie in the same component.

\par
We expand this explanation considerably in the proof of the following lemma, which formalises the claim that $(Z_n)_{n\ge 0}$ is Markov, and characterises its transition probabilities via separation of the current stack in the remainder of the graph.

\begin{lemma}\label{incrdescription}
Let $(Z_n)_{n\ge 0}$ be the exploration process of $F(N,p)$. Then $Z$ is a Markov chain, and for $n\geq 0$ and $r\geq 1$, 
%\begin{align}
%&\Prob{Z_{m+1} - Z_m = \ell-1 \,\big|\, Z_m=r, Z_{m-1}=r_{m-1},\ldots,Z_1=r_1}\label{eq:ZmMarkov}\\
%&\qquad \propto \binom{N-m-r}{\ell}p^\ell (1-p)^{N-m-r-\ell}\Prob{[r+\ell-1]\text{ separated in }F(N-m-1,p)},\nonumber
%\end{align}
%for $\ell=0,1,\ldots,N-m-r$.
\begin{multline}
\Prob{Z_{n+1} = r+\ell-1 \,\big|\, Z_n=r}
\label{eq:ZmMarkov}
\\
\propto
\binom{N-n-r}{\ell}p^\ell (1-p)^{N-n-r-\ell}
\Prob{[r+\ell-1]\text{ separated in }F(N-n-1,p)}
\end{multline}
as $\ell$ varies over $\{0, 1,\dots, N-n-r\}$.

The distribution of $Z_{n+1}$ given $Z_n=0$ is 
the same as the distribution of $Z_{n+1}$ given $Z_n=1$.
\end{lemma}
\begin{proof}

%Note: this version of the proof refers 
%to a set-up where the exploration process 
%is deterministic. 

In the $F(N,p)$ model, 
each forest $H$ appears with probability proportional to 
$\left(\frac{p}{1-p}\right)^{|E(H)|}$,
where $|E(H)|$ is the number of edges of $H$.

Consider the first $n$ steps of the exploration process.
As well as conditioning on $Z_n=r$, consider conditioning
further on the history $(Z_1, \dots, Z_{n-1})$, 
on the identity of the processed vertices 
$v_1, \dots, v_n$ and on the vertices 
$v_{n+1}, \dots, v_{n+r}$ currently on the stack. 
Let us write $\cV_n=\{v_1, \dots, v_n\}$ 
for the processed vertices, 
$\cZ_n=\{v_{n+1}, \dots, v_{n+r}\}$ for the stack,
and $\cU_n=[N]\setminus \left(\cV_n \cup \cZ_n\right)$
for the remaining vertices. Assume for
the moment that $r\geq1$, i.e.\ that $\cZ_n$ 
is non-empty. 

The conditioning determines the edges of the 
graph $H$ restricted to the vertex set 
$\cV_n\cup \cZ_n$. Furthermore, under this condition
there are no edges between 
$\cV_n$ and $\cU_n$, with probability 1. 
So to specify $H$ fully it is now 
enough to give the restriction 
$\tH=H|_{\cZ_n\cup \cU_n}$ 
of $H$ to the vertex set $\cZ_n\cup \cU_n$.

The set of $\tH$ which are consistent with
the conditioning is the set of $\tH$
for which $H$ is a forest; for this, we require
precisely that $\tH$ is a forest in which  
the vertices of the stack $\cZ_n$
are separated. Subject to this constraint,
each $\tH$ appears with probability 
proportional to 
$\left(\frac{p}{1-p}\right)^{|E(\tH)|}$
where $|E(\tH)|$ is the number of edges of $\tH$. 
After a suitable relabelling of the vertices, 
this gives the model 
$F(N-n,p)$ subject to the condition
that the vertices of $[r]$ are separated.

The event in (\ref{eq:ZmMarkov}) occurs if the next increment $Z_{n+1}-Z_n$ of the exploration process has size $\ell-1$.
This occurs if vertex $v_{n+1}$ 
(the next vertex to be processed, which is currently on the stack) has $\ell$ neighbours in $\cU_n$. 
The conditional probability of an increment of
size $\ell-1$ is then equal to 
%\[
%\alpha_{N,m,r,\ell}:=
%\Prob{1 \text{ has degree }\ell \text{ in }
%F(N-m-r, p) \,\big| \, 
%[r] \text{ separated in } F(N-m-r, p)}.
%\]
\[
\alpha_{N,n,r,\ell}:=
\Prob{\deg_{\tH}(1)=\ell
\,\big|\,
[r] \text{ separated in }\tH}
\text{ where }\tH\sim F(N-n-r,p).
\]
Note that this $\alpha_{N,n,r,\ell}$ depends on the 
history we conditioned on only through the
value of $r$; hence in particular, given $Z_n=r$,
the next increment
is independent of the history $(Z_1,\dots, Z_{n-1})=
(r_1,\dots, r_{n-1})$ of the exploration process,
as required for the Markov property to hold, 
and the conditional probability on the left of
(\ref{eq:ZmMarkov}) is also equal to $\alpha_{N,n,r,\ell}$.  

There are $\ch{N-n-r}{\ell}$ ways to choose 
$\ell$ neighbours outside $[r]$ for vertex $1$.
Without loss of generality, consider the
case where these $\ell$ neighbours are
$r+1, r+2, \dots, r+\ell$. Then the property that 
$[r]$ is separated
is equivalent to the property that the set 
${2,3,\dots,r+\ell}$
is separated in the graph with vertex $1$ removed. 

Reasoning in this way, and omitting factors which are constant in $\ell$, we obtain 
%that the conditional 
%probability that the next increment is $\ell-1$ 
%as desired for (\ref{eq:ZmMarkov}) is given by
\begin{align*}
\alpha_{N,n,r,\ell}
&\propto 
\Prob{F(N-n-r,p) \text{ has }
[r] \text{ separated and } \deg(1)=\ell}\\
&\propto 
\Prob{G(N-n-r,p) \text{ is a forest with }
[r] \text{ separated and } \deg(1)=\ell}\\
&=
\Ch{N-n-r}{\ell}p^{\ell}(1-p)^{N-n-r-\ell}
\\
&\,\,\,\,\,\,\,\,\,\,\,\,\,\,\,\,\,\,
\times
\Prob{G(N-n-r-1, p) \text{ is a forest with }
[r+\ell-1] \text{ separated}}\\
&\propto
\Ch{N-n-r}{\ell}p^{\ell}(1-p)^{N-n-r-\ell}
\\
&\,\,\,\,\,\,\,\,\,\,\,\,\,\,\,\,\,\,
\times
\Prob{F(N-n-r-1, p) \text{ has } 
[r+\ell-1] \text{ separated}}.
\end{align*}
This is equal to the right-hand side of 
(\ref{eq:ZmMarkov}) as desired.

Observe that in the argument above, if $r=|\cZ_n|=1$
then
the property that the vertices of $\cZ_n$
are separated in $\tH$ becomes vacuously true;
all forests $\tH$ are consistent with the
history of the exploration process,
and the conditional distribution of $\tH$ 
above becomes simply that of $F(N-n,p)$. In
the case $r=0$, where the stack is empty,
we start exploring again from a new vertex
(specifically, the vertex in $\cU_n$ with 
smallest label). Again all forests $\tH$ are 
consistent with the history, and so in fact
the law of the rest of the process in the case 
$Z_n=0$ is the same as that in the case $Z_n=1$,
as desired. 
\end{proof}

We want to quantify exactly how large a probabilistic penalty is incurred by adding an extra vertex to the stack, and so will consider limits of the quantity
$$\frac{\Prob{[r+\ell]\text{ separated in }F(N-n-1,p)}}{\Prob{[r+\ell-1]\text{ separated in }F(N-n-1,p)}}.$$

Given a graph in which $[r+\ell-1]$ are separated, the conditional probability that $r+\ell$ is also separated depends on the size of the stack forest rooted by $[r+\ell-1]$. So we will calculate the expected size of a stack forest in Section \ref{Estack}. We need precise asymptotics for the probability that $G(N,p)$ is acyclic, which we derive in Section \ref{Pacyclic}. We then use this to calculate the probability that the stack forest has a particular size.

\subsection{Enumerating weighted stack forests}\label{Pacyclic}

In this section, we consider the probability that $G(N,p)$ is acyclic.

\begin{definition}Let $f(N,m)$ be the number of forests with vertex set $[N]$ and exactly $m$ edges. With a mild abuse of notation, we also define
\begin{equation}\label{eq:weightedFs1}f(N,p):=\Prob{G(N,p)\text{ acyclic}} = (1-p)^{\binom{N}{2}} \sum_{m=0}^{N-1} f(N,m)\left(\frac{p}{1-p}\right)^m,\quad p\in[0,1].\end{equation}
\end{definition}

\begin{lemma}\label{fNvsfN+1lemma}
For any $N\ge 0$ and any $p\in(0,1)$,
\begin{equation}\label{eq:bdconsecutivefnm}f(N,p)\ge f(N+1,p)\ge f(N,p)\left[1-\frac12 Np^2\E{|C_{G(N,p)}(v)|}\right],\end{equation}
where $C_{G(N,p)}(v)$ is the component containing a uniformly-chosen vertex $v$ in $G(N,p)$.
\begin{proof}
Graphs with zero, one or two vertices are certainly acyclic, so $f(0,p)=f(1,p)=f(2,p)=1$, the statement is true for $N=0,1$. We assume from now on that $N\ge 2$. We can define a forest on $[N+1]$ via the restriction to $[N]$ (which is clearly also a forest) and the neighbourhood of vertex $N+1$, where the latter must obey some conditions to avoid cycles. We take $\mathbb{P}$ to be a probability distribution which couples $G(N,p)$ and $G(N+1,p)$ such that $E(G(N,p))\subseteq E(G(N+1,p))$, $\mathbb{P}$-a.s. Recall that in a graph $G$, for $v\in V(G)$, $\Gamma(v)$ is the set of vertices connected to $v$ by an edge in $E(G)$. Then
$$f(N+1,p)= f(N,p) \Prob{\Gamma(N+1) \text{ separated in }G(N,p) \,\big|\, G(N,p)\text{ acyclic}},$$
and so the first inequality in \eqref{eq:bdconsecutivefnm} certainly holds. Now, for any set $A\subset [N]$, the event that $A$ is separated in $G$ is decreasing, while the event that $G$ is acyclic is also decreasing. So, again by the Harris inequality,
$$f(N+1,p)\ge f(N,p) \Prob{\Gamma(N+1)\text{ separated in }G(N,p) },$$
and so
\begin{equation}\label{eq:Fratioviaseparation}1-\frac{f(N+1,p)}{f(N,p)} \le \Prob{\Gamma(N+1)\text{ not separated in }G(N,p)}.\end{equation}

Observe that the event that $\Gamma(N+1)$ is not separated in $G(N,p)$ is the union over $i,j\in[N]$ of the events
$$\{i,j\text{ both in }\Gamma(N+1)\text{ and both in the same component of }G(N,p)\}.$$
Thus, by exchangeability of the vertices in $[N]$,
$$\Prob{\Gamma(N+1)\text{ not separated in }G(N,p) } \le \binom{N}{2}p^2 \,\Prob{\text{1 and 2 in same component of }G(N,p)}.$$

Then, if $|C_{G(N,p)}(1)|$ is the size of the component of $G(N,p)$ containing vertex 1,
$$\Prob{\text{1 and 2 in same component of }G(N,p)} = \frac{\E{|C_{G(N,p)}(1)|}-1}{N-1}.$$

We conclude that
\begin{align*}
\Prob{\Gamma(N+1)\text{ not separated in }G(N,p) } & \le \binom{N}{2}p^2\cdot  \frac{\E{|C_{G(N,p)}(1)|}-1}{N-1}
\\
&\le \frac12 Np^2\E{|C_{G(N,p)}(1)|},
\end{align*}
from which the result follows, using \eqref{eq:Fratioviaseparation} and the fact that the vertices in $G(N,p)$ are exchangeable.
\end{proof}
\end{lemma}

Now, using the asymptotics for $f(N,m)$ in \eqref{eq:Britikovresult}, we may obtain asymptotics for $f(N,p)$. Here, and in subsequent sections, some straightforward but lengthy calculations are required, and in some places, various expansions have to be taken to fifth order. To avoid breaking the flow of the main argument, we postpone this proof until Section \ref{acyclicboundproof}.

\begin{lemma}\label{acyclicboundlemma}Fix $\lambda^-<\lambda^+\in\R$. Given $p\in(0,1)$, let $\Lambda=\Lambda(N,p)=N^{1/3}(Np-1)$. Then
\begin{equation}\label{eq:weightedFs2}f(N,p)=\Prob{G(N,p)\text{ acyclic}}=(1+o(1))g(\Lambda)e^{3/4}\sqrt{2\pi} N^{-1/6} ,\end{equation}
uniformly for $\Lambda\in[\lambda^-,\lambda^+]$ as $N\rightarrow\infty$.
\end{lemma}

Motivated by the definition of stack forests, for each $0\le r\le N$, let $\mathcal A_{N,r}\subseteq \mathcal F_N$ denote the set of forests where the vertices $1,\ldots,r$ are separated. Furthermore, given a forest $F\in\mathcal A_{N,r}$, let $k_r(F)$ be the sum of the sizes of the components containing vertices $1,\ldots,r$. We also define
\begin{equation}\label{eq:defnAnrk}\mathcal A_{N,r,k}:=\{F\in \mathcal A_{N,r},\, k_r(F)=k\},\end{equation}
the set of forests where $1,\ldots,r$ are separated, and their stack forest has size $k$.

\begin{definition}\label{defnrescalings}
Given $p\in(0,1)$ and $N,N',r,k\in\N$ satisfying $N'\le N$, and $r\le k\le N$, we will use the following rescalings:
$$\Lambda= \Lambda(N,p):= N^{1/3}(Np-1),\quad a= a(N,k):=\frac{k}{N^{2/3}},$$
\begin{equation}\label{eq:defnbetc}b=b(N,r):=\frac{r}{N^{1/3}},\quad s=s(N,N'):= \frac{N-N'}{N^{2/3}}.\end{equation}
Here, $b$ represent the rescaled size of the stack and $a$ represents the rescaled size of the stack forest. When analysing the exploration process of $G(N,p)$, we require estimates for the graph structure on the $N'\le N$ vertices which have not yet been explored. Then $s$ represents the rescaled number of vertices already explored in the exploration process.
\begin{note}Observe that for $p(N)$ satisfying the conditions of Theorem \ref{FNptheorem}, $\Lambda(N,p(N))\rightarrow \lambda$.
\end{note}
\end{definition}

\begin{definition}\label{defnPsi}
For much of this and the following sections, it will be necessary to make estimates uniformly across several variables. For constants $T<\infty$, and $\lambda^-<\lambda^+$, and $0<\epsilon<K<\infty$, we let
$$\Psi^N(\lambda^-,\lambda^+,\epsilon,K,T) := \Big\{(N',p,r,k)\in \N\times (0,1)\times \N\times \N \,:\, s(N,N')\in[0,T],$$ $$\Lambda(N,p)\in[\lambda^-,\lambda^+],\, b(N,r)\in [\epsilon,K],\, k\in[r,KN^{2/3}]\Big\}.$$
In addition, we define the projection this set onto its first three entries
$$\Psi^N_0(\lambda^-,\lambda^+,\epsilon,K,T):= \Big\{(N',p,r)\,:\, s(N,N')\in[0,T],\, \Lambda(N,p)\in[\lambda^-,\lambda^+],\, b(N,r)\in[\epsilon,K]\Big\},$$
and a variant with a broader range of $r$
$$\bar \Psi^N_0(\lambda^-,\lambda^+,K,T):= \Big\{(N',p'r)\,:\, s(N,N')\in[0,T],\, \Lambda(N,p)\in[\lambda^-,\lambda^+],\, r\in[1,KN^{1/3}]\Big\}.$$

\end{definition}

The following lemma gives uniform asymptotics for the probability that $G(N',p)$ lies in $\mathcal{A}_{N',r,k}$. The proof is postponed until Section \ref{Anrklemmaproof}.

\begin{lemma}\label{Anrklemma}Fix constants $\lambda^-,\lambda^+,\epsilon,K,T$ as in Definition \ref{defnPsi}. Then,
\begin{equation}\label{eq:asymPAn'rk}\Prob{G(N',p)\in\mathcal A_{N',r,k}} = (1+o(1))e^{3/4} g(\Lambda-s-a)N^{-5/6} ba^{-3/2}\end{equation}
$$\qquad\qquad \times \exp\left(-b(\Lambda-s)-\tfrac{b^2}{2a}-\tfrac{(\Lambda-s-a)^3-(\Lambda-s)^3}{6}\right),$$
uniformly on $(N',p,r,k)\in\Psi^N(\lambda^-,\lambda^+,\epsilon,K,T)$, as $N\rightarrow\infty$.
\end{lemma}

\subsection{Expected size of the stack forest}\label{Estack}

We now condition on $[r]$ being separated in $F(N',p)$, and obtain an estimate for the expected size of the corresponding stack forest. Recall from \eqref{eq:defnbetc} the definitions $b=b(N,r)$ and $s=s(N,N')$, the rescaled stack size, and graph vertex count deficit, respectively.

\begin{lemma}\label{Estacklemma}Fix constants $\lambda^-,\lambda^+,K,T$ as in Definition \ref{defnPsi}. Then,
\begin{equation}\label{eq:EstackF}N^{-2/3}\E{ k_r(F(N',p)) \,\big| \, F(N',p)\in\mathcal A_{N',r}} - \alpha \left(b, \Lambda-s\right) \rightarrow 0,\end{equation}
uniformly on $(N',p,r)\in\bar \Psi_0^N(\lambda^-,\lambda^+,K,T)$, as $N\rightarrow \infty$.

\begin{proof}
We can rewrite the expectation in \eqref{eq:EstackF} in terms of the unconditioned random graphs $G(N',p)$ as follows.
\begin{align}
\E{ k_r(F(N',p)) \,\big| \, F(N',p)\in\mathcal A_{N',r}} &= \frac {\sum_{k=r}^{N'} k\Prob{F(N',p)\in\mathcal A_{N',r,k}}}{\sum_{k=r}^{N'} \Prob{F(N',p)\in\mathcal A_{N',r,k}}}\nonumber\\
&= \frac {\sum_{k=r}^{N'} k\Prob{ G(N',p)\in\mathcal A_{N',r,k}}}{\sum_{k=r}^{N'} \Prob{G(N',p)\in\mathcal A_{N',r,k}}}.\label{eq:Ekrasfrac}
\end{align}
We shall see that both of the sums in \eqref{eq:Ekrasfrac} are dominated by contributions from $k=\Theta(N^{2/3})$.

\par
In order to use Lemma \ref{Anrklemma}, we assume $\epsilon\in(0,K)$ is given. We will first show that \eqref{eq:EstackF} holds uniformly on $\Psi^N(\lambda^-,\lambda^+,\epsilon,K,T)$. Then, at the end, we will take $\epsilon\rightarrow 0$. We also select $M>K$, which we will take to $\infty$ shortly.

\par
We write $h(a,b):=a^{-3/2}g(\Lambda-s-a)\exp\left(\tfrac{(\Lambda-s-a)^3-(\Lambda-s)^3}{6}\right)\exp(-b^2/2a)$. Since $g$ is bounded, $h(a,b)\rightarrow 0$ as $a\rightarrow 0$ (indeed uniformly on $b\in[\epsilon,K]$, $\Lambda\in \R$, $s\in \R_{\ge 0}$), so $\int_0^M h(a,b)\mathrm{d}a<\infty$ for all $M<\infty$. On compact intervals in $(a,b,\Lambda,s)$, $h$ is uniformly continuous and bounded away from zero. We may now use Lemma \ref{Anrklemma} to approximate every summand in \eqref{eq:Ekrasfrac}, uniformly over the required range. (Recall from \eqref{eq:defnbetc} that $a$ is a linear function of $k$.) So
$$\sum_{k=r}^{\ce{MN^{2/3}}} \Prob{ G(N',p)\in\mathcal A_{N',r,k}} =(1+o(1))bN^{-5/6}\exp\left(-b(\Lambda-s)-\tfrac{(\Lambda-s)^3}{6}+\tfrac34\right)$$
$$\qquad \times N^{2/3}\int_0 ^M a^{-3/2}g(\Lambda-s-a)\exp\left(\tfrac{(\Lambda-s-a)^3}{6}\right)\exp\left(-\tfrac{b^2}{2a}\right)\mathrm{d}a,$$
$$\sum _{k=r}^{\ce{MN^{2/3}}} k\Prob{ G(N',p)\in\mathcal A_{N',r,k}} =(1+o(1)) bN^{-5/6}\exp\left(-b(\Lambda-s)-\tfrac{(\Lambda-s)^3}{6}+\tfrac34\right)$$
\begin{equation}\label{eq:Estackfnumerator}\qquad \times N^{4/3}\int_0 ^M a^{-1/2}g(\Lambda-s-a)\exp\left(\tfrac{(\Lambda-s-a)^3}{6}\right)\exp\left(-\tfrac{b^2}{2a}\right)\mathrm{d}a,\end{equation}
uniformly on $(N',p,r)\in\Psi^N_0(\lambda^-,\lambda^+,\epsilon,K,T)$, as $N\rightarrow\infty$.

\par
Observe, by comparison with the definition of $\alpha$ in \eqref{eq:defnalpha}, that
$$\lim_{M\rightarrow\infty}N^{-2/3}\frac{\sum_{k=r}^{\ce{MN^{2/3}}} k\Prob{ G(N',p)\in\mathcal A_{N',r,k}} }{\sum _{k=r}^{\ce{MN^{2/3}}} \Prob{ G(N',p)\in\mathcal A_{N',r,k}}} = (1+o(1))\alpha(b,\Lambda-s),$$
uniformly on $(N',p,r)\in\Psi^N_0(\lambda^-,\lambda^+,\epsilon,K,T)$.

\par
Therefore, to apply \eqref{eq:Ekrasfrac} to verify \eqref{eq:EstackF}, we must check that the contribution to the expectation from the event that the size of the stack forest is larger than $M N^{2/3}$ vanishes as $M\rightarrow\infty$. From \eqref{eq:Estackfnumerator}, the contribution to the numerator of \eqref{eq:Ekrasfrac} from summands for which $k\in[r,\ce{MN^{2/3}}]$ has order $N^{-5/6}\times N^{4/3}=N^{1/2}$. So to verify \eqref{eq:EstackF} uniformly on $\Psi^N_0(\lambda^-,\lambda^+,\epsilon,K,T)$, it will suffice to check that the following statement holds:
\begin{equation}\label{eq:bdlargestackf} \lim_{M\rightarrow \infty} \limsup_{N\rightarrow \infty} \sup_{(N',p,r)\in \Psi^N_0(\lambda^-,\lambda^+,\epsilon,K,T)}N^{-1/2} \sum_{k=\fl{MN^{2/3}}}^{N'} k\Prob{ G(N',p)\in\mathcal A_{N',r,k}}=0.\end{equation}

\subsubsection*{The stack forest is not too large}
To show \eqref{eq:bdlargestackf}, we will show that the sequence $(k\Prob{ G(N',p)\in \mathcal{A}_{N',r,k}})_{k\ge r}$ is eventually bounded by a geometric series. From the definition of $F(N,p)$ in \eqref{eq:weightedFs1}, we have that
\begin{equation}\label{eq:PAnrk}\Prob{G(N,p)\in \mathcal{A}_{N,r,k}} = (1-p)^{\binom{N}{2}-\binom{N-k}{2}}\binom{N-r}{k-r}\left(\frac{p}{1-p}\right)^{k-r}rk^{k-r-1} F(N-k,p).\end{equation}

An explanation of where each term in this expression comes from is given in the proof of Lemma \ref{Anrklemma} in Section \ref{Anrklemmaproof}. We will use this to control the ratio of the probabilities $\Prob{G(N',p)\in \mathcal{A}_{N',r,k}}$ in the following lemma.
\begin{lemma}\label{geomratiolemma}
Given the same constants as in Lemma \ref{Estacklemma}, there exist constants $M<\infty$ and $\gamma>0$ such that
\begin{equation}\label{eq:geomratiobd}\frac{(k+1)\Prob{ G(N',p)\in \mathcal{A}_{N',r,k+1}}}{k\Prob{ G(N',p)\in \mathcal{A}_{N',r,k}}} \le 1- \gamma N^{-2/3},\end{equation}
for large enough $N$, whenever $(N',p,r)\in \bar \Psi_0^N(\lambda^-,\lambda^+,K,T)$ and $k\in[MN^{2/3},N'-1]$.
\end{lemma}

This lemma is proved in Section \ref{geomratioproof}. But then, we can bound \eqref{eq:bdlargestackf} via a geometric series as
$$N^{-1/2}\sum_{k=MN^{2/3}}^{N'} k\Prob{G(N',p)\in \mathcal{A}_{N',r,k} } \le N^{-1/2}\frac{\ce{M N^{2/3}} \Prob{G(N',p)\in \mathcal{A}_{N',r,\ce{MN^{2/3}}}}}{1 - (1-\gamma N^{-2/3})}.$$
By Lemma \ref{Anrklemma}, this RHS is
\begin{align*}
&(1+o(1))N^{-1/2}\frac{1}{\gamma} N^{2/3}\cdot MN^{2/3} e^{3/4} g(\Lambda-s-M) N^{-5/6} b M^{-3/2}\\
&\qquad\times\,\exp\left(-b(\Lambda-s)-\tfrac{b^2}{2M}+\tfrac{(\Lambda-s-M)^3-(\Lambda-s)^3}{6}\right)\\
&= (1+o(1)) M^{-1/2} e^{-b^2/2M} \exp\left(\tfrac{(\Lambda-s-M)^3-(\Lambda-s)^3}{6}\right)\times g(\Lambda-s-M)\\
&\qquad\times \tfrac{e^{3/4}}{\gamma} b \exp\left(-b(\Lambda-s)-\tfrac{(\Lambda-s)^3}{6}\right).
\end{align*}
Recall that $g$ is uniformly bounded above and $\exp\left(\tfrac{(\Lambda-s-M)^3-(\Lambda-s)^3}{6}\right)\le 1$. Then observe that $M^{-1/2}e^{b^2/2M}\rightarrow 0$ as $M\rightarrow\infty$. Therefore
$$\lim_{M\rightarrow \infty} \limsup_{N\rightarrow \infty} \sup_{(N',p,r)\in \Psi^N_0(\lambda^-,\lambda^+,\epsilon,K,T)}N^{-1/2} \sum_{k=\fl{MN^{2/3}}}^{N'} k\Prob{ G(N',p)\in\mathcal A_{N',r,k}}=0.$$
So we have finished the proof of \eqref{eq:bdlargestackf}, and thus we have shown that \eqref{eq:EstackF} holds uniformly on $\Psi_0^N(\lambda^-,\lambda^+,\epsilon,K,T)$.

\subsubsection*{Small stacks}
To finish this proof of Lemma \ref{Estacklemma}, it remains to extend the convergence to uniformity on $r\in[1,\ce{K n^{1/3}}]$, rather than on $[\fl{\epsilon N^{1/3}},\ce{KN^{1/3}}]$.

\par
Recall from Lemma \ref{gregularitylemma} that $\alpha(b,\Lambda)\rightarrow 0$ as $b\downarrow 0$ uniformly on compact intervals in $\Lambda$. In particular
\begin{equation}\label{eq:limsupalphaoverPsi}\lim_{\epsilon\rightarrow 0} \limsup_{N\rightarrow\infty} \sup_{\substack{\Lambda\in[\lambda^-,\lambda^+]\\s\in[0,T],\,r\in[1,\epsilon N^{1/3}]}} \alpha\left(\tfrac{r}{N^{1/3}}, \Lambda-s\right) =0.\end{equation}

\par
Before Definition \ref{defnPsi}, we defined $k_r(F)$ for a forest $F$, but we can extend the definition to a general graph $G$ with vertex set $[N]$. If $|C(i)|$ is the size of the component containing vertex $i\in[N]$, then set $k_r(G):= |C(1)|+\ldots+|C(r)|$, so some components may be counted at least twice. In particular, $k_r(G)$ is an increasing function of graphs. However, for any $r$, the set $\mathcal{A}_{N,r}$ is a decreasing family of graphs. Therefore
\begin{equation}\label{eq:boundEkr}\E{k_r(G(N',p)) \,\big|\, G(N',p)\in \mathcal{A}_{N',r}} \le \E{k_r(G(N',p))}\le r \E{|C_{G(N',p)}(1)|},\end{equation}
where $|C_{G(N',p)}(1)|$ is the size of the component containing vertex 1 in $G(N',p)$. From Lemma \ref{JansonSpencerprop}, for the range of $N',p$ under consideration,
\begin{equation}\label{eq:unifEC1bd}\limsup_{N\rightarrow\infty}\sup_{\substack{N'\in[N-TN^{2/3},N]\\\Lambda(N,p)\in[\lambda^-,\lambda^+]}} N^{-1/3}\E{|C_{G(N',p)}(1)|}\le \Theta^{\lambda^+}<\infty.\end{equation}
We now take $r\le \epsilon N^{1/3}$ in \eqref{eq:boundEkr}, and apply \eqref{eq:unifEC1bd} to obtain
$$\lim_{\epsilon\rightarrow 0}\limsup_{N\rightarrow\infty}\sup_{\substack{N'\in[N-TN^{2/3},N]\\\Lambda(N,p)\in[\lambda^-,\lambda^+]\\r\in[1,\epsilon N^{1/3}]}}  N^{-2/3}\E{k_r(G(N',p)) \,\big|\, G(N',p)\in \mathcal{A}_{N',r}}\le\lim_{\epsilon\rightarrow 0}\epsilon \Theta^{\lambda^+}=0.$$
So, with \eqref{eq:limsupalphaoverPsi}, this gives
\begin{equation}\label{eq:approxgoodnear0}\lim_{\epsilon\rightarrow 0}\limsup_{N\rightarrow\infty}\sup_{\substack{N'\in[N-TN^{2/3},N]\\\Lambda(N,p)\in[\lambda^-,\lambda^+]\\r\in[1,\epsilon N^{1/3}]}}  \left| N^{-2/3}\E{k_r(G(N',p)) \,\big|\, G(N',p)\in \mathcal{A}_{N',r}} - \alpha\left(\tfrac{r}{N^{1/3}}, \Lambda-s\right) \right|=0.\end{equation}

We already know that \eqref{eq:EstackF} holds uniformly on $\Psi^N(\lambda^-,\lambda^+,\epsilon,K,T)$. So, combining with \eqref{eq:approxgoodnear0} and taking $\epsilon$ small shows that \eqref{eq:EstackF} does hold uniformly on $(N',p,r)\in\bar \Psi^N_0(\lambda^-,\lambda^+,K,T)$, as required for the full statement of Lemma \ref{Estacklemma}.
\end{proof}

\end{lemma}

\subsection{Proof of Proposition \ref{limdriftprop}:
convergence of the drift}\label{Etiltsection}

Recall that $\mathcal{A}_{N,r}\subseteq \mathcal{F}_N$ is the set of forests on $[N]$ where vertices $1,\ldots,r$ are separated. Let $F$ be a uniform choice from $\mathcal A_{N,r}$. Then
$$\Prob{F\in \mathcal A_{N,r+1} \,\big|\, F\in \mathcal A_{N,r,k}} = \frac{N-k}{N-r},$$
as the labels of the $k-r$ other vertices in the stack forest containing vertices $[r]$ are uniformly chosen from $\{r+1,\ldots,N\}$. Furthermore, $\mathcal A_{N,r+1}\subseteq \mathcal A_{N,r}$, and so
\begin{align*}
\frac{\Prob{F(N,p)\in\mathcal A_{N,r+1}}}{\Prob{F(N,p)\in\mathcal A_{N,r}}} &= \Prob{F(N,p)\in\mathcal A_{N,r+1} \, \big |\, F(N,p)\in\mathcal A_{N,r}}\\
& = \sum_{k=r}^N \Prob{F(N,p)\in\mathcal A_{N,r+1} \, \big |\, F(N,p)\in\mathcal A_{N,r,k}}\\
&\quad \times \,\Prob{F(N,p)\in\mathcal A_{N,r,k} \, \big |\, F(N,p)\in\mathcal A_{N,r}}\\
& = \frac{N- \E{k_r(F(N,p))\,\big|\, F(N,p)\in\mathcal A_{N,r}}}{N-r}.
\end{align*}

It follows that uniformly on $(N',p,r)\in \bar{\Psi}_0^N(\lambda^-,\lambda^+,K,T)$, as in Lemma \ref{Estacklemma}, as $N\rightarrow\infty$,
$$N^{1/3}\left[1- \frac{\Prob{F(N',p)\in\mathcal A_{N',r+1}}}{\Prob{F(N',p)\in\mathcal A_{N',r}}}\right]- \alpha \left(\tfrac{r}{N^{1/3}}, \Lambda-s\right)\rightarrow 0.$$

The sequence $p(N)$ satisfies the conditions in the statement of Theorem \ref{FNptheorem}, that is $\Lambda(N,p(N))\rightarrow\lambda\in\R$ in the notation of Definition \ref{defnrescalings}. So in fact we may replace $\Lambda$ with $\lambda$, obtaining, again uniformly on $(N',p,r)\in \bar{\Psi}_0^N(\lambda^-,\lambda^+,K,T)$,
\begin{equation}\label{eq:ratioAnr}N^{1/3}\left[1- \frac{\Prob{F(N',p)\in\mathcal A_{N',r+1}}}{\Prob{F(N',p)\in\mathcal A_{N',r}}}\right]- \alpha \left(\tfrac{r}{N^{1/3}}, \lambda-s\right)\rightarrow 0.\end{equation}

Now we can return to the increments of $Z^{N,p}$, the exploration process of $F(N,p)$. Recall Lemma \ref{incrdescription}, which asserts that
$$\Prob{Z^{N,p}_{n+1} - Z^{N,p}_n = \ell-1 \,\big|\, Z^{N,p}_n=r} \propto \Prob{ B^{N-n-r,p}=\ell}$$
$$ \times \,\Prob{F(N-n-1,p)\in \mathcal{A}_{N-n-1,r+\ell-1}},\quad \ell\ge 0,$$
where $B^{N-n-r,p}\sim \Bin{N-n-r}{p}$. So we define
\begin{equation}\label{eq:defnqNmr}q^{N,n,r}_\ell := \Prob{B^{N-n-r,p}=\ell}\times \frac{\Prob{F(N-n-1,p)\in \mathcal{A}_{N-n-1,r+\ell-1}}}{\Prob{F(N-n-1,p)\in \mathcal{A}_{N-n-1,r-1}}}.\end{equation}
Therefore we also have $\Prob{Z^{N,p}_{n+1} - Z^{N,p}_n = \ell-1 \,\big|\, Z^{N,p}_n=r}\propto q^{N,n,r}_\ell$. Heuristically, from \eqref{eq:ratioAnr}, this quotient, which we will think of as a \emph{weight}, should be approximately
$$\left(1-\alpha\left( \tfrac{r}{N^{1/3}},\Lambda - \tfrac{n}{N^{2/3}}\right) N^{-1/3}\right)^\ell,$$
and so we will be able to approximate $\sum q^{N,n,r}_\ell$ by the probability generating function of $B^{N-n-r,p}$. Indeed, this approximation only breaks down when $r+\ell-1 \ge KN^{1/3}$, that is, outside the range of \eqref{eq:ratioAnr}. Therefore, for any $\delta>0$, for large enough $N$, we have, for all $n\in [0,TN^{2/3}]$, $r\in[1,\frac{K}{2} N^{1/3}]$, and $\ell \le N^{1/4}$.
$$\frac{\Prob{F(N-n-1,p)\in \mathcal{A}_{N-n-1,r+\ell-1}}}{\Prob{F(N-n-1,p)\in \mathcal{A}_{N-n-1,r-1}}}\le \prod_{i=0}^{\ell-1} \left(1 - \left(\alpha\left(\tfrac{r+i-1}{N^{1/3}}, \lambda - \tfrac{n-1}{N^{2/3}} \right)-\delta\right)N^{-1/3} \right).$$
The function $\alpha$ is uniformly continuous. Since the range of $i$ in this product is asymptotically negligible relative to $N^{1/3}$, for large enough $N$, for large enough $N$ we may replace $r+i-1$ by $r$, and $\Lambda=\Lambda(N,p)$ by $\lambda$. That is,
$$\frac{\Prob{F(N-n-1,p)\in \mathcal{A}_{N-n-1,r+\ell-1}}}{\Prob{F(N-n-1,p)\in \mathcal{A}_{N-n-1,r-1}}}\le \left(1 - \left(\alpha\left(\tfrac{r}{N^{1/3}}, \lambda - \tfrac{n}{N^{2/3}} \right)-\delta\right)N^{-1/3} \right)^\ell.$$
An identical argument gives
$$\frac{\Prob{F(N-n-1,p)\in \mathcal{A}_{N-n-1,r+\ell-1}}}{\Prob{F(N-n-1,p)\in \mathcal{A}_{N-n-1,r-1}}}\ge \left(1 - \left(\alpha\left(\tfrac{r}{N^{1/3}}, \lambda - \tfrac{n}{N^{2/3}} \right)+\delta\right)N^{-1/3} \right)^\ell,$$
under the same conditions. From now on, we write $\alpha^N_{n,r}= \alpha\left(\tfrac{r}{N^{1/3}}, \lambda - \tfrac{n}{N^{2/3}} \right)$ for brevity.

\par
Keeping $\delta>0$ fixed, we now address the sums $\sum_{\ell=0}^\infty q_\ell^{N,n,r}$ and $\sum_{\ell=0}^\infty (\ell-1) q_\ell^{N,n,r}$. (Note first that both $q^{N,n,r}_0 $ and $q^{N,n,r}_1\rightarrow 1/e$, so these sums are uniformly bounded below.) For large enough $N$, we have, again for all $n\in [0,TN^{2/3}]$, $r\in[1,\frac{K}{2} N^{1/3}]$,
\begin{align*}
\sum_{\ell=0}^{N-n-r} q_\ell^{N,n,r} &\le \sum_{\ell=0}^{\ce{N^{1/4}}}\Prob{B^{N-n-r,p}=\ell} \left(1-(\alpha^N_{n,r} - \delta)N^{-1/3}\right)^{\ell} + \Prob{B^{N-n-r,p}\ge N^{1/4}}\\
&\le \left[(1-p)+p\left(1-(\alpha^N_{n,r}-\delta) N^{-1/3}\right)\right]^{N-n-r} + \Prob{B^{N-n-r,p}\ge N^{1/4}}
\end{align*}
Now, note that
$$\left[(1-p)+p\left(1-(\alpha^N_{n,r}-\delta) N^{-1/3}\right)\right]^{N-n-r} = \left[1 - (\alpha^N_{n,r}-\delta)N^{-4/3} + O(N^{-5/3}) \right]^{N-n-r},$$
from which we find that
\begin{equation}\label{eq:asympexpcorrection}N^{1/3}\left[1 - \left[(1-p)+p\left(1-(\alpha^N_{n,r}-\delta) N^{-1/3}\right)\right]^{N-n-r}\right] + \left(\alpha^N_{n,r}-\delta \right) \rightarrow 0,\end{equation}
uniformly as $N\rightarrow\infty$. The probability $\Prob{B^{N-n-r,p}\ge N^{1/4}}$ decays exponentially with some positive power of $N$, so we have shown that for large enough $N$,
\begin{equation}\label{eq:upbdsumq}\sum_{\ell=0}^{N-n-r} q_\ell ^{N,n,r} \le 1 - \left(\alpha^N_{n,r}-2\delta\right) N^{-1/3}.\end{equation}
Under the same conditions,
$$\sum_{\ell=0}^{N-n-r} q_\ell ^{N,n,r} \ge 1 - \left(\alpha^N_{n,r}+2\delta\right) N^{-1/3}.$$

Now we consider the sum $\sum \ell q^{N,n,r}_\ell$.
\begin{align}
\sum_{\ell=0}^{N-n-r} \ell q_\ell^{N,n,r} &\le \sum_{\ell=0}^{N-n-r}\ell\Prob{B^{N-n-r,p}=\ell} \left(1-(\alpha^N_{n,r} - \delta)N^{-1/3}\right)^{\ell}\nonumber\\
&\qquad + N\Prob{B^{N-n-r,p}\ge N^{1/4}}\nonumber\\
&\le (N-n-r)p\left(1-(\alpha^N_{n,r}-\delta)N^{-1/3}\right)\nonumber\\
&\qquad\times\left[(1-p)+p\left(1-(\alpha^N_{n,r}-\delta) N^{-1/3}\right)\right]^{N-n-r-1}\label{eq:lqfirstmoment}\\
&\qquad\qquad+ N\Prob{B^{N-n-r,p}\ge N^{1/4}}.\nonumber
\end{align}
We can treat the term $\left[(1-p)+p\left(1-(\alpha^N_{n,r}-\delta) N^{-1/3}\right)\right]^{N-n-r-1}$ as in \eqref{eq:asympexpcorrection}. We also have
$$(N-n-r)p\left(1-(\alpha^N_{n,r}-\delta)N^{-1/3}\right) = 1 + \left(\lambda - \tfrac{n}{N^{2/3}} - (\alpha^N_{n,r}-\delta)\right) N^{-1/3} + O(N^{-2/3}).$$
So, in a similar fashion to \eqref{eq:upbdsumq}, we establish
\begin{equation}\label{eq:upbdsumlq}\sum_{\ell=0}^{N-n-r} \ell q_\ell^{N,n,r} \le 1+ \left(\lambda-2\alpha^N_{n,r}+3\delta- \tfrac{n}{N^{2/3}}\right) N^{-1/3},\end{equation}
and
$$\sum_{\ell=0}^{N-n-r} \ell q_\ell^{N,n,r} \ge 1+ \left(\lambda-2\alpha^N_{n,r}-3\delta- \tfrac{n}{N^{2/3}}\right) N^{-1/3}.$$
Therefore, (where each successive statement holds whenever $(N-n,p,r)\in\bar\Psi^N_0(\lambda^-,\lambda^+,\frac{K}{2},T)$ for large enough $N$)
\begin{align*}
\E{Z^{N,p}_{n+1} - Z^{N,p}_n \,\big|\, Z^{N,p}_n=r} &= \frac{\sum_{\ell=0}^{N-n-r} \ell q_\ell^{N,n,r} - \sum_{\ell=0}^{N-n-r} q_\ell^{N,n,r}}{\sum_{\ell=0}^{N-n-r} q_\ell^{N,n,r}}\\
&\le \frac{\left(\lambda-2\alpha^N_{n,r}+3\delta- \tfrac{n}{N^{2/3}}\right)N^{-1/3} + \left(\alpha^N_{n,r}+2\delta\right) N^{-1/3}}{1+\left(\lambda-\alpha^N_{n,r}-2\delta- \tfrac{n}{N^{2/3}}\right) N^{-1/3}}\\
&\le \left(\lambda - \alpha^N_{n,r} - \tfrac{n}{N^{2/3}} +6\delta \right) N^{-1/3}.
\end{align*}
Similarly
$$\E{Z^{N,p}_{n+1} - Z^{N,p}_n \,\big|\, Z^{N,p}_n=r} \ge \left(\lambda - \alpha^N_{n,r} - \tfrac{n}{N^{2/3}} -6\delta \right) N^{-1/3},$$
and so since $\delta>0$ was arbitrary, after replacing $\frac{K}{2}$ with $K$, we have completed the proof of \eqref{eq:limdrift} in Proposition \ref{limdriftprop}.

%\subsubsection{An Alternative Expression}

%\begin{lemma}Fix $\lambda\in\R$, and set $p=p(n)=\frac{1+\lambda n^{-1/3}}{n}$. Then
%$$\frac{\mathbb P_{n,p}(\mathcal A_{n,r+1})}{\mathbb P_{n,p}(\mathcal A_{n,r})}=1-n^{-1/3}\left[\frac{\int_0^\infty [\frac{1}{b}-\frac{b}{a}]a^{-3/2} p(\lambda'-a)\exp\left(-\tfrac{b^2}{2a}\right)\mathrm{d}a}{\int_0^\infty a^{-3/2} p(\lambda'-a)\exp\left(-\tfrac{b^2}{2a}\right)\mathrm{d}a}\right] + o(n^{-1/3}),$$
%as $n\rightarrow\infty$.

%\begin{proof}
%This isn't important for now. It gives another expression for $\alpha(b,\lambda)$. Will need to confirm that contributions to this sum from off the scale disappear in the limit as before, but argument might not be identical.

%\end{proof}

%\end{lemma}

\subsection{Proof of Proposition \ref{limdriftprop}:
variance, jumps and reflection}\label{varjumpproof}

\subsubsection*{Variance of increments}
We can show \eqref{eq:limvar} using the estimates from Section \ref{Etiltsection}. Recall the definition of $q^{N,n,r}_\ell$ from \eqref{eq:defnqNmr}. As in \eqref{eq:lqfirstmoment}, we have
\begin{align*}
\sum_{\ell=0}^{N-n-r} \ell(\ell-1)q_\ell^{N,n,r} &\le (N-n-r)(N-n-r-1) p^2 \left(1-(\alpha^N_{n,r}-\delta)N^{-1/3}\right)^2\\
&\qquad\times\left[(1-p)+p\left(1-(\alpha^N_{n,r}-\delta) N^{-1/3}\right)\right]^{N-n-r-2}\\
&\qquad\qquad+ N\Prob{B^{N-n-r,p}\ge N^{1/4}}.
\end{align*}

Again, we use \eqref{eq:asympexpcorrection} and similarly to \eqref{eq:upbdsumlq}, we have
\begin{multline*}
1 + \left(2\lambda - 3\alpha^N_{n,r}-4\delta - \tfrac{2n}{N^{2/3}}\right)N^{-1/3}
\\
\le\sum_{\ell=0}^{N-n-r} \ell(\ell-1) q_\ell^{N,n,r} \le 1 + \left(2\lambda - 3\alpha^N_{n,r}+4\delta - \tfrac{2n}{N^{2/3}}\right)N^{-1/3}.
\end{multline*}
In particular, we obtain
$$\sum_{\ell=0}^{N-n-r} (\ell-1)^2 q_\ell^{N,n,r} = \sum_{\ell=0}^{N-n-r} \ell(\ell-1) q_\ell^{N,n,r} - \sum_{\ell=0}^{N-n-r} \ell  q_\ell^{N,n,r}+\sum_{\ell=0}^{N-n-r}  q_\ell^{N,n,r}\rightarrow 1,$$
uniformly, which is exactly \eqref{eq:limvar}.

\subsubsection*{Jumps in the limit}
For any $n\in[N]$,
\begin{align*}
\Prob{|Z^{N,p}_{n+1} - Z^{N,p}_n| > \delta N^{1/3}}&\le \;\Prob{\exists v\in[N], \mathrm{deg}_{F(N,p)}(v)>\delta N^{1/3}}\\
&\stackrel{\mathclap{\text{Prop \ref{barGstochdom}}}}{\le} \;\Prob{\exists v\in[N], \mathrm{deg}_{G(N,p)}(v)>\delta N^{1/3}}\\
&\le \;N \Prob{\mathrm{deg}_{G(N,p)}(1) > \delta N^{1/3}}.
\end{align*}
But $\mathrm{deg}_{G(N,p)}(1) \sim \Bin{N-1}{p}$, and so for any $\delta>0$, this final term vanishes exponentially fast. So \eqref{eq:limnojumps} follows.

\subsubsection*{Speed at the boundary}
Finally, we check that the discrete processes $(Z^{N,p})$ do not get stuck at zero. By Lemma \ref{incrdescription}, we have
$$\frac{\Prob{Z^{N,p}_{n+1}=1 \,\big|\, Z^{N,p}_n=0}}{\Prob{Z^{N,p}_{n+1}=0\,\big|\, Z^{N,p}_n=0}} = \frac{\Prob{B^{N-n-1,p}=1}}{\Prob{B^{N-n-1,p}=0}}=\frac{(N-n-1)p}{1-p}.$$
Therefore
$$\liminf_{N\rightarrow\infty} \inf_{n\in[0,TN^{2/3}]} \frac{\Prob{Z^{N,p}_{n+1}=1 \,\big|\, Z^{N,p}_n=0}}{\Prob{Z^{N,p}_{n+1}=0\,\big|\, Z^{N,p}_n=0}}\ge 1,$$
and so
$$\liminf_{N\rightarrow\infty} \inf_{n\in[0,TN^{2/3}]} \E{\left[Z^{N,p}_{n+1}\right]^2 \,\big|\, Z^{N,p}_n=0}\ge \frac12,$$
as required for \eqref{eq:limnonsticky}.

This completes the proof of Proposition \ref{limdriftprop}
(subject to the proofs of Lemmas \ref{acyclicboundlemma},
\ref{Anrklemma}, and \ref{geomratiolemma} in Section 
\ref{section:combinatorial}).

\section{Excursions and component sizes}\label{cptsizes}
In this section, we will prove that Theorem \ref{FNptheorem} follows from Theorem \ref{explconvtheorem}.
\par
As in Aldous \cite{Aldous97}, we must check that excursions of the limiting reflected SDE are matched by excursions of the discrete exploration processes. In particular, it must happen with vanishing probability that a zero of the limiting process $Z^\lambda$ appears only as the limit of small \emph{positive} local minima of the discrete processes $Z^{N,p}$. In addition, we must show that there are with high probability no large discrete components which appear late enough in the exploration that they are not represented in the limit. Several stages of the argument will be based on a comparison of $F(N,p)$ and the original model $G(N,p)$, for which some of the results are easier, or known.

\subsection{Large components are explored early}

Theorem \ref{explconvtheorem} establishes convergence of the exploration processes on compact time intervals. To use this to study the sizes of the largest components in $F(N,p)$, we need to ensure that these largest components appear early in the exploration process. We establish this in the following series of lemmas.

\par
%The following lemma follows from Janson's corresponding result on $G(N,p)$ {\bf [CITE]} and the coupling in Lemma \ref{barGstochdom}. {\bf Also follows from bound on $\E{|C(v)|}$ plus Markov?}

\begin{lemma}Fix $\lambda^+\in\R$. Then
\begin{equation}\label{eq:cptsizeupbd}\lim_{\gamma\rightarrow\infty}\limsup_{N\rightarrow\infty}\sup_{\Lambda(N,p)\le \lambda^+}\Prob{ C_1(F(N,p)) \ge \gamma N^{2/3}} = 0.\end{equation}
\begin{proof}
We have
\begin{align}
\E{|C_{F(N,p)}(v)|}&=\E{\frac{1}{N}\sum_{i\ge 1}C_i(F(N,p))}\label{eq:decomposeECv}\\
&\ge \frac{1}{N}\E{|C_1(F(N,p))|^2}\ge \frac{(\gamma N^{2/3})^2}{N} \Prob{C_1(F(N,p)) \ge \gamma N^{2/3}}.\nonumber
\end{align}
Result \eqref{eq:cptsizeupbd} then follows by using Lemma \ref{JansonSpencerprop} and the coupling of Lemma \ref{barGstochdom} to control the first expectation in \eqref{eq:decomposeECv}.
\end{proof}
\end{lemma}

The following lemma shows that critical components will with high probability include a vertex with label $O(N^{1/3})$.

\begin{lemma}\label{bigcptsmalllabel}Fix $\epsilon>0$, and $\lambda^+\in\R$. Then
\begin{equation}\label{eq:bigcptsmalllabel} \lim_{\Gamma\rightarrow\infty}\limsup_{N\rightarrow\infty}\sup_{\Lambda\le \lambda^+} \mathbb{P}\left(\exists \text{ cpt }C\text{ in } F\left(N,p\right)\, :\,|C|\ge \epsilon N^{2/3}\right.\hspace{0.5in}\end{equation}
$$\hspace{1.8in}\left. \text{and }C\cap\left\{1,\ldots,\fl{\Gamma N^{1/3}}\right\}=\varnothing\right)= 0.$$
\begin{proof}
Applying Markov's inequality to \eqref{eq:barECvlimit}, and summing over all vertices,
$$\limsup_{N\rightarrow\infty}\sup_{\Lambda(N,p)\le \lambda^+}N^{-2/3}\E{\left| \left\{v\in[N]\,:\, \left|C_{F(N,p)}(v)\right|\ge \epsilon N^{2/3}    \right\} \right|}\le \frac{\Theta^{\lambda^+}}{\epsilon}.$$
Therefore,
$$\limsup_{N\rightarrow\infty}\sup_{\Lambda(N,p)\le \lambda^+}\E{\#\text{cpts }C\text{ in }F\left(N,p\right)\text{ s.t. }|C|\ge \epsilon N^{2/3}}\le \frac{\Theta^{\lambda^+}}{\epsilon^2}.$$
Then, since the labelling is independent of the component sizes in $F(N,p)$,
\begin{align*}
&\limsup_{N\rightarrow\infty}\sup_{\Lambda(N,p)\le \lambda^+}\mathbb{E}\left[\#\text{cpts }C\text{ in }F\left(N,p\right)\text{ s.t. }|C|\ge \epsilon N^{2/3}\right.\\
&\hspace{2.5in}\left.\text{and }C\cap\left\{1,\ldots,\fl{\Gamma N^{1/3}}\right\}=\varnothing\right]\\
&\qquad\le \frac{\Theta^{\lambda^+}}{\epsilon^2}\times \lim_{N\rightarrow\infty}\frac{\binom{N-\fl{\Gamma N^{1/3}}}{\fl{\epsilon N^{2/3}}}}{\binom{N}{\fl{\epsilon N^{2/3}}}}\\
&\qquad\le \frac{\Theta^{\lambda^+}}{\epsilon^2} \exp(-\Gamma \epsilon),
\end{align*}
where the final limit can be evaluated using Stirling's approximation. \eqref{eq:bigcptsmalllabel} follows.
\end{proof}
\end{lemma}

We now use the previous result to show that the largest components will typically appear near the start of the exploration process. This will be important later, since if large critical components appear arbitrarily late in the exploration process, then they cannot be treated via convergence on compact intervals.
\begin{lemma}Fix $\epsilon>0$ and $\lambda^+\in\R$ as before. Then
\begin{equation}\label{eq:bigcptlateinexp} \lim_{T\rightarrow\infty}\limsup_{N\rightarrow\infty}\sup_{\Lambda(N,p)\le \lambda^+} \mathbb{P}\left(\exists \text{ cpt }C\text{ in } F\left(N,p\right)\, :\,|C|\ge \epsilon N^{2/3}\right.\hspace{0.5in}\end{equation}
$$\hspace{1.8in}\left. \text{and }C\cap\left\{v_1,\ldots,v_{\fl{TN^{2/3}}}\right\}=\varnothing\right)= 0,$$
where $(v_1,v_2,\ldots,v_N)$ is the exploration process of $F(N,p)$.

\begin{proof}
Fix $\Gamma>0$, and let $C_{F(N,p)}(k)$ be the component of vertex $k$ in $F(N,p)$. We define the events
$$A^{\Gamma,T}(F(N,p)):=\{ |C_{F(N,p)}(1)|+\ldots+| C_{F(N,p)}(\fl{\Gamma N^{1/3}})| > TN^{2/3}\},$$
$$B^{\epsilon,\Gamma}(F(N,p)):= \left\{ \exists \text{ cpt }C \text{ in }F(N,p)\, :\, |C|\ge \epsilon N^{2/3},\, C\cap\left\{1,\ldots,\fl{\Gamma N^{1/3}}\right\}=\varnothing \right\},$$
as in Lemma \ref{bigcptsmalllabel}. Then, by Markov's inequality,
$$\Prob{ A^{\Gamma, T}(F(N,p))} \le \frac{\Gamma N^{1/3}\E{|C_{F(N,p)}(1)|}}{TN^{2/3}},$$
So by Lemma \ref{JansonSpencerprop}
\begin{equation}\label{eq:probAGammaT}\limsup_{N\rightarrow\infty}\sup_{\Lambda(N,p)\le \lambda^+}\Prob{ A^{\Gamma, T}(F(N,p))} \le  \frac{\Theta^{\lambda^+}\Gamma}{T}.\end{equation}
Whenever $F(N,p)$ contains a component of size at least $\epsilon N^{2/3}$ which is not exhausted during the first $TN^{2/3}$ steps of the exploration process, at least one of $A^{\Gamma,T}(F(N,p))$ and $B^{\epsilon,\Gamma}(F(N,p))$ must hold. So take $\Gamma=\sqrt{T}$, then let $T\rightarrow\infty$. By \eqref{eq:probAGammaT} and Lemma \ref{bigcptsmalllabel}, the result follows.
\end{proof}

\end{lemma}

\subsection{Components and excursions up to time $T$ - notation and goal}
Throughout this section, we fix $\lambda\in\R$ and work with a sequence $p(N)$ for which $\Lambda(N,p):=N^{1/3}[Np-1]\rightarrow \lambda$ as before. We will mostly suppress notational dependence on $p$ and $\lambda$.

\par
First, we establish the notation we will use to describe the sequence of rescaled component sizes in $F(N,p)$. Fix $T>0$, then:
\begin{itemize}
\item Let $\mathcal{C}^{N}:=(C_1^{N},C_2^{N},\ldots)$ be the sequence of sizes of components of $F(N,p)$, in non-increasing order.
\item Analogously, let $\C^{N,T}:=(C_1^{N,T},C_2^{N,T},\ldots)$ be the sequence of sizes of components of $F(N,p)$ which have non-empty intersection with $\{v_1,\ldots,v_{\fl{TN^{2/3}}}\}$, an initial segment of the breadth-first ordering introduced in Section \ref{subsec:intro-exploration}. That is, least one vertex has been seen by step $\fl{TN^{2/3}}$ of the exploration process. Again, we assume the sequence is ordered such that $C^{N,T}_1\ge C^{N,T}_2\ge\ldots$
\end{itemize}

We first show that for any $T<\infty$ the excursion lengths in the exploration processes on the interval $[0,T]$ appear correctly in the limit.

\par
In everything that follows, we work on the probability space $(\Omega,\F,\mathbb{P})$ whose existence is guaranteed by the Skorohod representation theorem, where $\tilde Z^{N,p} \stackrel{\mathbb{P}-\text{a.s.}}\rightarrow Z^\lambda$ with respect to the topology of uniform convergence on compact intervals.

\par
In a mild abuse of notation, let $C^T_1\ge C^T_2\ge\ldots$ be the lengths of excursions of $Z^\lambda$ above zero 
which have non-empty intersection with $[0,T]$, in non-increasing order. 
Set $\mathcal{C}^T:=(C^T_1,C^T_2,\ldots)$. We will prove the following convergence result 
for the components seen within the first $TN^{2/3}$ steps of the exploration process.

\begin{prop}\label{kcptsconv}Fix $T>0$ and $k\ge 1$. Then as $N\rightarrow\infty$,
\begin{equation}\label{eq:kcptsconv} N^{-2/3}(C_1^{N,T},C_2^{N,T},\ldots,C_k^{N,T}) \quad \stackrel{d}\rightarrow \quad (C_1^T,C_2^T,\ldots,C_k^T).  \end{equation} 
\end{prop}

The concern is that the reflected exploration process might regularly approach zero without actually hitting zero, and thus starting a new component. To show that this effect does not appear in the limit, we use the fact that the components of $F(N,p)$ have the structure of uniform random trees. Then we can approximate the exploration process within a component by a Brownian excursion, and show that the probability of zeros in the limit which do not correspond to the start or end of a component is small.

\begin{definition}Given two sequences $a=(a_1,\ldots,a_k)$, $b=(b_1,\ldots,b_k)$, let $a^\downarrow,b^\downarrow$ denote the sequences rearranged into non-increasing order. Then, we say $a\succeq b$ or $a$ \emph{weakly majorises} $b$ if for every $\ell\le k$,
$$\sum_{i=1}^\ell a^\downarrow_i\ge \sum_{i=1}^\ell b^\downarrow_i.$$
It is easy to check that this gives a pre-order on $(\R\cup\{\infty\})^k$, and a partial order on non-increasing sequences finer than the standard ordering.
\end{definition}

We will prove Proposition \ref{kcptsconv} by stochastically sandwiching $\C^T$ between any weak limit of $\C^{N,T}$, and any weak limit of a related sequence of lengths $\C^{N,T,\delta}$ associated with $\tilde Z^N$, which will be defined shortly. This stochastic ordering will be with respect to weak majorisation. The two directions of this sandwiching argument occupy the next two sections. Finally, we show that for small enough $\delta$, these outer distributions are close in the sense of the L\'evy--Prohorov metric.

\subsection{Limits of component sizes stochastically majorise excursion lengths}

We show that limit points of $\C^{N,T}$ majorise $\C^T$, $\mathbb{P}$-almost surely.

\par
For any reference time $s\in[0,T]$, we define
$$\ell(s):=\sup\{t\le s : Z(t)=0\},\quad  \ell^N(s):=\sup\{t\le s : \tilde Z^N(t)=0\},$$
$$r(s):=\inf\{t\in [s,\infty) : Z(t)=0\},\quad r^N(s):=\beta(T)\wedge \inf\{t\in [s,r(T)] : \tilde Z^N(t)=0\},$$
so that $r(s)-\ell(s)$ is the width of the excursion of $Z$ around time $s$. It will be convenient to avoid values of $s$ where $\ell^N$ and $r^N$ are non-constant, so we define
$$\bar {\mathbb{Q}}:= \bigcup_{N\in \N} N^{-2/3}\Z.$$
We also define the event
$$\Psi^T:=\left\{\tilde Z^N\rightarrow Z\text{ uniformly on }[0,r(T)],\, Z\text{ continuous on }[0,r(T)]\right\}.$$
Since $r(T)<\infty$ almost surely, and $\tilde Z^N\rightarrow Z$ uniformly on compact intervals, we have $\Prob{\Psi^T}=1$. It follows easily that on $\Psi^T$,
\begin{equation}\label{eq:baralpha}\limsup_{N\rightarrow\infty}\ell^N(s)\le \ell(s),\quad  \liminf_{N\rightarrow\infty} r^N(s) \ge r(s), \quad\forall s\in[0,T].\end{equation}

Now, on $\Psi^T$, given $Z$, choose $s_1,\ldots,s_k\in[0,T]\backslash \bar{\mathbb{Q}}$ such that each $s_i$ lies in the $i$th longest excursion of $Z$, which has non-empty intersection with $[0,T]$. That is, $r(s_i)-\ell(s_i)=C^T_i$. Now consider any limit point
\begin{equation}\label{eq:limpointofcpts}(\bar\ell(s_1),\ldots,\bar\ell(s_k),\bar r(s_1),\ldots, \bar r(s_k), \bar C^T_1,\ldots,\bar C^T_k),\end{equation}
of $(\ell^N(s_1),\ldots,\ell^N(s_k), r^N(s_1),\ldots, r^N(s_k), C^{N,T}_1,\ldots, C^{N,T}_k)$, as $N\rightarrow\infty$, where we allow $\bar C_1^T$ and at most one of the $\bar r(s_i)$ to be $\infty$. By compactness, we can be sure that there are such limit points. To avoid introducing extra notation, we will assume that \eqref{eq:limpointofcpts} is a true limit, rather than a subsequential limit.

By \eqref{eq:baralpha}, for any $m\le k$,
$$\bigcup_{i=1}^m [\bar \ell(s_i),\bar r(s_i) ] \supseteq \bigcup_{i=1}^m [\ell(s_i),r(s_i)],$$
where the sets in the union on the right-hand side have disjoint interiors. By construction of $\ell^N(s_i), r^N(s_j)$, any pair of intervals $[\ell^N(s_i),r^N(s_i)]$ and $[\ell^N(s_j),r^N(s_j)]$ are either equal or disjoint. Therefore the intervals in the union on the left-hand side are either equal or have disjoint interiors. So for any limit point \eqref{eq:limpointofcpts}, let $\Gamma_m\subseteq [m]$ be some set of indices such that
$$[\bar\ell(s_i),\bar r(s_i)]\ne [\bar \ell(s_j),\bar r(s_j)],\quad \forall i\ne j\in\Gamma_m,\quad\text{and}\quad \bigcup_{i\in\Gamma_m} [\bar \ell(s_i),\bar r(s_i) ] \supseteq \bigcup_{i=1}^m [\ell(s_i),r(s_i)].$$
Furthermore, we may demand $\Gamma_1\subseteq \Gamma_2\subseteq \ldots\subseteq \Gamma_k$. Thus
$$\sum_{i\in\Gamma_m} (\bar r(s_i)-\bar \ell(s_i) ) \ge \sum_{i=1}^m (r(s_i)-\ell(s_i)).$$
That is,
\begin{equation}\label{eq:Gammalmaj}\left(\bar r(s_1)-\bar \ell(s_1),\ldots,\bar r\left(s_{|\Gamma_k|}\right) - \bar \ell\left(s_{|\Gamma_k|}\right),0,\ldots,0\right)\succeq \left(r(s_1)-\ell(s_1),\ldots,r(s_k)-\ell(s_k)\right).\end{equation}

For any $N$, and any $s\in[0,T]\backslash\bar{\mathbb{Q}}$, the interval $[\ell^N(s),r^N(s)]$ is associated via the reflected exploration process with exactly one component of $F(N,p)$. The size of this component is at least $(r^N(s)-\ell^N(s))N^{2/3}$.
\begin{note}The two cases where the size of the component is not equal to $(r^N(s)-\ell^N(s))N^{2/3}$ are: 1) when $r^N(s)=r(T)$; 2) when $\tilde Z^N(s)=0$. In the latter case, since we have excluded the possibility $s\in N^{-2/3}\Z$, it must hold that $\tilde Z^N$ is locally constant and equal to zero around $s$, so the component has size 1.\end{note}
For large enough $N$, the intervals $\{[\bar \ell^N(s_i), \bar r^N(s_i)]: i\in \Gamma_k\}$ are disjoint, and so
$$N^{-2/3}(C_1^{N,T},\ldots, C_k^{N,T}) \succeq \left( r^N(s_1)- \ell^N(s_1),\ldots, r^N\left(s_{|\Gamma_k|}\right) -  \ell^N\left(s_{|\Gamma_k|}\right),0,\ldots,0\right).$$
Since majorisation is preserved under limits (as the relation is a finite union of closed sets in $\R^k\times \R^k$), we obtain
$$(\bar C_1^T,\ldots, \bar C_k^T) \succeq \left(\bar r(s_1)-\bar \ell(s_1),\ldots,\bar r\left(s_{|\Gamma_k|}\right) - \bar \ell\left(s_{|\Gamma_k|}\right),0,\ldots,0\right).$$

So, combining with \eqref{eq:Gammalmaj}, we obtain
\begin{equation}\label{eq:barCstocdom}(\bar C_1^T,\ldots, \bar C_k^T)\succeq (C_1^T,\ldots,C_k^T),\end{equation}
which holds for every limit point $(\bar C_1^T,\ldots,\bar C_k^T)$ of $N^{-2/3}(C^{N,T}_1,\ldots,C^{N,T}_k)$ on the event $\Psi^T$ and so, in particular, $\mathbb{P}$-almost surely.

\subsection{Stochastic sandwiching via excursions above $\delta$}
We now bound $\C^T$ below stochastically (again in the sense of weak majorisation).

\par
Fix some $\delta>0$. For any realisation of the path $\tilde Z^N$, the set $\mathcal D^{N,\delta,T}:=\{s\in[0,T]:\tilde Z^N(s)>\delta\}$ is a finite union of left-closed, right-open intervals. Let $N^{-2/3}(C_1^{N,\delta,T}\ge\ldots\ge C_k^{N,\delta,T})$ be the sequence of the $k$ largest lengths of those intervals which are contained within the support of some excursion of $\tilde Z^N$ (above zero) which has non-empty intersection with $[0,T]$. As before, augment with zeros if necessary. (Note that the $N^{-2/3}$ ensures that $C^{N,\delta,T}_1$ has the same scaling as $\mathcal{C}^{N,T}$.) Certainly, for any $\delta$, $(C^{N,T}_1,\ldots,C^{N,T}_k)\succeq (C^{N,\delta,T}_1,\ldots,C_k^{N,\delta,T})$ for each trajectory of $\tilde Z^N$. We will show that $\C^{T}$ majorises limit points of $N^{-2/3}(C^{N,\delta,T}_1,\ldots,C^{N,\delta,T}_k)$, again $\mathbb{P}$-almost surely.

\par
Again, we work on the event $\Psi^T$. Then, consider $\mathcal{D}^T:=\{s\in[0,T]: Z(s)>0\}$, the collection of open intervals where the limit process $Z$ is positive. On $\Psi^T$, for large enough $N$, we have $\tilde Z^N(s)\le \delta/2$ whenever $Z(s)=0$, and so $\mathcal D^{N,\delta,T}\subseteq \mathcal{D}^T$. Therefore the sequence of all interval lengths in $\mathcal D^{N,\delta,T}$ in non-increasing order is majorised by the corresponding ordered sequence of interval lengths in $\mathcal{D}^T$. So in particular
\[
(C_1^T,\ldots,C_k^T)\succeq N^{-2/3}(C_1^{N,\delta,T},\ldots,C_k^{N,\delta,T}),
\]
for large enough $N$, and hence on $\Psi^T$ any limit point 
$(\bar C^{\delta,T}_1,\ldots,\bar C^{\delta,T}_k)$ 
of\\
$N^{-2/3}(C^{N,\delta, T}_1,\ldots,C^{N,\delta,T}_k)$ satisfies
$$(C_1^T,\ldots,C_k^T)\succeq (\bar C_1^{\delta,T},\ldots,\bar C_k^{\delta,T}).$$

\par
By \eqref{eq:cptsizeupbd}, the collection 
$N^{-2/3}(C^{N,T}_1,\ldots,C^{N,T}_k,C^{N,\delta,T}_1,\ldots,C^{N,\delta,T}_k)_{N\ge 1}$ 
is tight in $\R^k\times \R^k$. 
Let $(\bar C^T_1,\ldots,\bar C^T_k, \bar C^{\delta,T}_1,\ldots,\bar C^{\delta,T}_k)$ 
be any joint weak limit of 
$N^{-2/3}(C^{N,T}_1,\ldots,C^{N,T}_k,$
$C^{N,\delta,T}_1,\ldots,C^{N,\delta,T}_k)$. Since $\Prob{\Psi^T}=1$, by combining with \eqref{eq:barCstocdom} we have shown that
\begin{equation}\label{eq:Csstochdom} (\bar C^T_1,\ldots,\bar C^T_k) \succeq_{st} (C^T_1,\ldots,C^T_k) \succeq_{st} (\bar C^{\delta,T}_1,\ldots,\bar C^{\delta,T}_k).\end{equation}

\subsection{Comparing $C^{N,T}$ and $C^{N,\delta,T}$ via uniform trees}

We will now show for small $\delta$, any weak limits $(\bar C^T_1,\ldots,\bar C^T_k)$ and $(\bar C^{\delta,T}_1,\ldots,\bar C^{\delta,T}_k)$ are themselves close in distribution in the sense of the L\'evy--Prohorov metric on $\R^k$. To do this, we have to bound above the probability that the exploration process drops below height $\delta N^{1/3}$ in the middle of an excursion above zero of width $\Theta(N^{2/3})$. The components of $F(N,p)$ are, conditional on their sizes, uniform trees. In \cite{AldousCRT2}, Aldous explains how to view the uniform tree as an example of a Galton--Watson tree, here with Poisson offspring distribution, conditioned on its total progreny. From this, large excursions of $\tilde Z^N$ are well-approximated by Brownian excursions. We then can then bound the probability that $\tilde Z^N$ hits $\delta$ without hitting zero using standard estimates.

\par
Let $\mathcal{T}_K$ be a uniform choice from the $K^{K-2}$ unordered trees with vertex labels given by $[K]$. Then, let $1=S^{\mathcal T_K}_0,S^{\mathcal T_K}_1,\ldots,S^{\mathcal T_K}_K=0$, be the corresponding breadth-first exploration process. The appropriate rescaling to consider is then $\tilde S^{\mathcal T_K}(s):=\frac{1}{\sqrt{K}} S^{\mathcal T_K}_{\fl{Ks}}$, for $s\in[0,1]$. From the description of $\mathcal{T}_K$ as a conditioned Galton--Watson process, we follow Le Gall (see \cite{LeGall05} Lemma 1.16) in using Kaigh's scaling limit result for conditioned random walks \cite{Kaigh} to obtain
\begin{equation}\label{eq:USTconverges}\left(\tilde S^{\mathcal T_K}(s), s\in[0,1]\right)\; \stackrel{d}\rightarrow \; \left(B^{\mathrm{ex}}(s),s\in[0,1]\right),\end{equation}
where $B^{\mathrm{ex}}$ is a standard normalised Brownian excursion on $[0,1]$, and convergence is in the uniform topology.

\par
We say the event $\chi^{N,T}(\delta,\epsilon,\gamma)$ holds if $\exists M,K\in\Z_{\ge 0}$ with $\tfrac{K}{N^{2/3}}\ge \gamma$, and $\tfrac{M}{N^{2/3}}\le T$, such that $\{v_{M},\ldots,v_{M+K-1}\}$ is a component of $F(N,p)$, and 
\begin{equation}\label{eq:exchaslowpoint}\exists n \in[\epsilon K,(1-\epsilon) K]\text{ s.t. } \tilde Z^N\left(\tfrac{M+n}{N^{2/3}} \right) \le \delta.\end{equation}

That is, $F(N,p)$ has a component of size at least $\gamma N^{2/3}$ which is seen, at least partially, in the exploration process before time $TN^{2/3}$, and for which the exploration process takes a small value in the macroscopic interior of the interval defining the component. Now, given any $M,K$, and conditional on the vertices $\{v_{M},\ldots,v_{M+K-1}\}$, and the statement that they form a component, the structure of this component is a uniform tree. That is,
$$(Z^N_M,\ldots,Z^N_{M+K-1})\equidist (S^{\mathcal T_K}_1,\ldots, S^{\mathcal T_K}_K ).$$

Therefore the following processes on $s\in[0,1]$ can be identified in distribution:
$$\left(\tilde Z^N\left(\tfrac{M+sK}{N^{2/3}}\right)\right) = \left(N^{-1/3} Z^N_{\fl{M+sK}}\right) \equidist \left(N^{-1/3} S^{\mathcal T _K}_{\fl{sK}}\right) = \left(\tfrac{K^{1/2}}{N^{1/3}} \tilde S^{\mathcal T_K}(s)\right).$$

Therefore, for every $M,K$, conditional on any choice of vertices $\{v_M,\ldots,v_{M+K-1}\}$, the probability that \eqref{eq:exchaslowpoint} holds is equal to the probability that
\begin{equation}\label{eq:USThaslowpoint}\inf_{s\in[\epsilon,1-\epsilon]}\tilde S^{\mathcal T_K}(s)\le \tfrac{N^{1/3}}{K^{1/2}}\delta.\end{equation}
By assumption $\frac{N^{1/3}}{K^{1/2}}\delta\le \gamma^{-1/2}\delta$, and by \eqref{eq:USTconverges}, and the Portmanteau lemma,
\begin{align*}
\limsup_{K\rightarrow\infty}\Prob{\inf_{s\in[\epsilon,1-\epsilon]} \tilde S^{\mathcal T_K}(s)\le \gamma^{-1/2}\delta}&\le \limsup_{K\rightarrow\infty}\Prob{\inf_{s\in[\epsilon,1-\epsilon]} \tilde S^{\mathcal T_K}(s)< 2\gamma^{-1/2}\delta}\\
&\le\Prob{\min_{s\in[\epsilon,1-\epsilon]} B^{\mathrm{ex}}(s)< 2\gamma^{-1/2}\delta}.
\end{align*}

Therefore, we obtain
\begin{align}
\limsup_{N\rightarrow\infty} \Prob{\chi^{N,T}(\delta,\epsilon,\gamma)}&\le \E{\#\text{ cpts size }\ge \gamma N^{2/3} \text{ seen before }TN^{2/3}\text{ in } Z^N}\nonumber\\
&\quad \times \Prob{\min_{s\in[\epsilon,1-\epsilon]}B(s) <2\gamma^{-1/2}\delta}\nonumber\\
\limsup_{N\rightarrow\infty} \Prob{\chi^{N,T}(\delta,\epsilon,\gamma)}&\le \left(\tfrac{T}{\gamma}+1\right)\,\Prob{\min_{s\in[\epsilon,1-\epsilon]}B(s) <2\gamma^{-1/2}\delta}.\label{eq:Pboundchi}
\end{align}
Given $\epsilon,\gamma$, we can choose $\delta>0$ so that the RHS of \eqref{eq:Pboundchi} is arbitrarily small. Now, fix some $\gamma>2\epsilon$, and consider the event $\chi^{N,T}(\delta,\frac{\epsilon}{2\gamma},\epsilon)$. Then, when $\chi^{N,T}(\delta,\frac{\epsilon}{2\gamma},\epsilon)$, does not hold, for every component with size $K\ge \epsilon N^{2/3}$, there is a unique excursion of $Z^N$ above $\delta N^{1/3}$ of length at least $K(1-\frac{\epsilon}{\gamma})$. We call such an excursion above $\delta N^{1/3}$ a \emph{principal excursion}. If we also have $C^N_1\le \gamma N^{2/3}$, then the length of any principal excursion is at least $K-\epsilon N^{2/3}$. Thus, any other excursion above $\delta N^{1/3}$ within the component of size $K$, has length at most $\epsilon N^{2/3}$.

\par
So, consider any $i\le k$ such that $C^{N,T}_1\ge \ldots C^{N,T}_{i}\ge \epsilon N^{2/3}$. Then, on $\chi^{N,T}(\delta,\frac{\epsilon}{2\gamma},\epsilon)^c$ and $\{C^N_1\le \gamma N^{2/3}\}$, at most $i-1$ elements of $(C^{N,\delta,T}_1,\ldots,C^{N,\delta,T}_k)$ can be larger than $C_i^{N,T}$. These are the principal excursions obtained from each of $C^{N,T}_1,\ldots,C^{N,T}_{i-1}$. No other excursions above $\delta N^{2/3}$ obtained from $C^{N,T}_1,\ldots,C^{N,T}_{i-1}$ are relevant, since they have lengths at most $\epsilon N^{2/3}$. However, these principal excursions from $C^{N,T}_1,\ldots,C^{N,T}_{i}$ all have length at least $C_i^{N,T}(1-\frac{\epsilon}\gamma)$. Thus we obtain
\begin{equation}\label{eq:CivCdeltai}C^{N,T}_i\ge C_i^{N,\delta,T}\ge C_i^{N,T}(1-\tfrac{\epsilon}{\gamma})\ge C_i^{N,T}-\epsilon N^{2/3}.\end{equation}

And so
\begin{multline*}
\limsup_{N\rightarrow\infty} \Prob{\max_{i\in[k]}\left|C^{N,T}_i - C^{N,\delta,T}_i\right|>\epsilon N^{2/3}}
\\
\le \limsup_{N\rightarrow\infty}\Prob{C^N_1>\gamma N^{2/3}} + \limsup_{N\rightarrow\infty} \Prob{\chi^{N,T}(\delta,\tfrac{\epsilon}{2\gamma},\epsilon)}.
\end{multline*}
For fixed $\epsilon>0$, letting $\gamma\rightarrow\infty$ we can make the first term on the RHS small, and then by letting $\delta\downarrow 0$ we can make the second term small. In particular, we can demand
\begin{equation}\label{eq:LevyboundCnTd}\limsup_{N\rightarrow\infty} \Prob{\max_{i\in[k]}\left|C^{N,T}_i - C^{N,\delta,T}_i\right|>\epsilon N^{2/3}}\le \epsilon.\end{equation}

Now, recall $(\bar C^T_1,\ldots,\bar C^T_k,\bar C^{\delta,T}_1,\ldots,\bar C^{\delta,T}_k)$ is some joint weak limit of\\ $N^{-2/3}(C^{N,T}_1,\ldots,C^{N,T}_k,C^{N,\delta,T}_1,\ldots,C^{N,\delta,T}_k)$. Let $\pi$ be the usual L\'evy--Prohorov metric for probability measures on $\R^k$, with respect to the $\ell_\infty$ norm on $\R^k$. From \eqref{eq:Csstochdom} and \eqref{eq:LevyboundCnTd}, we have for each $\epsilon>0$,
$$\pi(\mathcal{L}(\bar C^T_1,\ldots,\bar C^T_k), \mathcal{L}(\bar C^{\delta,T}_1,\ldots,C^{\delta,T}_k))\le \epsilon,$$
$$(\bar C^{\delta,T}_1,\ldots, \bar C^{\delta,T}_k)\preceq (C^T_1,\ldots,C^T_k)\preceq (\bar C^{T}_1,\ldots,\bar C^T_k).$$
From this, it is easy to see that $\pi(\mathcal{L}(C^T_1,\ldots,C^T_k),\mathcal{L}(\bar C^T_1,\ldots,\bar C^T_k))\le k\epsilon$. Since $\epsilon>0$ is arbitrary, we find $(\bar C^T_1,\ldots,\bar C^T_k))\equidist (C^T_1,\ldots,C^T_k)$, and thus the required convergence in distibution \eqref{eq:kcptsconv} follows, completing the proof of Proposition \ref{kcptsconv}.
%\footnote{If $a\preceq b\preceq c$ and $|c_i-a_i|\le \epsilon$, then for each $\ell\le k$, we have $c_1+\ldots+c_\ell - \ell\epsilon \le b_1+\ldots+b_\ell \le c_1+\ldots+c_\ell$. Comparing this relation for $\ell-1$ and $\ell$, $c_\ell-\ell\epsilon \le b_\ell\le c_\ell + (\ell-1)\epsilon$.}

\subsection{Proof of Theorem \ref{FNptheorem}}

\subsubsection*{Convergence in the product topology}
In both the discrete exploration processes and the limiting SDEs, we would expect the $k$ largest components/excursions to appear early. From \eqref{eq:bigcptlateinexp}, 
$$\limsup_{T\rightarrow\infty}\limsup_{N\rightarrow\infty} \Prob{\max_{i\in[k]}\left|C^{N,T}_i - C^N_i\right|\le\epsilon N^{2/3}} = 1.$$

Recall again that $\lambda$ is fixed. By comparing the drifts, we can couple $Z^\lambda$ and $B^\lambda$, as defined in Proposition \ref{prop:AldousGNp} and Proposition \ref{Zlambdaexists}, such that $Z^\lambda(t)\le B^\lambda(t)$ for all $t\ge 0$. The largest excursion of $B^\lambda$ above zero is almost surely finite, and so the same holds for $Z^\lambda$. Thus, now turning to $(C^\lambda_1,C^\lambda_2,\ldots)$, the sequence of excursions lengths of $Z^\lambda$ in decreasing order,
$$\limsup_{T\rightarrow\infty} \Prob{(C^T_1,\ldots,C^T_k)=(C^\lambda_1,\ldots,C^\lambda_k)} = 1.$$
So we can lift \eqref{eq:kcptsconv} and conclude that
\begin{equation}\label{eq:kcptsconvunif}N^{-2/3}(C^N_1,\ldots,C^N_k) \quad \stackrel{d}\rightarrow \quad (C^\lambda_1,\ldots,C^\lambda_k),\end{equation}
as $N\rightarrow\infty$.

\subsubsection*{Convergence in $\ell^2$}

To lift \eqref{eq:kcptsconvunif} to convergence in $\ell^2$, 
it is enough to show that for every $\epsilon>0$
\begin{equation}\label{eq:l2tailbdF}\lim_{k\rightarrow\infty} \limsup_{N\rightarrow\infty} \Prob{\sum_{i>k} (C_i^N)^2 \ge \epsilon N^{4/3}}=0.\end{equation}

The corresponding result for $G(N,p)$, namely that
for every $\epsilon>0$, 
\begin{equation}\label{eq:l2tailbdG}\lim_{k\rightarrow\infty} \limsup_{N\rightarrow\infty} \Prob{\sum_{i>k} C_i(G(N,p))^2 \ge \epsilon N^{4/3}}=0,
\end{equation}
is implied by Proposition \ref{prop:AldousGNp}; 
the argument establishing this property for $G(N,p)$ 
is given in the Proof of Proposition 15 of 
\cite{Aldous97}. We can again use the coupling of $F(N,p)$ and $G(N,p)$ to derive \eqref{eq:l2tailbdF} from \eqref{eq:l2tailbdG}. We require the following lemma.

\begin{lemma}\label{l2taillemma}
Consider a family of finite non-increasing sequences
$$\mathbf{x}_1:=\left(x_{1,1},\ldots,x_{1,\ell_1}\right),\; \mathbf{x}_2:=\left(x_{2,1},\ldots,x_{2,\ell_2}\right),\;\ldots,\;\mathbf{x}_M:=\left(x_{M,1},\ldots,x_{M,\ell_M}\right)$$
with sums $x_1\ge x_2\ge \ldots\ge x_M$. Now let $y_1\ge y_2\ge \ldots \ge y_{\sum \ell_i}$ be the permutation of all the terms $x_{i,j}$ in descending order. Then, for any $k\ge 1$ and $\eta>0$,
\begin{equation}\label{eq:tailconcatbd}\sum_{i>k/\eta} y_i^2 \le \eta\sum_{i=1}^k x_i^2 + \sum_{i>k}x_i^2.\end{equation}
\begin{proof}[Proof of Lemma \ref{l2taillemma}]
Since each sequence $\mathbf{x}_i$ is non-increasing, set $m_i:= \max\left\{m \,:\, x_{i,m}\ge \eta x_i\right\}$ (with $\max\varnothing:=0$). So $m_i\le \frac{1}{\eta}$ for all $i$. It follows that for any $k\ge 1$,
$$\sum_{i=1}^k \sum_{j=1}^{m_i} x_{i,j}^2 \le \sum_{i=1}^{k/\eta} y_i^2.$$

Thus
$$\sum_{i>k/\eta} y_i^2 \le \sum_{i=1}^k \sum_{j\ge m_i} x_{i,j}^2 + \sum_{i>k}x_i^2.$$
But $\sum_{j\ge m_i}x_{i,j}^2 \le \sum_{j\ge m_i}x_{i,j}(\eta x_i)\le \eta x_i^2$, and so \eqref{eq:tailconcatbd} follows.
%overall we obtain
%$$\sum_{i>k/\eta} y_i^2 \le \eta\sum_{i=1}^k x_i^2 + \sum_{i>k}x_i^2.$$
\end{proof}
\end{lemma}

\begin{proof}[Proof of \eqref{eq:l2tailbdF}]
Under the coupling of Lemma \ref{barGstochdom}, each component of $G(N,p)$ is the disjoint union of components of $F(N,p)$. So we apply Lemma \ref{l2taillemma} with $x_1\ge x_2\ge\ldots$ as the component sizes of $G(N,p)$, and $y_1\ge y_2\ge \ldots$ as the components of $F(N,p)$, obtaining, for every $\eta>0$,
\begin{align*}
\Prob{N^{-4/3}\sum_{i\ge k/\eta} (C_i^N)^2\ge \epsilon} &\le \Prob{N^{-4/3}\sum_{i\ge 1} C_i(G(N,p))^2 \ge \tfrac{\epsilon}{2\eta}}\\
&\qquad+ \Prob{N^{-4/3}\sum_{i\ge k} C_i(G(N,p))^2 \ge \tfrac{\epsilon}{2}}.
\end{align*}
So, using \eqref{eq:l2tailbdG} to eliminate the second term on the RHS, we have
\begin{align*}
\limsup_{k\rightarrow\infty}\limsup_{N\rightarrow\infty}\Prob{N^{-4/3}\sum_{i\ge k/\eta} (C_i^N)^2\ge \epsilon}&\le \limsup_{N\rightarrow\infty} \Prob{N^{-4/3}\sum_{i\ge 1} C_i(G(N,p))^2 \ge \tfrac{\epsilon}{2\eta}}\\
&\le \frac{2\eta}{\epsilon}\Theta^\lambda ,
\end{align*}
applying Lemma \ref{JansonSpencerprop} and Markov's inequality in the final step. Taking $\eta\rightarrow 0$ completes the proof of \eqref{eq:l2tailbdF}, and of Theorem \ref{FNptheorem}.
\end{proof}

%\section{Appendix}\label{finalsection}
\section{Detailed combinatorial calculations}
\label{section:combinatorial}
\subsection{Proof of Lemma \ref{acyclicboundlemma}}\label{acyclicboundproof}
For convenience, we recall the statement of Lemma \ref{acyclicboundlemma}:
\begin{lemma*}Fix $\lambda^-<\lambda^+\in\R$. Given $p\in(0,1)$, let $\Lambda=\Lambda(N,p)=N^{1/3}(Np-1)$. Then
\begin{equation}\label{eq:weightedFs2-repeat}\Prob{G(N,p)\text{ acyclic}}=(1+o(1))g(\Lambda)e^{3/4}\sqrt{2\pi} N^{-1/6} ,\end{equation}
uniformly for $\Lambda\in[\lambda^-,\lambda^+]$ as $N\rightarrow\infty$.

\begin{proof}For this range of $p$, we will see that the sum in \eqref{eq:weightedFs1} is dominated by contributions on the scale $m=\frac{N}{2}+\frac{\Lambda N^{2/3}}{2}+\Theta(N^{1/2})$. Shortly we will be required to approximate these relevant contributions in detail, but first we show that contributions from outside this regime vanish as $N\rightarrow\infty$. We consider those $m$ for which
$$\left| m - \frac{N}{2} - \frac{\Lambda N^{2/3}}{2}\right| \ge N^{3/5}.$$

Let $B\sim \Bin{\binom{N}{2}}{p}$. Since $f(N,m)\le \binom{\binom{N}{2}}{m}$,
$$(1-p)^{\binom{N}{2}}\left[\sum_{m=0}^{\ce{N/2+\Lambda N^{2/3}/2-N^{3/5}}} f(N,m)\left(\tfrac{p}{1-p}\right)^m + \sum_{\fl{N/2+\Lambda N^{2/3}/2+N^{3/5}}}^{N-1} f(N,m)\left(\tfrac{p}{1-p}\right)^m\right]$$
\begin{align}
&\le \Prob{\big|B - \tbinom{N}{2}p\big| \ge N^{3/5}}\nonumber\\
&\le \frac{\mathrm{Var}\left(B\right)}{N^{6/5}}\le \frac{N^2 p}{2N^{6/5}}\le \frac{1+\lambda^+ N^{-1/3}}{2N^{1/5}}\ll N^{-1/6}.\label{eq:Chebbound}
\end{align}
Here we used Chebyshev's inequality, which is sufficient for our purposes, but note that the probability of this moderate deviation event for $B$ decays exponentially in some positive power of $N$.

\par
Given $\Lambda\in\R$ and $m\le N\in\N$, define $x=x(N,m,\Lambda)=\frac{\sqrt 2}{N^{1/2}}\left[m-\frac{N}{2}-\frac{\Lambda N^{2/3}}{2}\right]$. Then, we consider the set of $m$ satisfying
\begin{equation}\label{eq:setofm}\left | m-\frac{N}{2}-\frac{\Lambda N^{2/3}}{2}\right| \le N^{3/5},\quad\text{that is, } |x| \le \sqrt{2} N^{1/10}.\end{equation}

Thus
$$N-m = \tfrac{N}{2} - \tfrac{\Lambda}{2}N^{2/3} - \tfrac{x}{\sqrt{2}} N^{1/2},\quad\text{ and so }\frac{2(N-m)}{N} = 1-\Lambda N^{-1/3}-\sqrt{2}x N^{-1/2}.$$
From this, we obtain
\begin{align*}\log\left(\frac{2(N-m)}{N}\right) &= -\Lambda N^{-1/3}-\sqrt{2}xN^{-1/2} - \tfrac{\Lambda^2}{2} N^{-2/3} - \sqrt{2}\Lambda xN^{-5/6}\\
&\qquad - \tfrac{\Lambda^3}{3}N^{-1} - x^2 N^{-1} + O(N^{-16/15}),\\
\intertext{uniformly on the set of $m$ defined at \eqref{eq:setofm}. In calculating the scale of this final error term, we use that $|x|\le \sqrt{2}N^{1/10}$. Then}
(N-m)\log\left(\frac{2(N-m)}{N}\right)& = -\left[  \tfrac{\Lambda}{2}N^{2/3}+\tfrac{x}{\sqrt{2}} N^{1/2} + \tfrac{\Lambda^2}{4}N^{1/3} + \tfrac{\Lambda x}{\sqrt{2}}N^{1/6}+ \tfrac{\Lambda^3}{3}+\tfrac{x^2}{2} \right]\\
&\quad  + \left[\tfrac{\Lambda^2}{2}N^{1/3}+\tfrac{\Lambda x}{\sqrt{2}} N^{1/6} + \tfrac{\Lambda^3}{4}\right]\\
&\quad + \left[\tfrac{\Lambda x}{\sqrt{2}} N^{1/6} + x^2\right] + O\left(N^{-1/15}\right)\\
&=-\tfrac{\Lambda}{2}N^{2/3} - \tfrac{x}{\sqrt{2}} N^{1/2}+ \tfrac{\Lambda^2}{4}N^{1/3} + \tfrac{\Lambda x}{\sqrt{2}} N^{1/6}\\
&\qquad -\tfrac{\Lambda^3}{12} + \tfrac{x^2}{2} + O\left(N^{-1/15}\right).
\end{align*}

We now return to \eqref{eq:Britikovresult} and use Stirling's approximation and the expression we have just shown, as well as continuity of $g$. Uniformly on the set of $m$ in \eqref{eq:setofm}, (for which, recall, $N-m=(1+o(1))N/2$),
\begin{align}
f(N,m) &= (1+o(1))\frac{\sqrt{2\pi}N^{N-1/6}}{2^{N-m}(N-m)!} g\left(\frac{2m-N}{N^{2/3}}\right),\nonumber\\
& = \left(1+o(1)\right)\frac{g(\Lambda)\sqrt{2\pi}N^{N-1/6}}{2^{N-m}} \cdot \frac{1}{\sqrt{2\pi}\sqrt{N-m}} \left(\frac{e}{N-m}\right)^{N-m}\nonumber\\
& = \left(1+o(1)\right)g(\Lambda)\sqrt{2}N^{m-2/3}\exp(N-m) \exp\left(-(N-m)\log\left( \tfrac{2(N-m)}{N} \right) \right)\nonumber\\
& = \left(1+o(1)\right)g(\Lambda)\sqrt{2}N^{m-2/3}\exp\left(\tfrac{N}{2}-\tfrac{\Lambda}{2}N^{2/3} - \tfrac{x}{\sqrt{2}} N^{1/2} \right)\nonumber\\
&\qquad\times \exp\left(\tfrac{\Lambda}{2}N^{2/3} + \tfrac{x}{\sqrt{2}} N^{1/2} - \tfrac{\Lambda^2}{4}N^{1/3} - \tfrac{\Lambda x}{\sqrt{2}} N^{1/6} +\tfrac{\Lambda^3}{12} - \tfrac{x^2}{2}\right)\nonumber\\
& = \left(1+o(1)\right) g(\Lambda)\sqrt{2} N^{m-2/3} \exp\left( \tfrac{N}{2}- \tfrac{\Lambda^2}{4}N^{1/3} - \tfrac{\Lambda x}{\sqrt{2}}N^{1/6} + \tfrac{\Lambda^3}{12}-\tfrac{x^2}{2} \right).\label{eq:asympfNm}
\end{align}

Now, we have
\begin{align}
\binom{N}{2}\log (1-p) &= \binom{N}{2}\left[ - \tfrac{1+\Lambda N^{-1/3}}{N} - \tfrac{1}{2}N^{-2} + O(N^{-7/3})\right]\nonumber\\
& = -\tfrac{N}{2} - \tfrac{\Lambda}{2}N^{2/3} + \tfrac{1}{4} + O\left(N^{-1/3}\right),\label{eq:NC2log1-p}\\
\intertext{and also}
\log\left(\frac{Np}{1-p}\right)& = \log(1+\Lambda N^{-1/3}) - \log(1-p)\nonumber\\
&= \Lambda N^{-1/3} - \tfrac{\Lambda^2}{2} N^{-2/3} + \tfrac{\Lambda^3}{3}N^{-1} + N^{-1}+O\left(N^{-4/3}\right).\nonumber\\
\intertext{At this point, recall the definition}
m&= \tfrac{N}{2}+\tfrac{\Lambda}{2}N^{2/3}+\tfrac{x}{\sqrt{2}}N^{1/2}.\nonumber\\
\intertext{So, uniformly on the  set of $m$ for which $|x|\le \sqrt{2}N^{1/10}$, as before,}
m\log\left(\frac{Np}{1-p}\right)& = \left[ \tfrac{\Lambda}{2}N^{2/3} - \tfrac{\Lambda^2}{4}N^{1/3} + \tfrac{\Lambda^3}{6}+\tfrac{1}{2}\right]\nonumber\\
&\quad + \left[ \tfrac{\Lambda^2}{2} N^{1/3} - \tfrac{\Lambda^3}{4}\right] + \tfrac{\Lambda x}{\sqrt{2}} N^{1/6} +O\left(N^{-1/6}\right),\label{eq:mlogNp/1-p}
\end{align}
where each bracket corresponds to a term in the definition of $m$.

\par
Therefore, combining \eqref{eq:NC2log1-p} and \eqref{eq:mlogNp/1-p}, uniformly in the same sense,
\begin{equation}\label{eq:asympBinbinom}(1-p)^{\binom{N}{2}}\left(\tfrac{p}{1-p}\right)^m = (1+o(1))N^{-m}\exp\left(-\tfrac{N}{2}+\tfrac{\Lambda^2}{4}N^{1/3}+\tfrac{\Lambda x}{\sqrt 2}N^{1/6} - \tfrac{\Lambda^3}{12}+ \tfrac{3}{4}\right).\end{equation}
Combining \eqref{eq:asympfNm} and \eqref{eq:asympBinbinom}, we obtain
\begin{equation}\label{eq:tobesummedoverm}(1-p)^{\binom{N}{2}}\left(\tfrac{p}{1-p}\right)^m f(N,m) = (1+o(1)) g(\Lambda)\sqrt{2}N^{-2/3}\exp\left(-\tfrac{x^2}{2} + \tfrac34\right).\end{equation}
We now fix $N$ and $\Lambda$, and sum this quantity over the range of $m$ given by \eqref{eq:setofm}. Recall that $x$ is linear in $m$, with scaling factor $\frac{N^{1/2}}{\sqrt{2}}$, and so as $N\rightarrow\infty$, the sum of \eqref{eq:tobesummedoverm} over this range of $m$ converges after rescaling to a integral. That is,

\begin{align*}
&(1-p)^{\binom{N}{2}}\sum_{m=\fl{N/2+\Lambda N^{2/3}/2 -N^{3/5}}}^{\ce{N/2+\Lambda N^{2/3}/2 +N^{3/5}}} f(N,m)\left(\tfrac{p}{1-p}\right)^m\\
&\qquad=(1+o(1))e^{3/4 }g(\Lambda)\sqrt{2} N^{-2/3} \sum_{m=\fl{N/2+\Lambda N^{2/3}/2 -N^{3/5}}}^{\ce{N/2+\Lambda N^{2/3}/2 +N^{3/5}}} e^{-x^2/2}\\
&\qquad=(1+o(1))e^{3/4} g(\Lambda)\sqrt{2} \frac{N^{-1/6}}{\sqrt 2} \int_{-\infty}^\infty e^{-x^2/2}\mathrm{d}x,\\
&\qquad=(1+o(1))e^{3/4 }g(\Lambda) \sqrt{2\pi}N^{-1/6}.
\end{align*}
Combining with \eqref{eq:Chebbound}, which showed that contributions to the sum \eqref{eq:weightedFs1} outside this range of $m$ are $o(N^{-1/6})$, we obtain the required result.
\end{proof}
\end{lemma*}

\subsection{Proof of Lemma \ref{Anrklemma}}\label{Anrklemmaproof}
We recall the statement of Lemma \ref{Anrklemma}:
\begin{lemma*}Fix constants $\lambda^-,\lambda^+,\epsilon,K,T$ as in Definition \ref{defnPsi}. Then,
\begin{equation}\label{eq:asymPAn'rk-repeat}
\Prob{G(N',p)\in\mathcal A_{N',r,k}} = (1+o(1)) g(\Lambda-s-a)e^{3/4}N^{-5/6} ba^{-3/2}\end{equation}
$$\qquad\qquad \times \exp\left(-b(\Lambda-s)-\tfrac{b^2}{2a}+\tfrac{(\Lambda-s-a)^3-(\Lambda-s)^3}{6}\right),$$
uniformly on $(N',p,r,k)\in\Psi^N(\lambda^-,\lambda^+,\epsilon,K,T)$, as $N\rightarrow\infty$.

\begin{proof}
We will add the required uniformity in $N'$ at the end of this proof. First, we show
\begin{equation}\label{eq:asymPAnrk}\Prob{G(N,p)\in\mathcal A_{N,r,k}} = (1+o(1)) g(\lambda-a)e^{3/4}N^{-5/6} ba^{-3/2}\end{equation}
$$\qquad\qquad \times \exp\left(-b\Lambda-\tfrac{b^2}{2a}+\tfrac{(\Lambda-a)^3-\Lambda^3}{6}\right),$$
uniformly on $(p,r,k)$ such that $(N,p,r,k)\in\Psi^N(\lambda^-,\lambda^+,\epsilon,K,0)$, as $N\rightarrow\infty$.

\par
Subject to the constraint that vertices $1,\ldots,r$ are in different tree components, with sum equal to $k$, there are $\binom{N-r}{k-r}$ ways to choose which remaining vertices are part of this stack forest. Given this choice, we can view the trees as rooted at the vertices $[r]$. In particular, Cayley's formula states that there are $rk^{k-r-1}$ such labelled rooted forests. Hence
\begin{equation}\label{eq:PANrk-variant}\Prob{G(N,p)\in\mathcal A_{N,r,k}} = (1-p)^{\binom{N}{2}}\tbinom{N-r}{k-r}\left(\tfrac{p}{1-p}\right)^{k-r} rk^{k-r-1} \sum_{m=0}^{N-k-1}f(N-k,m) \left(\tfrac{p}{1-p}\right)^m.\end{equation}
By Lemma \ref{acyclicboundlemma}, uniformly on $(p,k)$ and for any $r$ such that $(N,p,r,k)\in\Psi^N(\lambda^-,\lambda^+,\epsilon,K,0)$ (in fact $r$ is arbitrary),
$$(1-p)^{\binom{N-k}{2}} \sum_{m=0}^{N-k-1} f(N-k,m) \left(\tfrac{p}{1-p}\right)^m= (1+o(1)) g\left(\Lambda(N-k,p)\right)e^{3/4} \sqrt{2\pi} N^{-1/6} .$$
Recall that this final sum is, up to a power of $(1-p)$, the probability that $G(N-k,p)$ is acyclic. We also have
\begin{align*}
\Lambda(N-k,p)&=(N-aN^{2/3})^{1/3} \left[ (N-aN^{2/3})p - 1\right]\\
& = \left(1+o(1)\right)N^{1/3}\left[ (Np-1) - aN^{-1/3}\right]\\
& = \left(1+o(1)\right) \left[ \Lambda(N,p) - a + o(1)\right].
\end{align*}
So, again uniformly on $(p,r,k)$ such that $(N,p,r,k)\in\Psi^N(\lambda^-,\lambda^+,\epsilon,K,0)$,
\begin{equation}\label{eq:PGN-kacyclic}(1-p)^{\binom{N-k}{2}} \sum_{m=0}^{N-k-1} f(N-k,m) \left(\tfrac{p}{1-p}\right)^m= (1+o(1)) g\left(\Lambda-a\right)e^{3/4} \sqrt{2\pi} N^{-1/6}.\end{equation}

We now carefully address the other terms in \eqref{eq:PANrk-variant}, starting with $(1-p)^{\binom{N}{2}-\binom{N-k}{2}} \left(\tfrac{p}{1-p}\right)^{k-r}$. Recall that $Np=1+\Lambda N^{-1/3}$. Firstly
\begin{align*}
\log\left[ \left(1+\Lambda N^{-1/3}\right)^{k-r}\right]&= \left[aN^{2/3}-bN^{1/3}\right]\left[\Lambda N^{-1/3}-\tfrac{\Lambda^2}{2} N^{-2/3} + O\left(N^{-1}\right)\right]\\
&=\Lambda a N^{1/3}-\Lambda b-\tfrac{\Lambda^2 a}{2} + O\left(N^{-1/3}\right).
\end{align*}

Also
\begin{align*}
\tbinom{N}{2}-\tbinom{N-k}{2}-k+r&=\tfrac{N^2}{2} - \tfrac{(N-k)^2}{2} + \tfrac{k}{2} - k+r\\
&= Nk - \tfrac{k^2}{2} + O\left(N^{2/3}\right)\\
&=aN^{5/3}-\tfrac{a^2}{2}N^{4/3} + O\left(N^{2/3}\right),
\end{align*}
from which
\begin{multline*}
\log\left[(1-p)^{\tbinom{N}{2}-\tbinom{N-k}{2}-k+r}\right]
\\
\begin{split}
&=\log\left[ \left(1-N^{-1}-\Lambda N^{-4/3}\right)^{\tbinom{N}{2}-\tbinom{N-k}{2}-k+r}\right]\\
&=\left[aN^{5/3}-\tfrac{a^2}{2}N^{4/3}+O\left(N^{2/3}\right)\right]\left[-N^{-1} - \Lambda N^{-4/3} + O\left(N^{-2}\right)\right]\\
&= -aN^{2/3}-\Lambda a N^{1/3} + \tfrac{a^2}{2}N^{1/3}+\tfrac{\Lambda a^2}{2} + O\left( N^{-1/3}\right).
\end{split}
\end{multline*}

From this,
\begin{align}
&(1-p)^{\binom{N}{2}-\binom{N-k}{2}} \left(\tfrac{p}{1-p}\right)^{k-r}\nonumber\\
&\quad= (1+o(1))N^{-(k-r)}\exp\left(-aN^{2/3} + \tfrac{a^2}{2}N^{1/3} - \Lambda b +\tfrac{\Lambda a}{2}(a-\Lambda) \right).\label{eq:2ndestimate}
\end{align}

Turning now to the binomial coefficent $\binom{N-r}{k-r}$ in \eqref{eq:PANrk-variant}, we treat each factorial separately. First observe that
\begin{align*}
\log\left[\left(1-bN^{-2/3}\right)^{N-bN^{1/3}}\right] &= \left[N-bN^{1/3}\right]\left[-bN^{-2/3} +O\left(N^{-4/3}\right)\right]\\
&= -b N^{1/3}+O(N^{-1/3})\\
\log\left[\left(1-aN^{-1/3}\right)^{N-aN^{2/3}}\right] &=  -aN^{2/3}+\tfrac{a^2}{2}N^{1/3}+\tfrac{a^3}{6}+O\left(N^{-1/3}\right)\\%\left[N-aN^{2/3}\right]\left[-aN^{-1/3}-\tfrac{a^2}{2}N^{-2/3} - \tfrac{a^3}{3}N^{-1}+O\left(N^{-4/3}\right) \right]\\
\log\left[\left(1-\tfrac{b}{a}N^{-1/3}\right)^{aN^{2/3}-bN^{1/3}}\right]&= -bN^{1/3}+\tfrac{b^2}{2a}+O\left(N^{-1/3}\right).%\left[aN^{2/3}-bN^{1/3}\right]\left[-\tfrac{b}{a} N^{-1/3}-\tfrac{b^2}{2a^2}N^{-2/3} +O\left(N^{-1}\right)\right]\\
\end{align*}

Then Stirling's approximation gives
\begin{align*}
\left(N-bN^{1/3}\right)! & = \left(1+o(1)\right) \frac{\sqrt{2\pi N}}{e^{N-bN^{1/3}}}\left(N-bN^{1/3}\right)^{N-bN^{1/3}}\\
&= \left(1+o(1)\right) \frac{\sqrt{2\pi N}}{e^{N-bN^{1/3}}} N^{N-bN^{1/3}}\exp\left( -bN^{1/3} \right)\\
\left(N-aN^{2/3}\right)! &= \left(1+o(1)\right) \frac{\sqrt{2\pi N}}{e^{N-aN^{2/3}}} N^{N-aN^{2/3}}\exp\left( -aN^{2/3}+\tfrac{a^2}{2}N^{1/3}+\tfrac{a^3}{6} \right)\\%\left(1+o(1)\right) \frac{\sqrt{2\pi N}}{e^{N-aN^{2/3}}} \left(N-aN^{2/3}\right)^{N-aN^{2/3}}\\
\left(aN^{2/3}-bN^{1/3}\right)!&= \left(1+o(1)\right) \frac{\sqrt{2\pi}\sqrt{aN^{2/3}}}{e^{aN^{2/3}-bN^{1/3}}}a^{aN^{2/3}-bN^{1/3}}N^{\frac{2}{3}\left[aN^{2/3}-bN^{1/3}\right]}\exp\left(-bN^{1/3}+\tfrac{b^2}{2a}\right).\\%\left(1+o(1)\right) \frac{\sqrt{2\pi}\sqrt{aN^{2/3}}}{e^{aN^{2/3}-bN^{1/3}}} \left(aN^{2/3}-bN^{1/3}\right)^{aN^{2/3}-bN^{1/3}}\\
\end{align*}

So we obtain
\begin{align}
\binom{N-r}{k-r}&=(1+o(1))  \tfrac{1}{\sqrt {2\pi}}a^{-(aN^{2/3}-bN^{1/3}+1/2)}N^{\frac13(aN^{2/3}-bN^{1/3}-1)} \label{eq:N-rk-restimate}\\
&\quad\times \exp\left(aN^{2/3}-\tfrac{a^2}{2}N^{1/3}-\tfrac{b^2}{2a}-\tfrac{a^3}{6}\right). \nonumber
\end{align}

The final ingredient of \eqref{eq:PANrk-variant} is the term
\begin{equation}\label{eq:rk^k-r-1}rk^{k-r-1} = b a^{aN^{2/3}-bN^{1/3}-1} N^{\frac23\left[aN^{2/3}-bN^{1/3}\right] - \frac13}.\end{equation}
To recover \eqref{eq:PANrk-variant}, we study the product of \eqref{eq:PGN-kacyclic}, \eqref{eq:2ndestimate}, \eqref{eq:N-rk-restimate} and \eqref{eq:rk^k-r-1}. Note that\\$\exp\left(-\frac{\Lambda^2a}{2}+\frac{\Lambda a^2}{2}-\frac{a^3}{6}\right)=\exp\left( \frac{(\Lambda-a)^3-\Lambda^3}{6} \right)$. So we can treat the terms in \eqref{eq:PANrk-variant} uniformly on $(p,r,k)$ such that $(N,p,r,k)\in\Psi^N(\lambda^-,\lambda^+,\epsilon,K,0)$, as $N\rightarrow\infty$ and obtain \eqref{eq:asymPAnrk} as required.

\par
We now finish the proof of \eqref{eq:asymPAn'rk-repeat}, where in addition we require a uniform estimate over $N'\in[N-TN^{2/3},N]$. We consider $(N',p,r,k)\in\Psi^N(\lambda^-,\lambda^+,\epsilon,K,T)$ as $N\rightarrow\infty$. Observe that
\begin{equation}\label{eq:defnb'etc1}\Lambda':= \Lambda(N',p)=(1+o(1)) \left(\Lambda(N,p)-s\right),\qquad N'=(1+o(1))N,\end{equation}
\begin{equation}\label{eq:defnb'etc2}b':= b(N',r) = (1+o(1))b(N,r),\quad a'=a(N',k)=(1+o(1))a(N,k).\end{equation}
Now fix $\delta\in(0,\epsilon)$. Then, for large enough $N$,
\begin{align}
(N',p,r,k)&\in\Psi^N(\lambda^-,\lambda^+,\epsilon,K,T) \nonumber\\
\label{eq:linkN'N} \Rightarrow\quad (N',p,r,k)&\in \Psi^{N'}(\lambda^- -T-\delta,\lambda^+ + \delta, \epsilon-\delta,K+\delta,0).\end{align}
Certainly $N-TN^{2/3}\rightarrow\infty$ as $N\rightarrow\infty$, so by \eqref{eq:asymPAnrk} and \eqref{eq:linkN'N},
$$\Prob{G(N',p)\in\mathcal A_{N,r,k}} = (1+o(1))g(\Lambda'-a')e^{3/4}N'^{-5/6} b'a'^{-3/2}$$
$$\qquad\qquad \times \exp\left(-b'\Lambda'-\tfrac{b'^2}{2a'}+\tfrac{(\Lambda'-a')^3-\Lambda'^3}{6}\right),$$
uniformly on $(N',p,r,k)\in \Psi^N(\lambda^-,\lambda^+,\epsilon,K,T)$ as $N\rightarrow\infty$. Finally, using \eqref{eq:defnb'etc1}, \eqref{eq:defnb'etc2}, and the fact that $g$ is uniformly continuous, we may conclude
$$\Prob{G(N',p)\in\mathcal A_{N',r,k}} = (1+o(1)) g(\Lambda-s-a)e^{3/4}N^{-5/6} ba^{-3/2}$$
$$\qquad\qquad \times \exp\left(-b(\Lambda-s)-\tfrac{b^2}{2a}+\tfrac{(\Lambda-a-s)^3-(\Lambda-s)^3}{6}\right),$$
as required, uniformly on $(N',p,r,k)\in\Psi^N(\lambda^-,\lambda^+,\epsilon,K,T)$. 
\end{proof}
\end{lemma*}

\subsection{Proof of Lemma \ref{geomratiolemma}}\label{geomratioproof}
We repeat the statement of Lemma \ref{geomratiolemma}:
\begin{lemma*}
Given the same constants as in Lemma \ref{Estacklemma}, there exist constants $M<\infty$ and $\gamma>0$ such that
\begin{equation}\label{eq:geomratiobd-repeat}
\frac{(k+1)\Prob{ G(N',p)\in \mathcal{A}_{N',r,k+1}}}{k\Prob{ G(N',p)\in \mathcal{A}_{N',r,k}}} \le 1- \gamma N^{-2/3},\end{equation}
for large enough $N$, whenever $(N',p,r)\in \bar \Psi_0^N(\lambda^-,\lambda^+,K,T)$ and $k\in[MN^{2/3},N'-1]$.
\begin{proof}

Again, we will use \eqref{eq:PANrk-variant}, which for convenience we recall here.
\begin{align*}
\Prob{G(N,p)\in\mathcal A_{N,r,k}} &= (1-p)^{\binom{N}{2}}\tbinom{N-r}{k-r}\left(\tfrac{p}{1-p}\right)^{k-r} rk^{k-r-1} \sum_{m=0}^{N-k-1}f(N-k,m) \left(\tfrac{p}{1-p}\right)^m\\
&=(1-p)^{\binom{N}{2} - \binom{N-k}{2}}\tbinom{N-r}{k-r}\left(\tfrac{p}{1-p}\right)^{k-r} rk^{k-r-1}F(N-k,p).
\end{align*}

We apply this to \eqref{eq:geomratiobd-repeat} (with $N$ replaced by $N'$). Note that $\binom{N'-r}{k+1-r}/\binom{N'-r}{k-r} = \frac{N'-k}{k+1-r}$, and $\binom{N'-k}{2}-\binom{N'-k-1}{2} = N'-k-1$. We obtain
\begin{align}
&\frac{(k+1)\Prob{ G(N',p)\in \mathcal{A}_{N',r,k+1}}}{k\Prob{ G(N',p)\in \mathcal{A}_{N',r,k}}} \nonumber\\
&\quad=\frac{(k+1)(1-p)^{-\binom{N'-k-1}{2}} \binom{N'-r}{k+1-r}\left(\frac{p}{1-p}\right)r (k+1)^{k-r}F(N'-k-1,p)}{k(1-p)^{-\binom{N'-k}{2}} \binom{N'-r}{k-r} r k^{k-r-1}F(N'-k,p)}\nonumber\\
&\quad =\frac{k+1}{k+1-r} \cdot (1-p)^{N'-k-2}\cdot \frac{N'-k}{N}\cdot\left(1+\Lambda N^{-1/3}\right)\cdot \left( \frac{k+1}{k}\right)^{k-r}\cdot \frac{F(N'-k-1,p)}{F(N'-k,p)}.\label{eq:keyratio}
\end{align}
We proceed in two parts. First we control the ratio of the $F(N'-k,p)$ terms using \eqref{eq:bdconsecutivefnm}. Then, we control the ratio of the remaining terms with an elementary but long Taylor expansion.

\par
First, note that from the second inequality in \eqref{eq:bdconsecutivefnm}, that for $k\le N'-1$,
$$1-\frac{F(N'-k,p)}{F(N'-k-1,p)}\le \frac12 (N'-k-1)p^2\E{|C^{N'-k-1,p}(v)|}.$$
where $|C^{N,p}(v)|$ is the size of the component containing a uniformly-chosen vertex $v$ in $G(N,p)$. Now, via \eqref{eq:defnb'etc1},
$$\limsup_{N\rightarrow\infty} \Lambda\left(N-\fl{MN^{2/3}},p\right)\le {\lambda^+}-M.$$
When $k\ge MN^{2/3}$, we have
$$N^{-1/3}\E{|C^{N'-k-1,p}(v)|}\le N^{-1/3}\E{|C^{N-\fl{MN^{2/3}},p}(v)|},$$
and so from \eqref{eq:barECvlimit},
$$\limsup_{N\rightarrow\infty}\sup_{\substack{N'\in[N-TN^{2/3},N]\\k\ge MN^{2/3}}}N^{-1/3}\E{|C^{N'-k-1,p}(v)|}\le \Theta^{\lambda^+-M}.$$
We obtain
$$\limsup_{N\rightarrow\infty} \sup_{\substack{N'\in[N-TN^{2/3},N]\\\lambda(N,p)\in[\lambda^-,\lambda^+]\\0\le k\le N'-1}} N^{2/3}\left[1- \frac{F(N'-k,p)}{F(N'-k-1,p)} \right] \le  \frac12 \Theta^{\lambda^+-M},$$
from which it follows that
\begin{equation}\label{eq:lowbdThetalambda}\limsup_{N\rightarrow\infty} \sup_{\substack{N'\in[N-TN^{2/3},N]\\\lambda(N,p)\in[\lambda^-,\lambda^+]\\0\le k\le N'-1}} N^{2/3}\left[ \frac{F(N'-k-1,p)}{F(N'-k,p)} - 1 \right] \le  \frac12 \Theta^{\lambda^+-M}.\end{equation}

\par
We now treat the remaining terms in the ratio \eqref{eq:keyratio}, that is
$$\frac{k+1}{k+1-r} \cdot (1-p)^{N'-k-2}\cdot \frac{N'-k}{N}\cdot\left(1+\lambda N^{-1/3}\right)\cdot \left( \frac{k+1}{k}\right)^{k-r}.$$

We split the calculation into several steps. Recall the rescalings $a=\frac{k}{N^{2/3}}$ and $b=\frac{r}{N^{1/3}}$. Since we assume $k\ge MN^{2/3}$, we have $\frac{1}{a}=O(1)$.
\begin{align*}
\log\left(\frac{k+1}{k+1-r}\right) &= -\log\left(1-\tfrac{r}{k+1}\right) = \tfrac{r}{k+1}+\tfrac12\left(\tfrac{r}{k+1}\right)^2 +O(N^{-1})\\
&= \tfrac{b}{a} N^{-1/3} + \tfrac{b^2}{2a^2} N^{-2/3}+O\left(N^{-1}\right),\\
\log\left(1+\Lambda N^{-1/3}\right) &= \Lambda N^{-1/3} - \tfrac{\Lambda^2}{2} N^{-2/3}+O(N^{-1}),\\
\log\left[\left( \frac{k+1}{k}\right)^{k-r}\right] &= \left[aN^{2/3}-bN^{1/3}\right]\left[\tfrac{1}{a}N^{-2/3} - \tfrac{1}{2a^2}N^{-4/3} + O\left(N^{-2}\right)\right]\\
&= 1 - \tfrac{b}{a}N^{-1/3} - \tfrac{1}{2a}N^{-2/3}+O\left(N^{-1}\right).
\intertext{The final two terms in the product require extra care, because there is no finite upper bound on $a$. However, since $a\le N^{1/3}$, we can still handle the error in the following term:}
\log \left[(1-p)^{N'-k-2}\right] &=\left[N-(s+a)N^{2/3}-2\right]\left[-N^{-1}-\Lambda N^{-4/3}+O\left(N^{-2}\right)\right]\\
&= -1 + (s-\Lambda+a) N^{-1/3} + \Lambda (a+s) N^{-2/3}+ O(N^{-1}).
\end{align*}
Finally, we have
$$\log\left(\frac{N'-k}{N}\right) = \log\left(1-sN^{-1/3}-aN^{-1/3}\right)\le -(a+s)N^{-1/3} - \tfrac12(a+s)^2 N^{-2/3}.$$
So there exists a constant $C=C(\lambda^-,\lambda^+,\epsilon,K,T)<\infty$ such that
\begin{align}
&\log\left[\frac{(k+1)(1-p)^{-\binom{N'-k-1}{2}} \binom{N'-r}{k+1-r}\left(\frac{p}{1-p}\right)r (k+1)^{k-r}}{k(1-p)^{-\binom{N'-k}{2}} \binom{N'-r}{k-r} r k^{k-r-1}}\right]\nonumber\\
\label{eq:giantlogcalc} &\quad\le N^{-2/3}\left[ -\frac12(\Lambda -(a+s))^2 + \frac{b^2}{2a^2} - \frac{1}{2a}\right] + \frac{C}{N},
\end{align}
uniformly on $(N',p,r)\in\Psi^N_0(\lambda^-,\lambda^+,\epsilon,K,T)$ and $k\ge MN^{2/3}$, as $N\rightarrow\infty$. Recall that $b\in [\epsilon,K]$, and that $k\ge MN^{2/3}$ is equivalent to $a\ge M$. So for large enough $M$, the term $\tfrac{b^2}{2a^2}$ is dominated by the term $-\frac{1}{2a}$ in \eqref{eq:giantlogcalc}. Then it holds that for large enough $N$,
$$\frac{(k+1)(1-p)^{-\binom{N'-k-1}{2}} \binom{N'-r}{k+1-r}\left(\frac{p}{1-p}\right)r (k+1)^{k-r}}{k(1-p)^{-\binom{N'-k}{2}} \binom{N'-r}{k-r} r k^{k-r-1}} \le 1-\frac{1}{3K} N^{-2/3}.$$
Using Lemma \ref{Thetalimit}, we now also demand that $M$ be large enough that $\Theta^{\lambda^+-M}\le\frac{1}{6K}$. So combining with \eqref{eq:lowbdThetalambda}, we can now approximate the LHS of \eqref{eq:geomratiobd-repeat} as required. Now take $\gamma\in(0,\frac{1}{6K})$, and we find that for large enough $N$
$$ \frac{(k+1)\Prob{ G(N,p)\in \mathcal{A}_{N',r,k+1}}}{k\Prob{ G(N,p)\in \mathcal{A}_{N',r,k}}} \le 1- \gamma N^{-2/3}.$$
\end{proof}
\end{lemma*}

\section{Regularity of $g$ and $\alpha$}\label{alpharegularity}

In this section, we prove various regularity properties of the function $g$ defined in \eqref{eq:defng}, and from this the technical properties we require about $\alpha$. In particular, the content of Lemma \ref{gregularitylemma} is a subset of what follows.

\subsection{Properties of $g$}
Recall the definition of $g$ from (\ref{eq:defng}):
\begin{equation*}
%\tag{\ref{eq:defng}}
g(x):=\frac{1}{\pi}\int_0^\infty \exp(-\tfrac{4}{3}t^{3/2})\cos(xt+\tfrac43 t^{3/2})dt.\end{equation*}
Britikov \cite{Britikov} observes that $g$ is, after- stretching by a factor $(2/3)^{2/3}$, the density of the canonical stable distribution with self-similarity exponent $\alpha=3/2$ and skewness $\beta=-1$. The following lemma, which restates the regularity properties of $g$ required for Lemma \ref{gregularitylemma}, follows from standard properties of such distributions, as stated, for example, by Zolotarev \cite{Zolotarevbook}.
\begin{lemma}\label{greglemma2}
The function $g$ defined in \eqref{eq:defng} is smooth and positive and has finite integral. Furthermore, it is bounded, uniformly continuous, and satisfies $g(x)\rightarrow 0$ as $x\rightarrow \pm\infty$. 
%\begin{proof}
%The key fact, which emerges from Britikov's proof \cite{Britikov} as introduced in Section \ref{forestsintro}, is that $g$ is, after- stretching by a factor $(2/3)^{2/3}$, the density of the canonical stable distribution with self-similarity exponent $\alpha=3/2$ and skewness $\beta=-1$. See for example Zolotarev's book \cite{Zolotarevbook} for a more general introduction to such distributions and their properties. In particular, $g$ is positive and smooth. Then $g$ is certainly bounded as
%$$|g(x)|\le \frac{1}{\pi}\int_0^\infty \exp\left(-\tfrac43 t^{3/2}\right) \mathrm{d}t<\infty.$$
%It is clear from the Mean Value Theorem that $|\cos(x)-\cos(y)|\le |x-y|$. Uniform continuity of $g$ then follows as
%$$|g(x)-g(y)| \le \frac{|x-y|}{\pi}\int_0^\infty t\exp\left(-\tfrac43 t^{3/2}\right) \mathrm{d}t,$$
%and this integral is finite. The claim that $g(x)\rightarrow 0$ as $x\rightarrow\pm\infty$ follows from uniform continuity, since $g$ is a density, and so has finite integral.
%\end{proof}
\end{lemma}

\subsubsection{$\alpha$ is well-defined}

For $k\in \N$, we define
\begin{equation}\label{eq:defnJk}J_k(b,\lambda) := \int_0^\infty a^{-k/2}g(\lambda-a)\exp\left( \tfrac{(\lambda-a)^3}{6} \right)\exp\left(-\tfrac{b^2}{2a}\right)\mathrm{d}a,\quad b>0, \lambda\in\R.\end{equation}

\begin{lemma}\label{alphawelldefined}
For each $k\in \N$, this function $J_k$ is well-defined and continuous, and has partial derivative with respect to $b$ given by
\begin{equation}\label{eq:partialJk}\frac{\partial}{\partial b} J_k(b,\lambda)= -b J_{k+2}(b,\lambda).\end{equation}
Furthermore, the function $\alpha(b,\lambda):= \frac{J_1(b,\lambda)}{J_3(b,\lambda)}$ defined in \eqref{eq:defnalpha} is also well-defined, continuous and differentiable with respect to $b$.
\begin{proof}
To show that $J_k(b,\lambda)<\infty$, we consider the integral in \eqref{eq:defnJk} separately over the ranges $a\in(0,1]$ and $a\in[1,\infty)$. We have
\begin{equation}\label{eq:Jkfinite1}\int_1^\infty a^{-k/2}g(\lambda-a)\exp\left(\tfrac{(\lambda-a)^3}{6}\right) \exp\left(-\tfrac{b^2}{2a}\right)\mathrm{d}a < e^{\lambda^3/6}\int_1^\infty g(\lambda-a)\mathrm{d}a<\infty,\end{equation}
and
\begin{equation}\label{eq:Jkfinite2}\int_0^1 a^{-k/2}g(\lambda-a)\exp\left(\tfrac{(\lambda-a)^3}{6}\right)\exp\left(-\tfrac{b^2}{2a}\right)\mathrm{d}a< e^{\lambda^3/6}g_{\max}\int_0^1 a^{-k/2}\exp\left(-\tfrac{b^2}{2a}\right)\mathrm{d}a<\infty.\end{equation}
Thus we have $J_k(b,\lambda)<\infty$.

\par
Since the bounds \eqref{eq:Jkfinite1} and \eqref{eq:Jkfinite2} hold locally uniformly in $(b,\lambda)$, continuity of $J_k$ follows from the dominated convergence theorem.

\par
We can check that we may differentiate \eqref{eq:defnJk} inside the integral to obtain
that \eqref{eq:partialJk} holds for all $k\ge 1$.
%\begin{equation}\label{eq:partialJk}\frac{\partial}{\partial b} J_k(b,\lambda)= -b J_{k+2}(b,
%\lambda).\end{equation}
Well-definedness and continuity of $\alpha(b,\lambda):= \frac{J_1(b,\lambda)}{J_3(b,\lambda)}$ follow immediately, since $J_3(b,\lambda)>0$ for all $b>0,\lambda\in\R$, and furthermore $\alpha(b,\lambda)$ is differentiable in its first argument as required, with
\begin{equation}\label{eq:dalphaviaJks}\frac{\partial}{\partial b} \alpha(b,\lambda) = \frac{bJ_1(b,\lambda)J_5(b,\lambda)}{J_3(b,\lambda)^2} - b,\end{equation}
through two applications of \eqref{eq:partialJk}.
\end{proof}
\end{lemma}

\subsection{Monotonicity of $\alpha$}

Heuristically, we can view \eqref{eq:defnalpha} as the expectation of $a$ with respect to the measure with density $a^{-3/2}g(\lambda-a)$, weighted by a factor $\exp(-\frac{b^2}{2a})$. Increasing $b$ reweights in favour of larger values of $a$, so $\alpha(b,\lambda)$ is increasing in $b$. We make this formal with the following straightforward lemma.

\begin{lemma}\label{incrdomlemma}Let $f,h$ be functions $\R_+\rightarrow \R_+$ such that $h$ is strictly increasing, and the integrals
$$\int_0^\infty a f(a) h(a)\mathrm{d}a,\quad \int_0^\infty f(a)\mathrm{d}a,$$
exist and are finite. Then
$$\frac{\int_0^\infty af(a)h(a)\mathrm{d}a}{\int_0^\infty f(a)h(a)\mathrm{d}a} > \frac{\int_0^\infty af(a)\mathrm{d}a}{\int_0^\infty f(a)\mathrm{d}a}.$$
\end{lemma}

\begin{corollary}$\alpha(b,\lambda)$ is increasing as a function of $b$.
\begin{proof}Fix $\lambda\in\R$ and $b'>b$, then set
$$f(a):= a^{-3/2} g(\lambda-a) \exp\left( \tfrac{(\lambda-a)^3}{6} \right)\exp\left(-\tfrac{b^2}{2a}\right),\text{ and } h(a):= \exp\left( - \tfrac{b'^2-b^2}{2a}\right),$$
in Lemma \ref{incrdomlemma}.
\end{proof}
\end{corollary}

\subsection{Lipschitz property of $\alpha$}
The following proposition establishes the behaviour of $\alpha(b,\lambda)$ as $b\downarrow 0$ in the sense required to complete the proof of Lemma \ref{Estacklemma}. It also establishes a Lipschitz condition for $\alpha$, required in Proposition \ref{Zlambdaexists} for the well-posedness of the reflected SDE \eqref{eq:ZSDE}.
\begin{prop}\label{alphaLipprop}Given $-\infty<\lambda^-<\lambda^+<\infty$, we have
\begin{equation}\label{eq:alphaasbtozero}\lim_{b\downarrow 0}\sup_{\lambda\in[\lambda^-,\lambda^+]} \alpha(b,\lambda)=0.\end{equation}
Furthermore, given $\rho<\infty$, there exists a constant $C<\infty$ such that $\alpha$ satisfies the Lipschitz condition
\begin{equation}\label{eq:alphalocLip} \left| \alpha(b,\lambda) - \alpha(b',\lambda) \right| \le C |b-b'|,\quad b,b'\in(0,\rho],\, \lambda\in[\lambda^-,\lambda^+].\end{equation}

\begin{proof}
To show \eqref{eq:alphalocLip}, it suffices to prove the following:
\begin{equation}\label{eq:bdalphapartial}\sup_{b\in(0,\rho], \lambda\in[\lambda^-,\lambda^+ ]}\left|\frac{\partial}{\partial b}\alpha(b,\lambda)\right|<\infty.\end{equation}

The steps we take to prove \eqref{eq:alphalocLip} will also allow us to read off \eqref{eq:alphaasbtozero}. Recall the expression \eqref{eq:dalphaviaJks} from the proof of Lemma \ref{alphawelldefined}:
\begin{equation}\tag{\ref{eq:dalphaviaJks}}\frac{\partial}{\partial b} \alpha(b,\lambda) = \frac{bJ_1(b,\lambda)J_5(b,\lambda)}{J_3(b,\lambda)^2} - b.\end{equation}
From Lemma \ref{alphawelldefined}, we know that $\frac{\partial}{\partial b} \alpha(b,\lambda)$ is continuous, and so to verify \eqref{eq:bdalphapartial}, it remains to consider the limit as $b\downarrow 0$. We examine the behaviour of each of $J_1(b,\lambda),J_3(b,\lambda),J_5(b,\lambda)$ in this limit.

\par
First, we consider $J_1$. We define
$$\gamma_1(\lambda):= e^{\lambda^3/6}\int_0^\infty a^{-1/2} g(\lambda-a) \mathrm{d}a,$$
which is seen to be finite by a similar decomposition to \eqref{eq:Jkfinite1} and \eqref{eq:Jkfinite2}. Then
\begin{align*}
\gamma_1(\lambda) - J_1(b,\lambda) &\le g_{\max}\int_0^\infty a^{-1/2} \left[e^{\lambda^3/6}-\exp\left(\tfrac{(\lambda-a)^3}{6}\right)\exp\left(-\tfrac{b^2}{2a}\right)\right] \mathrm{d}a\\
&\le e^{\lambda^3/6} g_{\max}\int_0^\infty a^{-1/2}\left[1-\exp\left(-\tfrac{b^2}{2a}\right) \right] \mathrm{d}a,
\end{align*}
and so by monotone convergence we have as $b\downarrow 0$,
\begin{equation}\label{eq:limJ1}\sup_{\lambda\in(-\infty,\lambda^+]}\left|J_1(b,\lambda) - \gamma_1(\lambda)\right| \rightarrow 0.\end{equation}

Substituting $u=\frac{b^2}{2a}$ into \eqref{eq:defnJk} gives
$$J_3(b,\lambda) = \frac{\sqrt{2}}{b} \int_0^\infty u^{-1/2}g\left(\lambda - \tfrac{b^2}{2u}\right)\exp\left(\tfrac{(\lambda-\frac{b^2}{2u})^3}{6}\right)\exp(-u)\mathrm{d}u.$$
So we define
$$\gamma_3(\lambda):=\sqrt{2}g(\lambda)e^{\lambda^3/6}\int_0^\infty u^{-1/2}\exp(-u)\mathrm{d}u,$$
and then by dominated convergence and uniform continuity of $g$,
\begin{equation}\label{eq:limJ3}\lim_{b\downarrow 0}\sup_{\lambda\in(-\infty,\lambda^+]} \left|b J_3(b,\lambda)-\gamma_3(\lambda)\right|=0.\end{equation}

A very similar argument can be deployed to obtain
$$\lim_{b\downarrow 0}\sup_{\lambda\in(-\infty,\lambda^+]} \left|b^3 J_5(b,\lambda)-\gamma_5(\lambda)\right|=0,$$
where
$$\gamma_5(\lambda):=2\sqrt{2}g(\lambda)e^{\lambda^3/6}\int_0^\infty u^{1/2}\exp(-u)\mathrm{d}u.$$

So we can return to \eqref{eq:dalphaviaJks}, which we rewrite as
$$\frac{\partial}{\partial b} \alpha(b,\lambda) = \frac{J_1(b,\lambda)\cdot b^3J_5(b,\lambda)}{(bJ_3(b,\lambda))^2} - b.$$
We now take the limit $b\downarrow 0$, for $\lambda\in[\lambda^-,\lambda^+]$. This denominator is uniformly bounded away from zero for $\lambda\in[\lambda^-,\lambda^+]$. So we obtain
\begin{equation}\label{eq:dpartialbtozero}\lim_{b\downarrow 0} \sup_{\lambda\in[\lambda^-,\lambda^+]} \left|\frac{\partial}{\partial b}\alpha(b,\lambda) - \frac{\gamma_1(\lambda)\gamma_5(\lambda)}{\gamma_3(\lambda)^2}\right| = 0.\end{equation}

Now, $\gamma_3,\gamma_5$ are clearly continuous, and $\gamma_1$ is also continuous by the same argument as given for continuity of $J_k$ in the proof of Lemma \ref{alphawelldefined}. Furthermore, $\gamma_3$ is positive, and so we have
$$\max_{\lambda\in[\lambda^-,\lambda^+]}\gamma_1(\lambda)<\infty,\quad\min_{\lambda\in[\lambda^-,\lambda^+]} \gamma_3(\lambda)>0.$$
Taken with \eqref{eq:limJ3}, the latter shows that
$$\lim_{b\downarrow 0}\inf_{\lambda\in[\lambda^-,\lambda^+]}J_3(b,\lambda)=\infty.$$
Therefore, since $\alpha(b,\lambda)=\frac{J_1(b,\lambda)}{J_3(b,\lambda)}$, using \eqref{eq:limJ1} as well, we obtain precisely the first required statement \eqref{eq:alphaasbtozero}.

\par
For similar reasons, we have
\begin{equation}\label{eq:gammasbdd}\max_{\lambda\in[\lambda^-,\lambda^+]}\frac{\gamma_1(\lambda)\gamma_5(\lambda)}{\gamma_3(\lambda)^2}<\infty.\end{equation}

Since $\frac{\partial}{\partial b} \alpha(b,\lambda)$ is continuous on $(0,\rho]\times[\lambda^-,\lambda^+]$, from \eqref{eq:dpartialbtozero} and \eqref{eq:gammasbdd}, it's clear that
$$\sup_{b\in(0,\rho], \lambda\in[\lambda^-,\lambda^+ ]}\left|\frac{\partial}{\partial b}\alpha(b,\lambda)\right|<\infty,$$
from which \eqref{eq:alphalocLip} follows. This completes the proof of Proposition \ref{alphaLipprop}.
\end{proof}
\end{prop}

\subsection{Existence of $Z^\lambda$}\label{SDEsection}
First we prove Proposition \ref{Zlambdaexists}, which asserts that $Z^\lambda$ is well-defined. The short proof considers a limit of \emph{localised} reflected SDEs, whose existence is given by the following theorem, which assumes a \emph{global} Lipschitz and boundedness condition on the coefficients of the reflected SDE. 
\begin{theorem}\cite[\textsection IX 2.14]{RevuzYorbook}\label{ReflSDEthm}
Let $\sigma(s,x)$ and $b(s,x)$ be functions $\R_+\times \R_+\rightarrow \R$, and $W$ a Brownian motion. For $z_0\ge 0$, we call a solution to the SDE with reflection $e_{z_0}(\sigma,b)$ a pair $(Z,K)$ of processes such that
\begin{enumerate}
\item the process $Z$ is continuous, positive, $\mathcal{F}^W$-adapted, and
\begin{equation}\label{eq:generalSDE}Z(t)=z_0 + \int_0^t \sigma(s,Z(s))\mathrm{d}W(s) + \int_0^t b(s,Z(s))\mathrm{d}s + K(t),\end{equation}
\item the process $K$ is continuous, non-decreasing, vanishing at zero, $\mathcal{F}^W$-adapted, and
\begin{equation}\label{eq:generalrefl}\int_0^\infty Z(s)\mathrm{d}K(s)=0.\end{equation}
\end{enumerate}
If $\sigma$ and $b$ are bounded and satisfy the global Lipschitz condition
\begin{equation}\label{eq:Lipschsigmab}|\sigma(s,x)-\sigma(s,y)| + |b(s,x)-b(s,y)|\le C|x-y|,\end{equation}
for every $s,x,y\in(0,\infty)$ and some constant $C$, then there exists a solution to $e_{z_0}(\sigma,b)$, and furthermore this solution is unique.
\end{theorem}

\subsubsection{Proof of Proposition \ref{Zlambdaexists}}
We now return to the existence of $Z^\lambda$ as in \eqref{eq:ZSDE}, for fixed $\lambda\in\R$. In this setting $\sigma(s,x)\equiv 1$, but
\begin{equation}\label{eq:defnbsx}b(s,x):= \lambda - s -\alpha(x,\lambda - s),\end{equation}
is neither bounded below nor satisfies the global Lipschitz property. However, by Proposition \ref{alphaLipprop}, for any $R>0$, we can define $b^R(s,x)$ such that $b^R(s,x)$ is bounded and globally Lipschitz in $x$; and $b^R(s,x)=b(s,x)$ whenever $(s,x)\in[0,R]\times[0,R]$. Then Theorem \ref{ReflSDEthm} asserts that there is a unique pair of processes $(Z^{\lambda,R},K^{\lambda,R})$ corresponding to this drift, where $Z^{\lambda,R}(0)=0$.

\par
Let $\tau^{\lambda,R}$ be the time at which $Z^{\lambda,R}$ first hits $R$. Take $R'\ge R$. Then, it is clear that $Z^{\lambda,R}$ is equal to $Z^{\lambda,R'}$ up to time $R\wedge\tau^{\lambda,R}$ almost surely. Also, since $b(s,x)$ is bounded above by $\lambda$, it follows that $\tau^{\lambda,R}\rightarrow \infty$ as $R\rightarrow\infty$ almost surely. Therefore, we may define
$$Z^\lambda(t)= \lim_{R\rightarrow\infty} Z^{\lambda,R}(t),$$
for almost all paths of $W$, and $Z^\lambda$. It is immediate that $Z^\lambda$ satisfies \eqref{eq:ZSDE}. Furthermore, any solution $(Z^{\lambda},K^{\lambda})$ to \eqref{eq:ZSDE} must coincide with $(Z^{\lambda,R},K^{\lambda,R})$ up to $\tau^{\lambda,R}$, and so uniqueness of $(Z^\lambda,K^\lambda)$ follows as well, as required for Proposition \ref{Zlambdaexists}.

\subsection{Convergence of non-negative Markov processes}\label{SVsection}
It remains to show that Theorem \ref{explconvtheorem} follows from Proposition \ref{limdriftprop} as claimed.

\par
A general framework for showing convergence of Markov processes to the solutions of SDEs was introduced by Stroock and Varadhan in the 60s (see, for example, \cite{SVbook}). The convergence of Markov processes to reflected diffusions is treated in \cite{SVDiffBCs} in high generality, allowing for general boundaries in $\R^d$, and inhomogeneous stickiness at the boundaries.

\par
We assume that a sequence of Markov chains $Z^N,\,N\in\N$ is given, where $Z^N$ has discrete state space $\mathcal{S}^N\subseteq \R_{\ge 0}$, with $0\in \mathcal{S}^N$, and initial condition $Z^N(0)=0$. We define the time-inhomogeneous transition operator $\pi^N$ as 
$$\pi^N_n(x,y)=\Prob{Z^N(n+1)=y \,\big|\, Z^N(n)=x},
\quad n\in\N,\; x,y\in\mathcal{S}^N.$$

We consider a time-rescaling $(h(N))_{N\in\N}$ for which $h(N)\rightarrow 0$ as $N\rightarrow\infty$.

\begin{remark}In our specific example, we have $\mathcal{S}^N=N^{-1/3}\Z_{\ge 0}$, and $h(N)=N^{-2/3}$.\end{remark}

Then, for every $x\in\mathcal{S}^N$, we define the following rescaling transition quantities corresponding to drift, diffusivity, and macroscopic jump probabilities, respectively,
$$b^N(t,x) = \frac{1}{h(N)}\sum_{y\in\mathcal{S}^N} (y-x) \pi^N_{\fl{t/h(N)}} (x,y),$$
$$a^N(t,x) = \frac{1}{h(N)}\sum_{y\in\mathcal{S}^N} (y-x)^2 \pi^N_{\fl{t/h(N)}}(x,y),$$
$$\Delta_\epsilon = \frac{1}{h(N)} \sum_{\substack{y\in \mathcal{S}^N\\ |y-x|>\epsilon}}\pi^N_{\fl{t/h(N)}} (x,y).$$

The following theorem, which is a special case of Theorem 6.3 from \cite{SVDiffBCs}, gives conditions under which time-rescaled versions of $Z^N$ converge to SDEs with reflection. 

\begin{theorem}\label{altSVreflthm}
Suppose we have that for any $T,M>0$, and any $\epsilon>0$,
$$\lim_{N\rightarrow\infty} \sup_{t\in[0,T]}\sup_{\substack{x\in \mathcal{S}^N\\x\le M}}\Delta_\epsilon^N(t,x)=0,\qquad \liminf_{N\rightarrow\infty} \inf_{t\in[0,T]}a^N(t,0)>0$$
$$\lim_{N\rightarrow\infty}\sup_{t\in[0,T]}\sup_{\substack{x\in \mathcal{S}^N\\0<x\le M}}\left|a^N(t,x)-1\right|=0,\qquad \lim_{N\rightarrow\infty}\sup_{t\in[0,T]}\sup_{\substack{x\in \mathcal{S}^N\\0<x\le M}}\left|b^N(t,x)-b(t,x)\right|=0,$$
and that furthermore $b(\cdot,\cdot)$ satisfies the global Lipschitz condition \eqref{eq:Lipschsigmab} of the previous theorem. Then
$$Z^N\left( \fl{\tfrac{t}{h(N)}}\right)_{t\ge 0} \;\Rightarrow\; (Z(t))_{t\ge 0},$$
as $N\rightarrow\infty$ with respect to the topology of uniform convergence on $\mathbb{D}[0,T]$ for each $T<\infty$, where $Z$ is the unique solution to $e_0(1,b)$, as given by Theorem \ref{ReflSDEthm}.

\end{theorem}

\subsubsection{Proof of Theorem \ref{explconvtheorem}}
Now let $Z^{N,p}$ be the exploration process of $F(N,p)$, satisfying the conditions of Theorem \ref{FNptheorem}. Again, in our setting, we must account for the fact that the drift of $Z^\lambda$ is neither bounded nor globally Lipschitz.

\par
Recall from \eqref{eq:defnbsx} and the following paragraph the definitions of $b(s,x)$  and $b^R(s,x)$. For any $R\in \N$, we can construct a Markov process $(Z^{N,p,R}_n,n\ge 0)$ whose transition probabilities coincide with those of $Z^{N,p}$ whenever $n\in [0,TN^{2/3}]$ and $Z^{N,p,R}_n \le RN^{1/3}$, and for which, by Proposition \ref{limdriftprop},
$$N^{1/3} \E{Z^{N,p,R}_{tN^{2/3}+1} - Z^{N,p,R}_{tN^{2/3}} \, \big|\, Z^{N,p,R}_{tN^{2/3}}=xN^{1/3}} \rightarrow b^R(t,x),$$
uniformly for $t\in[0,T]$ and $x$ in any compact interval in $(0,\infty)$. We define the rescaled process $\tilde Z^{N,\lambda,R}$ from $Z^{N,p,R}$ analogously to \eqref{eq:defntildeZ}. Then we have $\tilde Z^{N,p,R}\stackrel{d}\rightarrow Z^{\lambda,R}$ uniformly on $[0,T]$.

\par
From this,
$$\Prob{\sup_{n\in[0,TN^{2/3}]} Z^{N,p,R}_n> R N^{1/3}}\rightarrow 0,$$
as $R\rightarrow\infty$, and so as processes on $[0,T]$, the law of $\tilde Z^{N,p,R}$ converges to the law of $\tilde Z^{N,p}$ as $R\rightarrow\infty$, and the law of $Z^{\lambda,R}$ converges to the law of $Z^{\lambda}$. Thus we have proved Theorem \ref{explconvtheorem}.

Combining with the results of Sections 
\ref{Zncvgsection} and \ref{cptsizes}, 
the proof of our main result 
Theorem \ref{FNptheorem} is now also complete.

\section{Lifting from $F(N,p)$ to $F(N,m)$}\label{mtopsection}

%The main body of this study is concerned with $F(N,p)$, since for the exploration process related to this model is easier to work with. Although the exploration processes of both models have the Markov property {\bf true?? check}, we shall see that that in $F(N,p)$ the increments are described by reweighted binomial distributions, which reduces the complexity of the combinatorial calculations somewhat.
%
%\par
%That said, we consider $F(N,m)$ to be the more natural model of random forests. We can show that our desired results for $F(N,m)$ follow from corresponding results for $F(N,p)$. This argument is essentially independent from the main result for $F(N,p)$ and so we present it first. {\bf This bit needs rewriting or more likely moving.}
So far we have worked in the context of the model 
$F(N,p)$ (since in that case the transition probabilities
in the exploration process are rather more 
straightforward to work with than in the case of 
the model $F(N,m)$). 
In this section we show that Theorem \ref{FNptheorem}
for $F(N,p)$ implies Theorem \ref{FNmtheorem}
for $F(N,m)$. 

As discussed in Section \ref{subsubsec:intro-monotonicity},
if we had natural monotonicity properties for the
families $F(N,p)$ and $F(N,m)$, then it would be straightforward to deduce
Theorem \ref{FNmtheorem} from Theorem 
\ref{FNptheorem} by a sandwiching argument. 
Instead, we will construct an 
``almost monotonic" coupling. The idea of
Lemma \ref{lem:almostmonotone} below is
that, within the scaling window, if
the difference between $m^-$ and $m^+$
is small compared to $N^{2/3}$ as $N\to\infty$, then we can 
couple $F(N,m^-)$ and $F(N,m^+)$ so that
with high probability, 
the former is contained in the latter.
This coupling is achieved, informally speaking,
by adding edges one by one uniformly at random,
unless doing so would create a cycle. 
The next lemma will provide an upper bound on 
the probability that
a cycle does in fact appear.

%\begin{prop}[Local limit theorem for the 
%number of edges]
%\label{prop:locallimit}
%Let $p=\frac{1}{N}+O(N^{-4/3})$ as $N\to\infty$. 
%Let $M_{N,p}$ be the number of edges of a 
%random forest distributed according to $F(N,p)$. 
%
%Then uniformly on $m\in(N^2p/2-N^{3/5}, N^2p/2+N^{3/5})$,
%\[
%\Prob{M_{N,p}=m}=
%(1+o(1))\frac{1}{\sqrt\pi}N^{-1/2}\exp\left(-\frac{x^2}{2}\right) \text{ as } N\to\infty,
%\]
%where 
%$x=\sqrt{2}N^{-1/2}
%\left(
%m-N^2 p/2
%\right)$.
%\end{prop}
%
%\begin{proof}
%Divide \eqref{eq:tobesummedoverm} by \eqref{eq:weightedFs2}. 
%\end{proof}

\begin{lemma}\label{lem:sequential}
Let $H$ be a forest on $[N]$,
and let $S^2=S^2(H)$ be the sum of the squares of the
component sizes of $H$. Let $k$ edges, chosen independently and uniformly at random from $[N]\times[N]$, 
be added to $H$. (For convenience we allow self-edges and
repeated edges). The probability that the 
resulting graph contains a cycle (including a self-edge
or a repeated edge) is at most 
$\displaystyle\frac{2kS^2/N^2}{1-2kS^2/N^2}$.
In particular if $k=o(N^2/S^2)$,
then the graph is a forest with high probability as 
$N\to\infty$.
\end{lemma}

\begin{proof}
Let the components of $H$ be $C_1,\dots, C_a$
with sizes $x_1, \dots, x_a$. 

To create a cycle, for some $r\geq 1$, and some
distinct $b_1, b_2, \dots, b_r$, we have to add 
an edge between $C_{b_{i}}$ and $C_{b_{i+1}}$ 
for each $1\leq i\leq r-1$, and an edge between $C_{b_r}$ and $C_{b_1}$. This creates a cycle containing $r$ new edges 
(and also perhaps some further edges which were already 
part of $H$). 

The probability that a given edge has endpoints in $C_b$ and $C_{b'}$ is $2x_{b}x_{b'}/N^2$ if $b\ne b'$, 
and $x_b^2/N^2$ if $b=b'$,
so by a union bound,
the probability that at least one of the $k$ new edges created has endpoints in $C_b$ and $C_{b'}$ is at most $2kx_{b}x_{b'}/N^2$. 
For fixed $r$ and $b_1, \dots, b_r$,
a simple conditional probability argument then gives
a bound on the probability of creating a collection 
of edges as specified, of 
$(2k/N^2)^r (x_{b_1}x_{b_2})\dots(x_{b_{r-1}}x_{b_r})(x_{b_r}x_{b_1})$,
which is $(2k/N^2)^r x_{b_1}^2\dots x_{b_r}^2$.

Summing over $r$ and over distinct $b_1, \dots, b_r$,
we obtain that the probability of creating a cycle
is at most 
$\sum_{r=1}^\infty \left(\frac{2k}{N^2}S^2\right)^r$,
which gives the claimed bound. 
\end{proof}

We don't know whether $F(N,m+1)$ stochastically
dominates $F(N,m)$ in general; that is, whether
there is a coupling such that 
$F(N,m)\subset F(N,m+1)$ with probability $1$.
We get round this by introducing a method to create a
coupling which is ``monotone with high probability". 

Let $H_m\sim F(N,m)$, and consider generating $\tH_{m+1}$
by adding an edge chosen uniformly at random 
(from $[N]\times[N]$) to $H_m$. Let $\cA$ be the event
that $\tH_{m+1}$ is a forest. 

We claim that conditional on 
$\cA$, the distribution of $\tH_{m+1}$ is $F(N,m+1)$. 
For 
\begin{align*}
\Prob{\tH_{m+1}=H'}&=\sum_{e\in E(H')}\frac{1}{N^2}
\Prob{H_m=H'\setminus\{e\}}\\
&=\sum_{e\in E(H')}\frac{1}{N^2}\frac{1}{f(N,m)}\\
&=\frac{m+1}{N^2 f(N,m)},
\end{align*}
which is indeed constant over $H'$. 

Define also $\kH_{m+1}$ to be distributed
according to $F(N,m+1)$, independently from $H_m$
and the added edge. Now define 
\[
H_{m+1}=\begin{cases} 
\tH_{m+1} &\text{on }\cA\\
\kH_{m+1} &\text{on }\cA^c
\end{cases}.
\]
Then indeed $H_{m+1}\sim F(N,m+1)$, and 
$\Prob{H_m\subset H_{m+1}}\geq \Prob{\cA}$.

We may extend this; starting from $H_m$, 
sequentially add $k$ edges independently and 
uniformly, to give graphs $\tH_{m+1}$, $\tH_{m+2},
\dots, \tH_{m+k}$. 

Let $\cA_j$ be the event that adding the first $j$ edges
does not create a cycle (including a self-edge or repeated edge). 

Let $\kH_{m+1},\dots,\kH_{m+k}$ be independent samples
from $F(N,m+1),\dots, F(N, m+k)$ respectively, and independent
of $H_{m}, \tH_{m+1}, \dots, \tH_{m+k}$.

Now define
\[
H_{m+j}=
\begin{cases} 
\tH_{m+j} &\text{on }\cA_j\\
\kH_{m+j} &\text{on }\cA_j^c
\end{cases}.
\]
Then $H_{m+j}\sim F(N,m+j)$ for $j=0,1,\dots, k$,
and $\Prob{H_m\subset H_{m+1}\subset\dots\subset H_{m+k}}
\geq \Prob{\cA_k}$.

\begin{lemma}\label{lem:almostmonotone}
Let $m=N/2+O(N^{2/3})$ as $N\to\infty$.

Define $p^-$ and $p^+$ by 
$N^2p^-/2=\lfloor m-N^{3/5}\rfloor$ and
$N^2p^+/2=\lceil m+N^{3/5}\rceil$.

Then there is a coupling of $F^-\sim F(N,p^-)$,
$F\sim F(N,m)$, and $F^+\sim F(N, p^+)$
such that with high probability as $N\to\infty$,
$F^-\subseteq F\subseteq F^+$.
\end{lemma}

\begin{proof}
Let $M^-$ have the distribution of the 
number of edges of $F(N, p^-)$, and independently
let $M^+$ have the distribution of the number of edges 
of $F(N, p^+)$. 

Note that if, conditional on $M^-$ and $M^+$, 
$F^-\sim F(N,M^-)$ and $F^+ \sim F(N, M^+)$ 
then the unconditional distributions of $F^-$ and $F^+$
are $F(N, p^-)$ and $F(N, p^+)$, respectively. 

\eqref{eq:weightedFs2} and \eqref{eq:Chebbound} tell 
us that in this regime, the probability that 
the number of edges of the graph $F(N,p)$
deviates from $N^2p/2$ by $N^{3/5}$ or more 
goes to 0 as $N\to\infty$. 

So with high probability as $N\to\infty$, we have 
$M^-\in(m-2N^{3/5}, m)$ and $M^+\in (m, m+2N^{3/5})$. 
If either of these fails, we give up trying to do
anything smart and simply set $F^-\sim F(N, M^-)$, 
$F\sim F(N,m)$ and $F^+\sim F(N, M^+)$ independently. 

Otherwise, we have $M^-<m<M^+$,
and we use the above idea of adding edges sequentially.
Throughout the construction below we condition on $M^-$ and
$M^+$ and regard them as fixed. 

Let $H_{M^-}\sim F(N, M^-)$, and sequentially add
$K=M^+-M^-\leq 4N^{3/5}$ edges independently and uniformly,
to give graphs 
$\tH_{M^-+1}$, $\tH_{M^-+2},
\dots, \tH_{M^+}$. 
As before, let $\cA_j$ be the event that adding the first $j$ edges
does not create a cycle. 

Let $\kH_{M^-+1},\dots,\kH_{M^+}$ be independent samples
from $F(N,M^-+1),\dots, F(N, M^+)$ respectively, and independent of $H_{M^-}, \tH_{M^-+1}, \dots, \tH_{M^+}$.

Now for $1\leq j\leq K$, define
\[
H_{M^-+j}=
\begin{cases} 
\tH_{M^-+j} &\text{on }\cA_j\\
\kH_{M^-+j} &\text{on }\cA_j^c
\end{cases}.
\]
Then $H_{M^-+j}\sim F(N,M^-+j)$ for $j=0,1,\dots, K$,
and $H_{M^-}\subset H_{M^-+1}\subset\dots\subset H_{M^+}$
whenever $\cA_K$ occurs. 

In particular define $F^-=H_{M^-}$, $F=H_m$ 
and $F^+=H_{M^+}$. Then (unconditionally),
$F^-$, $F$ and $F^+$ have the desired marginal 
distributions, and will be ordered as desired 
whenever the event $\cA_K$ occurs. So to complete the proof it suffices
to show that $\cA_K$ occurs with high probability
as $N\to\infty$. 

Let $S^2(H_{M^-})$ be the sum of squares of the component 
sizes of $H_{M^-}$. We know that, averaging over
$M^-$, the distribution of $H_{M^-}$ is that
of $F(N, p^-)$. This is stochastically dominated
by $G(N, p^-)$, and so Corollary \ref{S2corollary} tells us that $\E{S^2(H_{M^-})}\le \E{S^2(G(N,p^-)}=O(N^{4/3})$. In particular,
$S^2(H_{M^-})\leq N^{4/3+\epsilon}$ with high probability
as $N\to\infty$, for any $\epsilon>0$. 

But the number of edges $K$ that we add in the 
sequential construction is at most $4N^{3/5}$. 
So Lemma \ref{lem:sequential} tells us 
that if indeed $S^2(H_{M^-})\leq N^{4/3+\epsilon}$, then 
(if $\epsilon$ is taken sufficiently small) with
high probability no cycle is created by adding 
$K$ edges to $H_{M^-}$. Hence the event $\cA_K$ occurs 
with high probability as desired. 
\end{proof}

Finally, we can deduce our main
scaling limit result for the model $F(N,m)$.

\begin{proof}[Proof of Theorem \ref{FNmtheorem}]
If $m$ has the given asymptotics, and 
$p^-$ and $p^+$ are defined in terms of $m$ as in Lemma 
\ref{lem:almostmonotone}, then 
$p^-=1/N+(\lambda+o(1))N^{-4/3}$, and the same
is true for $p^+$. 

From Theorem \ref{FNptheorem}, this means that the rescaled component sizes 
of both $F(N, p^-)$ and $F(N, p^+)$ have 
the limit in distribution given by $\C^\lambda$.

But from Lemma \ref{lem:almostmonotone}, 
if the component sizes of $F(N, p^-)$ and $F(N, p^+)$
both have this distributional limit, then
the same must be true of $F(N,m)$, and we are done.
\end{proof}

\section*{Acknowledgments}
We are grateful to Christina Goldschmidt for many valuable discussions during
the course of this work. We thank Oliver Riordan for many insightful comments, especially concerning a simplication of Lemma \ref{fNvsfN+1lemma}, Bal{\'a}zs R{\'a}th for a valuable conversation
about the form of the diffusion $Z^\lambda$ at an early stage of the project, and Tom Kurtz for helpful advice about 
the methods of Section \ref{SVsection}.
We thank the referee for pointing out an oversight in the
proof, and for several further helpful comments.
The second author was supported by EPSRC doctoral
training grant EP/K503113, ISF grant 1325/14, and in part by the Joan and Reginald Coleman--Cohen Fund, and the work was also
supported by EPSRC grant EP/J019496/1.
\bibliographystyle{abbrv}
\bibliography{CRF}

\end{document}